\documentclass[11pt]{article}
\usepackage{t1enc}
\usepackage{amsmath,amssymb,color}

\newtheorem{Def}{Definition}
\newtheorem{Prop}{Proposition}
\newtheorem{Assumption}{Assumption}
\newtheorem{Lemma}{Lemma}

\newtheorem{Remark}{Remark}

\newenvironment{proof}[1][Proof]{\textbf{#1.} }{\hfill $\square$}

\def \1{\mathbf{1}}
\newcommand{\essinf}{\operatorname{essinf}}
\newcommand{\esssup}{\operatorname{esssup}}
\def \eps{\varepsilon}

\def \bN{\mathbb{N}}
\def \bR{\mathbb{R}}
\def \bQ{\mathbb{Q}}
\def \bE{\mathbb{E}}
\def \bF{\mathbb{F}}
\def \bH{\mathbb{H}}
\def \bD{\mathbb{D}}
\def \bS{\mathbb{S}}
\def \bP{\mathbb{P}}
\def \bL{\mathbb{L}}
\def \bI{\mathbb{I}}
\def \bM{\mathbb{M}}
\def \bG{\mathbb{G}}

\def \cF{\mathcal{F}}
\def \cY{\mathcal{Y}}
\def \cX{\mathcal{X}}
\def \cE{\mathcal{E}}

\def \cP{\mathcal{P}}
\def \cS{\mathcal{S}}
\def \cA{\mathcal{A}}

\def \cT{\mathcal{T}}

\def \cG{\mathcal{G}}
\def \bsp{\bar{\mathfrak{p}}}

\def \sq{\mathfrak{q}}
\def \kq{\vartheta}

\def\be{\begin{eqnarray}}
\def\ee{\end{eqnarray}}

\def\b*{\begin{eqnarray*}}
\def\e*{\end{eqnarray*}}

\newcommand{\No}[1]{\left\|#1\right\|}     

\def \E{\mathbb{E}}
\def \F{\mathbb{F}}
\def \H{\mathbb{H}}
\def \L{\mathbb{L}}
\def \M{\mathbb{M}}
\def \N{\mathbb{N}}
\def \P{\mathbb{P}}
\def \Q{\mathbb{Q}}
\def \R{\mathbb{R}}

\def\Fc{{\cal F}}

\def\Mc{{\cal M}}

\def\Pc{{\cal P}}

\def\Yc{{\cal Y}}

\def\Fb{{\bar F}}

\def\Pb{{\overline \P}}

\def\Ych{{\widehat \Yc}}

\def\esup{{\rm ess \, sup}}

\def\x{\times}

\def\={\;=\;}
\def\.{\;.}

\def\eps{\varepsilon}

\def\1{{\bf 1}}

\def\Om{\Omega}
\def\om{\omega}

\def\crochetX{\langle X \rangle}
\def\ah{\widehat{a}}

\def\Bf{\mathfrak{B}}
\def\Sp{\mathbb S_d^{\geq 0}}
\def\ox{\otimes}
\def \w{\mathsf{w}}

\def\Omb{\overline{\Om}}
\def\omb{\bar \om}
\def\Fb{\overline{\F}}
\def\Fcb{\overline{\Fc}}

\def\0{\mathbf{0}}

\def\normeL2#1{\left\|{#1}\right\|_{L^2}}

\begin{document}

\title{Second order BSDE under monotonicity condition \\ and liquidation problem under uncertainty\footnote{This work was started while the first author was visiting the National University of Singapore, whose hospitality is kindly acknowledged.}}
\author{A. Popier \thanks{Laboratoire Manceau de Math\'ematiques, Le Mans Universit\'e, Avenue Olivier Messiaen, 72085 Le Mans cedex 9, France.}\hspace{-0.0cm} 
\thanks{Partially supported by F\'ed\'eration de Math\'ematiques des Pays de Loire, FR 2962 du CNRS.
\hfill\break
e-mail: {\tt alexandre.popier@univ-lemans.fr}} ,
Chao Zhou \thanks{Department of Mathematics, National University of Singapore, 10 Lower Kent Ridge Road, 119076, Singapore. 
} \hspace{-0.1cm} \thanks{Supported by Singapore MOE (Ministry of Education's) AcRF grant R--146--000--219--112.\hfill\break e-mail: {\tt matzc@nus.edu.sg}}
}
\date{\today}

\maketitle

\begin{abstract}
In this work we investigate an optimal closure problem under Knightian uncertainty. We obtain the value function and an optimal control as the minimal (super-)solution of a second order BSDE with monotone generator and with a singular terminal condition.
\end{abstract}

\vspace{0.5cm}


\section{Introduction}

This paper is devoted to the study of an optimal liquidation problem under uncertainty. Roughly speaking, for some $\vartheta > 1$, we want to minimize the functional cost
\begin{equation}\label{eq:control_pb_intro}
J(\cX) = \sup_{\bP\in \cP} \bE^{\bP} \left[  \int_0^T \left( \eta_s |\alpha_s|^\vartheta + \gamma_s |\cX_s|^\vartheta \right) ds + \xi  |\cX_T|^\vartheta  \right]
\end{equation}
over all progressively measurable processes $X$ that satisfy the dynamics
\begin{equation*}
\cX_s =x +\int_0^s \alpha_u du. 
\end{equation*}
The non-negative quantity $\xi$ is a penalty on the remaining value $\cX_T$ of the state process $\cX$. In particular when $\xi = +\infty$, $J(\cX)$ is finite only if the terminal constraint $\cX_T \1_{\xi = +\infty} = 0$ is satisfied. If the set of probability measures $\cP$ is a singleton, then the problem is solved in \cite{anki:jean:krus:13} and \cite{krus:popi:15} using a backward stochastic differential equation (BSDE in short) with singular terminal condition. Our goal is to extend these results to the case where there is model uncertainty, that is when the probability measure $\bP$ is not unique. Minimizing \eqref{eq:control_pb_intro} corresponds for an agent to compute the worst case scenario for the liquidation of her portfolio. 

The analysis of optimal control problems with state constraint on the terminal value is motivated by models of optimal portfolio liquidation under stochastic price impact (see, e.g. \cite{almg:12}, \cite{almg:chri:01}, \cite{fors:kenn:12}, \cite{gath:schi:11}, \cite{hors:nauj:14}, or \cite{krat:scho:13}, among many others). For a fixed probability $\bP$ (that is without the supremum in \eqref{eq:control_pb_intro}):
\begin{equation}\label{eq:class_control_pb_intro}
J(\cX,\bP) =  \bE^{\bP} \left[  \int_0^T \left( \eta_s |\alpha_s|^\vartheta + \gamma_s |\cX_s|^\vartheta \right) ds + \xi  |\cX_T|^\vartheta  \right]
\end{equation}
this position targeting problem \eqref{eq:class_control_pb_intro} and some variants have been studied in \cite{anki:jean:krus:13}, \cite{krus:popi:15}, \cite{grae:hors:qiu:13}, \cite{grae:hors:sere:13} or \cite{schie:13}. In this framework the state process $\cX$ denotes the agent's position in the financial market. At each point in time $t$ she can trade in the primary venue at a rate $\alpha_t$ which generates costs $\eta_t |\alpha_t|^\vartheta$ incurred by the stochastic price impact parameter $\eta_t$. The term $\gamma_t |\cX_t|^\vartheta$ can be understood as a measure of risk associated to the open position. $J(\cX,\bP)$ thus represents the overall expected costs for closing an initial position $x$ over the time period $[0,T]$ using strategy $\cX$, with a terminal cost $\xi |\cX_T|^\vartheta$. The penalization $\xi$ is $\F_T$-measurable and takes value in $[0,\infty]$. The total cost $J(\cX,\bP)$ is finite if and only if $\cX_T \1_{\xi = +\infty} = 0$ a.s. The case $\xi = +\infty$ a.s. corresponds to the liquidation constraint: $X_T =0$ a.s., that is the position has to be closed imperatively. The optimal strategies and the value function of this control problem \eqref{eq:class_control_pb_intro} are characterized in \cite{anki:jean:krus:13} and \cite{krus:popi:15} (see also \cite{grae:hors:qiu:13} for the use of BSPDEs) by the minimal supersolution $(y,m)$ of the BSDE:
\begin{equation} \label{eq:bsde_contr_prob}
dy_t  =  \frac{y_t^{\sq}}{(\sq-1)\eta_t^{\sq-1}} dt - \gamma_t dt  + dm_t
\end{equation}
with $\displaystyle \liminf_{t\to T} y_t \geq \xi$. Here $\sq > 1$ is the H\"older conjugate of $\vartheta$ and $m$ is a martingale. Since $\xi$ can be equal to $+\infty$, such a BSDE is called \textit{singular}. In \cite{anki:jean:krus:13} and \cite{krus:popi:15} sufficient conditions on the coefficient processes $\eta$ and $\gamma$ are provided such that there exists a minimal supersolution to \eqref{eq:bsde_contr_prob} and then by a verification theorem based on a penalization argument, it is proved that $\inf_{\cX} J(\cX,\bP) = y_0$ (see the details in Section \ref{sect:standard_form_liquid_pb}). 

\vspace{0.5cm}
When $\bP$ is not unique, we need to solve 
$$J(\cX) = \sup_{\bP \in \cP} J(\cX,\bP) = \sup_{\bP\in \cP} y_0^\bP$$
where $y^\bP$ is the minimal supersolution of \eqref{eq:bsde_contr_prob} under the probability measure $\bP$. From the theory of second order BSDE (2BSDE in short) introduced by \cite{sone:touz:zhan:12} and \cite{sone:touz:zhan:13}, our problem can be solved with this useful tool. Nevertheless the generator $f$ of our BSDE \eqref{eq:bsde_contr_prob} is not Lipschitz continuous but only monotone w.r.t. $y$:
$$f(t, y) = - \frac{y^{\sq}}{(\sq-1)\eta_t^{\sq-1}} + \gamma_t.$$  
This condition has been already considered in \cite{poss:13}, but under the additional assumption that the generator is of linear growth. The possibility to extend the existing results to a general monotone driver is mentioned in the paper \cite{poss:tan:zhou:15} (see in particular Section 2.4.5). Thereby following the ideas of \cite{poss:tan:zhou:15}, we want to show that a 2BSDE with monotone generator w.r.t. $y$ still has a unique solution. 

\vspace{1cm}
Let us precise the main contributions of the paper. Roughly speaking, the paper \cite{poss:tan:zhou:15} shows that if nice properties are known for BSDEs and reflected BSDEs, then it is possible to construct a solution for the 2BSDE. Hence to follow the scheme of \cite{poss:tan:zhou:15}, several properties of classical (reflected or not) BSDEs are needed. Several general results can be found in \cite{krus:popi:14} and \cite{krus:popi:17} for BSDEs in a general filtration and \cite{lepe:mato:xu:05}, \cite{klim:15} or \cite{bouc:poss:zhou:15} for reflected BSDEs (see also the references therein). But some technical results were missing in the general setting of the 2BSDEs: {\it Lipschitz approximation of BSDEs} (Lemma \ref{lem:L2_approx}) or {\it existence and uniqueness of the solution for a reflected BSDEs in a general filtration} (without quasi-left continuity) {\it when the driver is monotone} (Proposition \ref{prop:ref_BSDE_monotone_driver}). These results, even though useful for applications to 2BSDE, are expectable and the techniques employed are rather standard. This is the reason why they are postponed in the Appendix. 

The first part is devoted to 2BSDEs with a monotone driver (condition {\bf (H)}). The probabilistic setting is the same as \cite{poss:tan:zhou:15}. But to overcome this difficulty induced by monotonicity, we will impose some stronger integrability conditions\footnote{Sufficient to solve our control problem. Weaker integrability assumptions are left for future research.} {\bf C1} and {\bf C2} on the terminal value $\xi$ (and on the process $f^0$). Our main results are Proposition \ref{prop:uniq_2BSDE} (uniqueness) and Proposition \ref{prop:2BSDE_existence} (existence). Although the sketch is almost the same as in \cite{poss:tan:zhou:15}, several technical issues have to be taken into account in our setting. Moreover the monotonicity of the driver forces us to change the minimality condition on the non-decreasing process $K^\bP$. Instead of the classical assumption
$$\underset{\bP' \in \cP(t,\bP,\bF_+)}{\essinf^\bP} \bE^{\bP'} \left[K^{\bP'}_T - K^{\bP'}_t\bigg|  \cF_t^+ \right] = 0, \quad 0 \leq t \leq T, \ \bP-\mbox{a.s.}, \ \forall \bP \in \cP$$
we have here:
$$\underset{\bP' \in \cP(t,\bP,\bF_+)}{\essinf^\bP} \bE^{\bP'} \left[ \int_t^T \exp \left( \int_{t}^s \lambda_u^{\bP'} du \right)dK^{\bP'}_s \bigg|  \cF_t^+ \right] = 0, \quad 0 \leq t \leq T, \ \bP-\mbox{a.s.}, \ \forall \bP \in \cP$$
where $\lambda_s^{\bP'}$ is the increment of the generator evaluated at the solution $Y$ of the 2BSDE and at the solution  $y^{\bP'}$ of the classical BSDE under $\bP'$. 
In the Lipschitz setting, $\lambda^{\bP'}$ is bounded and thus can be removed, whereas under the monotone assumption, it is only bounded from above (see Definition \ref{def:sol_2BSDE}, Conditions \eqref{eq:minim_cond_K_2BSDE} and \eqref{eq:class_minim_cond_K} and the discussion in Section \ref{ssect:discussion}). 

Then we come back to our initial goal: the resolution of the optimal control problem \eqref{eq:control_pb_intro}. From our results on 2BSDEs, we can now obtain directly the value function and an optimal control for the unconstrained problem (Proposition \ref{prop:optim_unconst_pb}). For the constrained problem, a known difficulty concerns the filtration. Indeed to avoid the possibility of a uncontrolled jump for the orthogonal martingale part at the terminal time $T$, some additional hypothesis on the filtration is needed (see \cite{popi:16}, Section 2.2). Under this technical condition on the filtration, we prove that the 2BSDE with singular terminal condition  has a minimal super-solution (Proposition \ref{prop:sing_2BSDE} and Remark \ref{rem:sing_gen_2BSDE}) and that we can solve \eqref{eq:control_pb_intro} using this super-solution (Proposition \ref{prop:constr_pb}).

\vspace{1cm}
The paper is decomposed as follows. In Section \ref{sect:2BSDE} we use the scheme developed in \cite{poss:tan:zhou:15} to obtain existence and uniqueness for the 2BSDE. In Section \ref{sect:liquidation_pb} we solve the control problem \eqref{eq:control_pb_intro} using 2BSDE with monotone driver and singular terminal condition. In the Appendix, we recall and develop some results concerning BSDE and reflected BSDE with monotone driver.

\section{Second order BSDE with monotone generator} \label{sect:2BSDE}

This section is devoted to the extension of the results in \cite{poss:tan:zhou:15} to the case where the generator $f$ is only monotone w.r.t. $y$ (see {\bf H2}). Compared with \cite{poss:13}, we do not assume that $f$ is of linear growth w.r.t. $y$ (see {\bf H4}). Nevertheless let us immediately emphasize that we will assume stronger integrability assumptions of $\xi$ and $f^0$ to overcome this difficulty and we change the minimality condition on the non-decreasing process $K$ (see Condition \eqref{eq:minim_cond_K_2BSDE}). 

\subsection{The probabilistic setting} \label{sect:setting}

Our framework is exactly the same as in \cite{poss:tan:zhou:15}. Let us recall the notations and assumptions. Let $\mathbb{N}^*:=\mathbb{N}\setminus\{0\}$ and let $\mathbb{R}_+^*$ be the set of real positive numbers. For every $d-$dimensional vector $b$ with $d\in \mathbb{N}^*$, we denote by $b^{1},\ldots,b^{d}$ its coordinates and for $\alpha,\beta \in \R^d$ we denote by $\alpha\cdot \beta$ the usual inner product, with associated norm $\No{\cdot}$, which we simplify to $|\cdot|$ when $d$ is equal to $1$. We also let ${\bf 1}_d$ be the vector whose coordinates are all equal to $1$. For any $(l,c)\in\mathbb N^*\times\mathbb N^*$, $\mathcal M_{l,c}(\mathbb R)$ will denote the space of $l\times c$ matrices with real entries. Elements of the matrix $M\in\mathcal M_{l,c}$ will be denoted by $(M^{i,j})_{1\leq i\leq l,\ 1\leq j\leq c}$, and the transpose of $M$ will be denoted by $M^\top$. When $l=c$, we let $\mathcal M_{l}(\mathbb R):=\mathcal M_{l,l}(\mathbb R)$. We also identify $\mathcal M_{l,1}(\R)$ and $\R^l$. Let $\Sp$ denote the set of all symmetric positive semi-definite $d\x d$ matrices. We fix a map $\psi : \Sp \longrightarrow \Mc_d(\R)$ which is (Borel) measurable and satisfies $\psi(a) (\psi(a))^\top = a$ for all $a \in \Sp$, and denote $a^{\frac 12} := \psi(a)$.

\subsubsection{Canonical space}

Let $d \in \N^{\ast}$, we denote by $\Om := C\left( \left[ 0,T \right], \mathbb{R}^d \right)$ the canonical space of all $\R^d-$valued continuous paths $\om$ on $[0,T]$ such that $\om_0 = 0$, equipped with the canonical process $X$, {\it i.e.} $X_t(\om) := \om_t,$ for all $\om  \in \Om$. Denote by $\F = (\Fc_t)_{0 \le t \le T}$ the canonical filtration generated by $X$, and by $ \mathbb{F}_+ = ( \mathcal{F}^+_t)_{ 0 \le t \le T }$ the right limit of  $ \mathbb{F} $ with $ \Fc_t^+ \mathrel{\mathop:}= \cap_{ s> t} \Fc_s $ for all $t \in [0, T)$ and $\Fc^+_T := \Fc_T$. We equip $\Om$ with the uniform convergence norm $\No{\om}_{\infty} := \sup_{0 \le t \le T} \No{\om_t}$, so that the Borel $\sigma-$field of $\Om$ coincides with $\Fc_T$. Let $\P_0$ denote the Wiener measure on $\Om$ under which $X$ is a Brownian motion.

\vspace{1mm}
	
Let $\M_1$ denote the collection of all probability measures on $(\Om, \Fc_T)$. Notice that $\M_1$ is a Polish space equipped with the weak convergence topology. We denote by $\Bf$ its Borel $\sigma-$field. Then for any $\P\in \M_1$, denote by $\Fc_t^{\P}$ the completed $\sigma-$field of $\Fc_t$ under $\P$. Denote also the completed filtration by $\F^{\P} = \big(\Fc_t^{\P} \big)_{t \in [0,T]}$ and $\F^{\P}_+$ the right limit of $\F^{\P}$, so that $\F^{\P}_+$ satisfies the usual conditions. Moreover,  for $\Pc \subset \M_1$, we introduce the universally completed filtration $\F^{U} := \big(\Fc^{U}_t \big)_{0 \le t \le T}$, $\F^{\Pc} := \big(\Fc^{\Pc}_t \big)_{0 \le t \le T}$, and $\F^{\Pc+} := \big(\Fc^{\Pc+}_{t} \big)_{0 \le t \le T}$, defined as follows
$$\mathcal F^{U}_t := \bigcap_{\P \in \M_1} ~\Fc_t^{\P},\  \Fc^{\Pc}_t := \bigcap_{\P \in \Pc} ~\Fc_t^{\P},\ t\in[0,T],\  \Fc^{\Pc+}_t := \Fc^{\Pc}_{t+},\ t\in[0,T), \ \text{and}\ \Fc^{\Pc+}_T:=\Fc^{\Pc}_T.$$
We also introduce an enlarged canonical space $\Omb := \Om \x \Om'$, where $\Om'$ is identical to $\Om$. By abuse of notation, we denote by $(X, B)$ its canonical process, {\it i.e.} $X_t(\omb) := \om_t$, $B_t(\omb) := \omega'_t$ for all $\omb := (\om, \omega') \in \Omb$, by $\Fb = (\Fcb_t)_{0 \le t \le T}$ the canonical filtration generated by $(X,B)$, and by $\Fb^X = (\Fcb^X_t)_{0 \le t \le T}$ the filtration generated by $X$. Similarly, we denote the corresponding right-continuous filtrations by $\Fb^X_+$ and $\Fb_+$, and the augmented filtration by $\Fb^{X, \Pb}_+$ and $\Fb^{\Pb}_+$, given a probability measure $\Pb$ on $\Omb$.

\subsubsection{Semi-martingale measures}

We say that a probability measure $\P$ on $(\Om, \Fc_T)$ is a semi-martingale measure if $X$ is a semi-martingale under $\P$. Then on the canonical space $\Om$, there is some $\F-$progressively measurable non-decreasing process (see e.g. Karandikar \cite{kara:95}), denoted by $\crochetX = (\crochetX_t)_{0 \le t \le T}$, which coincides with the quadratic variation of $X$ under each semi-martingale measure $\P$. Denote further
\begin{equation*} 
\ah_t ~:=~ \limsup_{\eps \searrow 0}~ \frac{\crochetX_t - \crochetX_{t-\eps}}{\eps}.
\end{equation*}
	
\vspace{0.5em}
For every $t \in [0, T]$, let $\Pc^W_t$ denote the collection of all probability measures $\P$ on $(\Om, \Fc_T)$ such that
\begin{itemize}
\item $(X_s)_{s \in [t,T]}$ is a $(\P, \F)-$semi-martingale admitting the canonical decomposition (see e.g. \cite[Theorem I.4.18]{jaco:shir:03})
$$X_s=\int_t^sb^\P_rdr + X^{c,\P}_s, \ s\in [t, T], \ \P-a.s., $$ 
where $b^\P$ is a $\F^{\P}-$predictable $\R^d-$valued process, and $X^{c,\P}$ is the continuous local martingale part of $X$ under $\P$.

\item $\big(\crochetX_s \big)_{s \in [t,T]}$ is absolutely continuous in $s$ with respect to the Lebesgue measure, and $\ah$ takes values in $\Sp$, $\P-a.s.$
\end{itemize}
Given a random variable or process $\lambda$ defined on $\Om$, we can naturally define its extension on $\Omb$ (which, abusing notations slightly, we still denote by $\lambda$) by
\begin{equation} \label{eq:extension_def}
\lambda(\omb) := \lambda(\om), ~~\forall \omb = (\om, \om') \in \Omb.
\end{equation}
In particular, the process $\ah$ can be extended on $\Omb$. Given a probability measure $\P \in \Pc^W_t$, we define a probability measure $\Pb$ on the enlarged canonical space $\Omb$ by $\Pb := \P \otimes \P_0$, so that $X$ in $(\Omb, \Fcb_T, \Pb, \Fb)$ is a semi-martingale with  the same triplet of characteristics as $X$ in $(\Om, \Fc_T, \P, \F)$, $B$ is a $\Fb-$Brownian motion, and $X$ is independent of $B$.
Then for every $\P \in \Pc^W_t$, there is some $\R^d-$valued, $\Fb-$Brownian motion $W^{\P} = (W^{\P}_r)_{t \le r \le s}$ such that  (see e.g. Theorem 4.5.2 of \cite{stro:vara:79})
\begin{equation} \label{eq:XW}
X_s = \int_t^s b_r^{\P} dr + \int_t^s \ah^{\frac 12}_r dW_r^{\P}, ~~s \in [t, T], ~~\Pb-a.s.,
\end{equation}
where we extend the definition of $b^{\P}$ and $\ah$ on $\Omb$ as in \eqref{eq:extension_def}, and where we recall that $\widehat a^{\frac 1 2}$ has been defined in the notations above.

\vspace{0.5em}
Notice that when $\ah_r$ is non-degenerate $\P-a.s.$, for all $r \in [t, T]$, then we can construct the Brownian motion $W^{\P}$ on $\Om$ by 
$$W_t^\P:=\int_0^t\widehat a_s^{-\frac 1 2}dX_s^{c,\P},\ t\in[0,T],\ \P-a.s.,$$
and do not need to consider the above enlarged space equipped with an independent Brownian motion to construct $W^{\P}$.
	 
\begin{Remark}[On the choice of $\ah^{\frac{1}{2}}$]
The measurable map $\psi: a \longmapsto a^{\frac 12}$ is fixed throughout the paper. A first choice is to take $a^{\frac 12}$ as the unique non-negative symmetric square root of $a$ $($see e.g. Lemma 5.2.1 of \cite{stro:vara:79}$)$. One can also use the Cholesky decomposition to obtain $a^{\frac 12}$ as a lower triangular matrix. Finally the reader can read \cite{poss:tan:zhou:15}, Remark 2.2 where the sets $\cP(t,\omega)$ are given by the collections of probability measures induced by a family of controlled diffusion processes. In this case  one can take $\ah^{\frac 12}$ in the following way:
\begin{equation} \label{eq:def_a12}
a ~= \left( \begin{array}{cc}
\sigma \sigma^T & \sigma \\
\sigma^T & I_n 
\end{array} \right)
~~\mbox{and}~~
a^{\frac12} ~= \left( \begin{array}{cc}
\sigma & 0 \\
I_n & 0
\end{array} \right),
~~~\mbox{for some}~ \sigma \in \Mc_{m,n}.
\end{equation}
\end{Remark} 
	
\subsubsection{Conditioning and concatenation of probability measures}
	
We also recall that for every probability measure $\P$ on $\Om$ and $\F-$stopping time $\tau$ taking value in $[0,T]$, there exists a family of regular conditional probability distribution (r.c.p.d. for short) $(\P^{\tau}_{\om})_{\om \in \Om}$ (see e.g. Stroock and Varadhan \cite{stro:vara:79}), satisfying:

\begin{itemize}
\item[(i)] For every $\om \in \Om$, $\mathbb{P}^{\tau}_{\om}$ is a probability measure on $(\Om,  \Fc_T)$.

\vspace{0.5em}
\item[(ii)] For every $ E \in \mathcal{F}_T $, the mapping $ \om \longmapsto \mathbb{P}^{\tau}_{\om}(E) $ is $\mathcal{F}_\tau-$measurable.

\vspace{0.5em}
\item[(iii)] The family $ (\mathbb{P}^{\tau}_{\om})_{\om \in \Om} $ is a version of the conditional probability measure of $ \mathbb{P} $ on $ \mathcal{F}_{\tau}$, {\it i.e.}, for every integrable $ \mathcal{F}_T-$measurable random variable $ \xi $ we have $ \mathbb{E}^{\mathbb{P} } [ \xi | \mathcal{F}_{\tau}](\omega)=\mathbb{E}^{ \mathbb{P}^{\tau}_{\omega}} \big[\xi \big],$ for $\P-a.e.$ $\om\in\Om$.

\vspace{0.5em}
\item[(iv)] For every $\om \in \Om$, $ \mathbb{P}^{\tau}_{\om} (\Omega_{\tau}^{\om})=1$, where $\Omega_{\tau}^{\om}\mathrel{\mathop:}= \big \lbrace \overline\omega \in \Omega : \ \overline\omega(s)=\om(s), \ 0\le s \le \tau(\om) \big \rbrace.$
\end{itemize}

\vspace{0.5em}
Furthermore, given some $\P$ and a family $(\Q_{\om})_{\om \in \Om}$ such that $\om \longmapsto \Q_{\om}$ is $\Fc_{\tau}-$measurable and $\Q_{\om} (\Om_{\tau}^{\om}) = 1$ for all $\om \in \Om$, one can then define a concatenated probability measure $\P \otimes_{\tau} \Q_{\cdot}$ by
\begin{equation*}
\P \otimes_{\tau} \Q_{\cdot} \big[ A \big]:=\int_{\Om} \Q_{\om} \big[A \big] ~\P(d \om), \ \forall A \in \Fc_T.
\end{equation*}
	
\subsubsection{Hypotheses on $\Pc(t,\om)$}

We are given a family $\cP=\left (\Pc(t,\om) \right)_{(t,\om) \in [0,T] \x \Om}$ of sets of probability measures on $(\Om, \Fc_T)$, where $\Pc(t,\om) \subset \Pc^W_t$ for all $(t,\om) \in [0,T] \x \Om$. Denote also $\Pc_t := \cup_{\om \in \Om} \Pc(t,\om)$. We make the following assumption on the family $\left (\Pc(t,\om) \right)_{(t,\om) \in [0,T] \x \Om}$.
\begin{Assumption} \label{assum:main}
	
\vspace{0.5em}

$\mathrm{(i)}$ For every $(t,\om) \in [0,T] \x \Om$, one has $\Pc(t, \om) = \Pc(t, \om_{\cdot \wedge t})$ and $\P(\Om_t^{\om}) = 1$ whenever $\P \in \Pc(t, \om)$. The graph $[[ \Pc ]]$ of $\Pc$, defined by $[[ \Pc ]] := \{ (t, \om, \P): \P \in \Pc(t,\om) \}$, is upper semi-analytic in $[0,T] \x \Om \x \M_1$.

\vspace{0.5em}	
$\mathrm{(ii)}$ $\Pc$ is stable under conditioning, {\it i.e.} for every $(t,\om) \in [0,T] \x \Om$ and every $\P \in \Pc(t,\om)$ together with an $\F-$stopping time $\tau$ taking values in $[t, T]$, there is a family of r.c.p.d. $(\P_{\w})_{\w \in \Om}$ such that $\P_{\w}$ belongs to $\Pc(\tau(\w), \w)$, for $\P-a.e.$ $\w \in \Om$.

\vspace{0.5em}	
$\mathrm{(iii)}$ $\Pc$ is stable under concatenation, {\it i.e.} for every $(t,\om) \in [0,T] \x \Om$ and $\P \in \Pc(t,\om)$ together with a $\F-$stopping time $\tau$ taking values in $[t, T]$, let $(\Q_{\w})_{\w \in \Om}$ be a family of probability measures such that $\Q_{\w} \in \Pc(\tau(\w), \w)$ for all $\w \in \Om$ and $\w \longmapsto \Q_{\w}$ is $\Fc_{\tau}-$measurable, then the concatenated probability measure $\P \ox_{\tau} \Q_{\cdot} \in \Pc(t,\om)$.

\end{Assumption}

We notice that for $t = 0$, we have $\Pc_0 := \Pc(0, \om)$ for any $\om \in \Om$.

\subsubsection{Spaces and norms} \label{subsec:space_norms}

We now give the spaces and norms which will be needed in the rest of the paper. Fix some $t\in[0,T]$ and some $\omega\in\Omega$. In what follows, $\bG:=(\cG_s)_{t\leq s\leq T}$ will denote an arbitrary filtration on $(\Om,\Fc_T)$, and $\P$ an arbitrary element in $\Pc(t,\omega)$. Denote also by $\bG^{\P}$ the $\P-$augmented filtration associated to $\bG$.

\vspace{0.5em}

For $p \geq  1$, $\L^{p}_{t,\omega}(\bG)$ (resp. $\L^p_{t,\omega}(\bG,\P)$) denotes the space of all $\mathcal G_T-$measurable scalar random variable $\xi$ with
$$\No{\xi}_{\L^{p}_{t,\omega}}^p:=\underset{\mathbb{P} \in \mathcal{P}(t,\omega)}{\sup}\mathbb E^{\mathbb P}\left[|\xi|^p\right]<+\infty,\ \left(\text{resp. }\No{\xi}_{\L^{p}_{t,\omega}(\P)}^p:=\mathbb E^{\mathbb P}\left[|\xi|^p\right]<+\infty\right).$$

$\mathbb H^{p}_{t,\omega}(\bG)$ (resp. $\mathbb H^p_{t,\omega}(\bG,\P)$) denotes the space of all $\bG-$predictable $\mathbb R^d-$valued processes $Z$, which are defined $\widehat a_s ds-a.e.$ on $[t,T]$, with
\begin{align*}
&\No{Z}_{\mathbb H^{p}_{t,\omega}}^p:=\underset{\mathbb{P} \in \mathcal{P}(t,\omega)}{\sup}\mathbb E^{\mathbb P}\left[\left(\int_t^T\No{\widehat a_s^{\frac 1 2}Z_s}^2ds\right)^{\frac p2}\right]<+\infty, \\
& \left(\text{resp. } \No{Z}_{\mathbb H^{p}_{t,\omega}(\P)}^p:=\mathbb E^{\mathbb P}\left[\left(\int_t^T\No{\widehat a_s^{\frac 12}Z_s}^2ds\right)^{\frac p2}\right]<+\infty\right).
\end{align*}

$\mathbb M^p_{t,\omega}(\bG,\P)$ denotes the space of all $(\bG,\P)-$optional martingales $M$ with $\P-a.s.$ c\`adl\`ag paths on $[t,T]$, with $M_t=0$, $\P-a.s.$, and 
$$\No{M}_{\mathbb M^p_{t,\omega}(\P)}^p:=\E^\P\left[\left[M\right]_T^{\frac p2}\right]<+\infty.$$
Furthermore, we say that a family $(M^\P)_{\P\in\mathcal P(t,\omega)}$ belongs to $\mathbb M^p_{t,\omega}((\bG^\P)_{\P\in\Pc(t,\omega)})$ if, for any $\P\in\mathbb P(t,\omega)$, $M^\P\in\mathbb M^p_{t,\omega}(\bG^\P,\P)$ and 
$$\underset{\mathbb P\in\mathcal P(t,\omega)}{\sup}\No{M^\P}_{\mathbb M^p_{t,\omega}(\P)}<+\infty.$$

$\mathbb I^p_{t,\omega}(\bG,\P)$ 
denotes the space of all $\bG-$predictable 
processes $K$ with $\P-a.s.$ c\`adl\`ag and non-decreasing paths on $[t,T]$, with $K_t=0$, $\P-a.s.$, and
$$\No{K}_{\mathbb I^p_{t,\omega}(\P)}^p:=\E^\P\left[K_T^{p}\right]<+\infty.
~~
$$
We will say that a family $(K^\P)_{\P\in\mathcal P(t,\omega)}$ belongs to $\mathbb I^{p}_{t,\omega}((\bG^\P)_{\P\in\Pc(t,\omega)})$ 
if, for any $\P\in\mathcal P(t,\omega)$, $K^\P\in\mathbb I^{p}_{t,\omega}(\bG_\P,\P)$ 
and 
$$
\underset{\mathbb P\in\mathcal P(t,\omega)}{\sup}\No{K^\P}_{\mathbb I^p_{t,\omega}(\P)}<+\infty.
$$

$\mathbb D^{p}_{t,\omega}(\bG)$ (resp. $\mathbb D^p_{t,\omega}(\bG,\P)$) denotes the space of all $\bG-$progressively measurable $\mathbb R-$valued processes $Y$ with $\mathcal P(t,\omega)-q.s.$ (resp. $\P-a.s.$) c\`adl\`ag paths on $[t,T]$, with
$$\No{Y}_{\mathbb D^{p}_{t,\omega}}^p:=\underset{\mathbb{P} \in \mathcal{P}(t,\omega)}{\sup}\mathbb E^{\mathbb P}\left[\underset{t\leq s\leq T}{\sup}|Y_s|^p\right]<+\infty, \ \left(\text{resp. }\No{Y}_{\mathbb D_{t,\omega}^{p}(\P)}^p:=\mathbb E^{\mathbb P}\left[\underset{t\leq s\leq T}{\sup}|Y_s|^p\right]<+\infty\right).$$

For each $\xi \in \L^{1}_{t,\omega}(\bG)$ and $s \in [t,T]$ denote
$$\mathbb E_s^{\mathbb P,t,\omega,\bG}[\xi]:=\underset{\mathbb P^{'}\in \mathcal P_{t,\omega}(s, \mathbb P,\bG)}{\esup^{\mathbb P}}\mathbb E^{\mathbb P^{'}}[\xi|\cG_s] \text{ where } \mathcal P_{t,\omega}(s,\mathbb P,\bG):=\left\{\mathbb P^{'}\in\mathcal P(t,\omega),\ \mathbb P^{'}=\mathbb P \text{ on }\cG_s\right\}.$$
Then we define for each $p\geq \kappa\geq 1$,
$$\mathbb L_{t,\omega}^{p,\kappa}(\bG):=\left\{\xi\in \bL_{t,\omega}^{p}(\bG),\ \No{\xi}_{\mathbb L_{t,\omega}^{p,\kappa}}<+\infty\right\},$$
where
$$\No{\xi}_{\mathbb L_{t,\omega}^{p,\kappa}}^p:=\underset{\mathbb P\in\mathcal P(t,\omega)}{\sup}\mathbb E^{\mathbb P}\left[\underset{t\leq s\leq T}{\esup}^{\mathbb P}\left(\mathbb E_s^{\mathbb P,t,\omega,\F^+}[|\xi|^\kappa]\right)^{\frac{p}{\kappa}}\right].$$

Similarly, given a probability measure $\Pb$ and a filtration $\overline \bG$ on the enlarged canonical space $\Omb$, we denote the corresponding spaces by
$\mathbb D^p_{t,\omega}(\overline \bG,\Pb)$,
$\mathbb H^p_{t,\omega}(\overline \bG,\Pb)$,
$\mathbb M^p_{t,\omega}(\overline \bG,\Pb)$,
... Furthermore, when $t=0$, there is no longer any dependence on $\omega$, since $\omega_0=0$, so that we simplify the notations by suppressing the $\omega-$dependence and write $\H^p_0(\bG)$, $\H^p_0(\bG,\P)$, ... Similar notations are used on the enlarged canonical space.

When there is no ambiguity (only one probability measure $\bP$, see Appendix), the Brownian motion will be denoted by $W$ and for simplicity in the notations of integrability spaces, we remove the reference to the filtration $\bF$, the probability measure and $\omega$: $\mathbb D^p_{0,\omega}(\bF,\bP) = \bD^p$ and with the same convention $\H^p$, $\bM^p$ and $\bI^p$. Moreover for $\alpha \in \bR$, for $(Z,M,K) \in \bH^p \times \bM^p \times \bI^p$, we define
\begin{eqnarray*}
\|Z\|_{\bH^{p,\alpha}}^p & = & \E \left[ \left( \int_0^T e^{\alpha s} \|Z_s\|^2 ds \right)^{p/2} \right], \\
\|M\|_{\bM^{p,\alpha}}^p & = & \E \left[ \left( \int_0^T e^{\alpha s} d[M]_s \right)^{p/2} \right] \\
\|K\|_{\bI^{p,\alpha}}^p & = & \E \left[ \left( \int_0^T e^{\alpha s /2} dK_s \right)^{p} \right].
\end{eqnarray*}

\subsection{Assumptions on $f$ and $\xi$}

We shall consider a $\Fc_T-$measurable random variable $\xi : \Om \longrightarrow \R$ and a generator function 
\begin{equation*}
f: (t, \om, y, z, a, b)  \in [0,T] \x \Om \x \R \x \R^d \x \Sp \x \R^d \longrightarrow  \R.
\end{equation*}
Define for simplicity
$$\widehat{f}^\P_s(y,z):= f(s, X_{\cdot\wedge s},y,z,\widehat{a}_s, b^\P_s) ~~\mbox{and}~~ \widehat{f}^{\P,0}_s:= f(s, X_{\cdot\wedge s},0,0,\widehat{a}_s, b^\P_s).$$
The generator function $f$ is jointly Borel measurable and: 
\begin{enumerate}
\item[{\bf H1.}] For any $(t,\omega,z,a,b)$, the map $y \mapsto f(t,\omega,y,z,a,b)$ is continuous.
\item[{\bf H2.}] $f$ satisfies the monotonicity assumption w.r.t. $y$: there exists a constant $L_1 \in \bR$ such that for every $(t,\omega,y,y',z,a,b)$
$$(f(t,\omega,y,z,a,b)-f(t,\omega,y',z,a,b))(y-y') \leq L_1 (y-y')^2.$$
\item[{\bf H3.}] $f$ is Lipschitz continuous w.r.t. $z$ uniformly w.r.t. all other parameters, that is there exists a non-negative constant $L_2$ such that for every $(t,\omega,y,z,z',a,b)$, 
$$|f(t,\omega,y,z,a,b) - f(t,\omega,y,z',a,b)| \leq L_2 \|z-z'\|.$$
\item[{\bf H4.}] The following growth assumption w.r.t. $y$ holds: there exists $\sq > 1$ and a jointly Borel measurable function $\Psi : [0,T] \times \Om \x \Sp  \to \R $ such that for any $(t,\omega,a,b,y)$
$$|f(t,\omega,y,0,a,b) -f^{0}_t| \leq \Psi(t,\omega,a)(1+ |y|^\sq).$$
\end{enumerate}
$f^0_t$ is the notation for $f(t,\omega,0,0,a,b)$. We say that $f$ satisfies {\bf Condition (H)} if {\bf H1} to {\bf H4} hold. As for the generator, we denote
$$\widehat{\Psi}_s:= \Psi(s, X_{\cdot\wedge s},\widehat{a}_s).$$
Finally on $\xi$, $f^0$ and $\Psi$ we impose:
\begin{enumerate}
\item[{\bf C1.}] For some fixed constants $\varrho > 1$ and $\bsp > \varrho/(\varrho-1)$, one has for every $(t,\omega) \in[0,T] \x \Omega$,
\begin{equation*} 
\sup_{\P \in \Pc(t,\omega)}  \E^{\P} \left[  |\xi|^{\bsp \sq} + \int_t^T  \big| \widehat{f}^{\P,0}_s \big|^{\bsp \sq}  ds +\int_t^T  \big| \widehat{\Psi}_s \big|^{\varrho}  ds \right]  < + \infty.
\end{equation*}
\end{enumerate}
\begin{enumerate}
\item[{\bf C2.}] 
There is some $\kappa \in (1,\bsp \sq ]$ such that $\xi \in \bL^{\bsp \sq,\kappa}_{0}$ and 
$$\phi^{\bsp \sq ,\kappa}_f = \sup_{\bP \in \cP_0}\bE^\bP \left[ \underset{0\leq s\leq T}{\esssup}^\bP \left( \underset{\mathbb P^{'}\in \mathcal P_{0}(s, \mathbb P,\bF_+)}{\esup^{\mathbb P}}\mathbb E^{\mathbb P^{'}} \left[ \int_0^T |\widehat f^{\bP',0}_t|^\kappa dt \bigg| \cF_s^+\right] \right)^{\frac{\bsp \sq}{\kappa}} \right] < +\infty. $$
\end{enumerate}
{\bf Notation.} In the rest of the paper, $p$ denotes any number larger than 1: $p> 1$; $\sq$ denotes the exponent in Condition {\bf H4} (or {\bf H4'} in the appendix); $\bsp$ and $\varrho$ are used in Assumptions {\bf C1} and {\bf C2} and satisfy $\bsp > \varrho/(\varrho-1)$ ($\bsp$ is greater than the H\"older conjugate of $\varrho$). Finally we will sometimes assume $p$ verifies 
\begin{equation}\label{eq:cond_int}
1 < p \leq \frac{\varrho \bsp}{\varrho + \bsp} < \bsp.
\end{equation}
Under this condition, $\widehat p = \dfrac{p\bsp}{(\bsp - p)}\leq \varrho$.

\begin{Remark}[On condition {\bf H2}] \label{rem:H2_L1_equal_0}
It is well-known that we can suppose w.l.o.g.\ that $L_1=0$. Indeed if $(y,z,m)$ is a solution of \eqref{eq:general_BSDE_P} below, then $(\bar y,\bar z,\bar m)$ with
$$\bar y_t = e^{L_1 t} y_t, \quad \bar z_t = e^{L_1 t} z_t, \quad d \bar m_t = e^{L_1 t} dm_t $$
satisfies an analogous BSDE with terminal condition $\bar \xi = e^{L_1 T}\xi$ and generator
$$\bar f(t,y,z) = e^{L_1 t} f(t,e^{-L_1 t} y, e^{-L_1 t}z) - L_1 y.$$
$\bar f$ satisfies assumptions ${\bf (H)}$ with $L_1 = 0$. If we consider a super-solution of a BSDE (see Equation \eqref{eq:supersolution_general_BSDE}), the non-decreasing $k$ is replaced by $d\bar k_t = e^{L_1 t} dk_t$. 
Hence in the rest of this paper, we will sometimes assume w.l.o.g. that $L_1 = 0$.
\end{Remark}
\begin{Remark}[On condition {\bf H4}]
Let us explain why we assume Condition {\bf H4} together with the integrability condition {\bf C1} on $\widehat{\Psi}$, and not some weaker growth condition (as in \cite{pard:99} or \cite{bria:dely:hu:03} for standard BSDE). Indeed to prove the existence of a solution for a 2BSDE we will need that the solution $(y,z,m)$ of the standard BSDE with data $\widehat f^\bP$ and $\xi$ (see Equation \eqref{eq:general_BSDE_P}) is obtained by approximation with a sequence of solutions $(y^{n},z^{n},m^{n})$ of Lipschitz BSDEs (see Lemma \ref{lem:L2_approx} in the Appendix). Moreover the fact that $\Psi$ does not depend on $b$ is used for regularization of the paths in order to control the downcrossings (see Section \ref{sssect:path_regul}). Finally notice that this setting is sufficient to solve our optimal control problem \eqref{eq:control_pb_intro}. Existence under weaker conditions is left for further research. 
\end{Remark}

\begin{Remark}[On condition {\bf C1} on $\xi$ and $\widehat{f}^{\P,0}$]
Compared to the integrability assumption imposed in \cite{bria:dely:hu:03} for example, {\bf C1} looks to be too strong. Note again that it is sufficient to solve our control problem. As in the previous remark this hypothesis is related to the method we use to obtain existence of the solution of the 2BSDE. In particular in the Lipschitz approximation procedure and in the proof of existence of the solution of the reflected BSDE (see Section \ref{sect:mono_RBSDE} in Appendix). Weaker integrability condition is also left for further research. 
\end{Remark}

\subsection{Definition, uniqueness and properties}

We consider the 2BSDE
\begin{equation}  \label{eq:_general_2BSDE}
Y_t = \xi +  \int_t^T \widehat f^\bP_u(Y_u,\widehat a_u^{\frac{1}{2}}Z_u) du -\left( \int_t^T Z_u dX^{c,\bP}_u \right)^{\bP} - \int_t^T dM^\bP_u + \int_t^T dK^\bP_u.
\end{equation}
In this equation $\left( \int_t^T Z_u dX^{c,\bP}_u \right)^{\bP}$ denotes the stochastic integral of $Z$ w.r.t. $X^{c,\bP}$ under $\bP$, $M^\bP$ is a martingale orthogonal to $X^{c,\bP}$ and $K^\bP$ is a  non-decreasing process.

In this part we want to obtain the same result as \cite[Theorem 4.1]{poss:tan:zhou:15} for a monotone generator. The difference is that our generator is not Lipschitz continuous w.r.t. $y$. Here we follow the arguments developed in \cite{poss:tan:zhou:15} and we check that all the results contained in \cite{poss:tan:zhou:15} still hold in our setting. In other words we explain how their results can be extended under {\bf H2} and {\bf H4}. When the Lipschitz condition is not used, we simply refer to their paper.

\begin{Def} \label{def:sol_2BSDE}
$(Y,Z,M^\bP,K^\bP)$ is a solution if \eqref{eq:_general_2BSDE} is satisfied $\cP-q.s.$ and if the family $(K^\bP,\ \bP \in \cP)$ satisfies the {\it minimality condition:}
\begin{equation} \label{eq:minim_cond_K_2BSDE}
\underset{\bP' \in \cP(t,\bP,\bF_+)}{\essinf^\bP} \bE^{\bP'} \left[ \int_t^T \exp \left( \int_{t}^s \lambda_u^{\bP'} du \right)dK^{\bP'}_s \bigg|  \cF_t^+ \right] = 0, \quad 0 \leq t \leq T, \ \bP-a.s., \ \forall \bP \in \cP
\end{equation}
where
$$\lambda_s^{\bP'} = \frac{\widehat f^{\bP'}_s(Y_s,z^{\bP'}_s) - \widehat f^{\bP'}_s(y^{\bP'}_s,z^{\bP'}_s)}{Y_s-y^{\bP'}_s} \1_{Y_s\neq y^{\bP'}_s} \leq L_1.$$
\end{Def}
$\cP$-q.s.\! means quasi-surely, that is $\bP-a.s.$ for any $\bP \in \cP$.
In the above definition and in the rest of this section, $(y^\bP,z^\bP,m^\bP)$ is the solution under the probability measure $\bP$ of the following BSDE
\begin{equation}  \label{eq:general_BSDE_P}
y_t = \xi +  \int_t^T f(u, X_{\cdot\wedge u},y_u,\widehat a_u^{\frac{1}{2}}z_u,\widehat{a}_u, b^\P_u) du  -\left( \int_t^T z_u dX^{c,\bP}_u \right)^{\bP} - \int_t^T dm_u, \ \bP-a.s.
\end{equation}
where again $m$ is an additional martingale, orthogonal to $X^{c,\bP}$. Moreover for $t \leq s$ and a $\cF^+_s$-measurable random variable $\zeta$ , $y_{t}^{\bP}(s,\zeta)$ is the solution of \eqref{eq:general_BSDE_P} with terminal time $s$ and terminal condition $\zeta$.
\begin{Remark}[Notation for solution]
In the rest of this paper, $(Y,Z,M,K)$ is a solution of a 2BSDE, whereas $(y,z,m)$ denotes a solution of a standard BSDE and $(\widetilde y, \widetilde z, \widetilde m, \widetilde k)$ stands for the solution of a reflected BSDE. If necessary, the dependence w.r.t. the probability measure will be added as a superscript ($y^\bP$, $M^\bP$, ...). $y(\tau,\zeta)$ always denotes the first component of the solution of a BSDE with terminal time $\tau$ and terminal condition $\zeta$, where $\tau$ is a $\bF^+$-stopping time and $\zeta$ is $\cF^+_\tau$-measurable.
\end{Remark}

\begin{Remark}
The BSDE \eqref{eq:general_BSDE_P} is defined on $(\Omega,\cF^\bP_T,\bP)$ w.r.t. the filtration $\bF^\bP_+$ and is equivalent to the BSDE on $(\overline{\Omega},\overline{\cF}^{X}_T ,\overline{\bP})$ w.r.t the filtration $\overline{\bF}^{X,\overline{\bP}}_+$:
\begin{equation}  \label{eq:general_BSDE_enlarged_P}
\bar y_t = \xi(X_.) +  \int_t^T \widehat f^\bP_u(\bar y_u,\widehat a_u^{\frac{1}{2}}\bar z_u) du  -\left( \int_t^T \bar z_u dX^{c,\bP}_u \right)^{\overline{\bP}} - \int_t^T d\bar m_u, \quad \overline{\bP}-a.s.
\end{equation}
Moreover on the enlarged space $(\overline{\Omega},\overline{\cF}_T ,\overline{\bP})$, with the filtration $\overline{\bF}_+$, one defines the BSDE 
\begin{equation}  \label{eq:general_BSDE_enlarged_P_bis}
\widetilde y_t = \xi(X_.) +  \int_t^T \widehat f^\bP_u(\widetilde y_u,\widehat a_u^{\frac{1}{2}}\widetilde z_u) du  -\left( \int_t^T \widetilde z_u \widehat a^{\frac{1}{2}}_u dW^\bP_u \right)^{\overline{\bP}} - \int_t^T d\widetilde m_u, \quad \overline{\bP}-a.s.
\end{equation}
The key point is contained in \cite{poss:tan:zhou:15}, Lemma 2.2, where ``equivalence'' between the three BSDEs is proved and with straightforward modifications in the proof, this result holds under our conditions {\bf (H)} and {\bf C1}. 
\end{Remark}

\vspace{0.5cm}
Let us begin with the uniqueness result, which corresponds to \cite[Theorem 4.2]{poss:tan:zhou:15}.
\begin{Prop} \label{prop:uniq_2BSDE}
Under Conditions {\bf (H)}, {\bf C1} and {\bf C2}, let $(Y,Z,M^\bP,K^\bP)$ be a solution of \eqref{eq:_general_2BSDE} and for any $\bP \in \cP$, let $(y^\bP,z^\bP,m^\bP)$ be the solution of the BSDE \eqref{eq:general_BSDE_P} in $ \bD^{\bsp \sq}_0(\bF^{\bP}_+,\bP) \times \bH^{\bsp \sq}_0(\bF^{\bP}_+,\bP) \times \bM^{\bsp \sq}_0(\bF^{\bP}_+,\bP)$. Then for any $0\leq t_1\leq t_2 \leq T$
\begin{equation}\label{eq:representation_formula_2BSDE}
Y_{t_1} = \underset{\bP' \in \cP(t_1,\bP,\bF_+)}{\esssup^\bP} y_{t_1}^{\bP'}(t_2,Y_{t_2}).
\end{equation}
Thus uniqueness holds in $ \bD^{\bsp \sq}_0(\bF^{\cP_0}_+) \times \bH^{p}_0(\bF^{\cP_0}_+) \times \bM^{p}_0((\bF^{\bP}_+)_{\bP\in \cP_0}) \times \bI^{p}_0((\bF^{\bP}_+)_{\bP\in \cP_0})$ for any $1 < p$ satisfying Condition \eqref{eq:cond_int}.
\end{Prop}
\begin{proof}
In \cite{poss:tan:zhou:15}, the proof is divided in three steps. There is no modification in the first one. By comparison:
$$Y_{t_1} \geq \underset{\bP' \in \cP(t_1,\bP,\bF_+)}{\esssup^\bP} y_{t_1}^{\bP'}(t_2,Y_{t_2}), \quad \bP-a.s.$$
For the second step we have almost the same estimate on $K^{\bP'}$. For $p> 1$ satisfying \eqref{eq:cond_int}
\begin{eqnarray*}
(K^{\bP'}_{t_2} - K^{\bP'}_{t_1})^p & \leq & C \left[  \sup_{t_1\leq t \leq t_2} |Y_t|^{p} + \left( \int_{t_1}^{t_2} |\widehat f^{\bP',0}_s| ds \right)^p +  \left( \int_{t_1}^{t_2} |\widehat a_s^{\frac{1}{2}}Z_s| ds \right)^p \right] \\
& + & C \left[  \left( \int_{t_1}^{t_2} \widehat \Psi_s (1+|Y_s|^\sq) ds \right)^p  + \left| \int_{t_1}^{t_2} Z_s dX^{c,\bP'}_s \right|^p +  \left| \int_{t_1}^{t_2}  dM^{\bP'}_s \right|^p  \right] \\
& \leq  & C \left[  \sup_{t_1\leq t \leq t_2} |Y_t|^{p} + \left( \int_{t_1}^{t_2} |\widehat f^{\bP',0}_s| ds \right)^p +  \left( \int_{t_1}^{t_2} |\widehat a_s^{\frac{1}{2}}Z_s| ds \right)^p \right] \\
& + & C  \left[ \left| \int_{t_1}^{t_2} Z_s dX^{c,\bP'}_s \right|^p +  \left| \int_{t_1}^{t_2}  dM^{\bP'}_s \right|^p  \right] \\
& + & C \left[  \left( \int_{t_1}^{t_2} (\widehat \Psi_s)^p ds  \right) \left( 1 + \sup_{t_1 \leq t \leq t_2} |Y_s|^{p \sq } \right) \right].
\end{eqnarray*}
Hence 
\begin{eqnarray*}
\sup_{\bP'\in \cP(t_1,\bP,\bF_+)} \bE^{\bP'} \left[ (K^{\bP'}_{t_2} - K^{\bP'}_{t_1})^p  \right] & \leq & C \left[ \phi_f^{p,\kappa} + \|Y\|^{p}_{\bD_0^{p}}   + \|Z\|^p_{\bH_0^p} + \sup_{\bP\in \cP_0} \bE^\bP \left(\left[ M^\bP \right]_T \right)^{p/2} \right]\\
& +& C \left[ \|\widehat \Psi\|^p_{\bL^{\widehat p}} \left(1 + \|Y\|^{\bsp \sq}_{\bD_0^{\bsp \sq}} \right) \right]
\end{eqnarray*}
for $\widehat p = \frac{p\bsp}{(\bsp - p)}\leq \varrho$. The rest of this step does not change (upward directed family, see also Theorem 4.4 in \cite{sone:touz:zhan:12}) and thus 
$$C_{t_1}^{\bP} = \underset{\bP' \in \cP(t_1,\bP,\bF_+)}{\esssup^\bP} \bE^{\bP'} \left[ (K^{\bP'}_{t_2} - K^{\bP'}_{t_1})^p \bigg| \cF^+_{t_1} \right] < +\infty, \quad \bP-a.s.$$
The third step can be followed almost exactly. We define for $t \geq t_1$
$$\Lambda_{t}^{\bP'} =  \exp \left( \int_{t_1}^t \lambda_u^{\bP'} du \right) , \quad \Delta_{t}^{\bP'} = \exp\left[ -\int_{t_1}^t \eta_s^{\bP'} dW_s^{\bP'} - \frac{1}{2} \int_{t_1}^t \|\eta_s^{\bP'} \|^2 ds \right].$$
Estimate (4.6) in \cite{poss:tan:zhou:15} holds only for $\Delta$ and any constant $p>1$:
$$\bE^{\bP'\otimes \bP_0} \left[ \sup_{t_1\leq t \leq t_2} \left| \Delta^{\bP'}_t\right|^p \bigg| \cF_{t_1}^+\right] \leq C, \quad   \bP'\otimes \bP_0 -a.s.$$
since we only have an upper estimate on $\lambda_u^{\bP'} \leq L_1$. Then the linearization argument shows 
$$\delta  \cY_{t_1} = Y_{t_1} - \widetilde y^{\bP'\otimes \bP_0}_{t_1} = \bE^{\bP'\otimes \bP_0} \left[ \int_{t_1}^{t_2} \Lambda_s^{\bP'}\Delta_s^{\bP'} dK^{\bP'}_s \bigg| \overline{\cF}_{t_1}^+ \right].$$
Thus by the monotone condition {\bf H2} and for $1< p\leq \frac{\varrho \bsp}{\varrho + \bsp}$:
\begin{eqnarray*}
\delta \cY_{t_1} & \leq & \left( \bE^{\bP'\otimes \bP_0} \left[ \sup_{t_1 \leq s \leq t_2} (\Delta_s^{\bP'})^{\frac{p+1}{p-1}}  \bigg| \overline{\cF}_{t_1}^+ \right] \right)^{\frac{p-1}{p+1}} \left( \bE^{\bP'\otimes \bP_0} \left[ \left( \int_{t_1}^{t_2} \Lambda_s^{\bP'} dK^{\bP'}_s\right)^{\frac{p+1}{2}} \bigg| \overline{\cF}_{t_1}^+ \right] \right)^{\frac{2}{p+1}} \\
& \leq  & C \left( \bE^{\bP'\otimes \bP_0} \left[ \left( \int_{t_1}^{t_2} \Lambda_s^{\bP'} dK^{\bP'}_s\right)^{p} \bigg| \overline{\cF}_{t_1}^+ \right] \right)^{\frac{1}{p+1}}  \left( \bE^{\bP'\otimes \bP_0} \left[  \int_{t_1}^{t_2} \Lambda_s^{\bP'} dK^{\bP'}_s \bigg| \overline{\cF}_{t_1}^+ \right] \right)^{\frac{1}{p+1}} \\
& \leq & C \exp\left( \frac{pL_1T}{p+1} \right)\left( \bE^{\bP'\otimes \bP_0} \left[ \left(K^{\bP'}_{t_2} - K^{\bP'}_{t_1}\right)^{p} \bigg| \overline{\cF}_{t_1}^+ \right] \right)^{\frac{1}{p+1}}\left( \bE^{\bP'\otimes \bP_0} \left[  \int_{t_1}^{t_2} \Lambda_s^{\bP'} dK^{\bP'}_s \bigg| \overline{\cF}_{t_1}^+ \right] \right)^{\frac{1}{p+1}} \\
& \leq & C \exp\left( \frac{pL_1T}{p+1} \right) (C_{t_1}^{\bP})^{\frac{1}{p+1}} \left( \bE^{\bP'\otimes \bP_0} \left[  \int_{t_1}^{t_2} \Lambda_s^{\bP'} dK^{\bP'}_s \bigg| \overline{\cF}_{t_1}^+ \right] \right)^{\frac{1}{p+1}} .
\end{eqnarray*}
By arbitrariness of $\bP'$, and from the condition \eqref{eq:minim_cond_K_2BSDE}, we deduce
$$Y_{t_1} - \underset{\bP' \in \cP(t_1,\bP,\bF_+)}{\esssup^\bP} y_{t_1}^{\bP'}(t_2,Y_{t_2}) \leq 0, \quad \bP-a.s. $$
The end of the proof follows the arguments of \cite{poss:tan:zhou:15}.
\end{proof}


The comparison principle (\cite[Theorem 4.3]{poss:tan:zhou:15}) , the {\it a priori} estimate (\cite[Theorem 4.4]{poss:tan:zhou:15}) and the stability result (\cite[Theorem 4.5]{poss:tan:zhou:15}) for 2BSDE remain unchanged here. Indeed it is a direct consequence of Lemmas \ref{lem:comp_sol_gene_BSDE} and \ref{lem:L2_stability} and the formula \eqref{eq:representation_formula_2BSDE}. The other arguments follow exactly the proofs in \cite{poss:tan:zhou:15}.

\subsection{Existence of a solution of a solution for second-order BSDE}

In order to obtain a solution for the 2BSDE \eqref{eq:_general_2BSDE}, we define for any $(t,\omega) \in [0,T] \times\Omega$ 
\begin{equation} \label{eq:def_sup}
\widehat \cY_t(\omega) := \sup_{\bP \in \cP(t,\omega) }\bE^\bP( y^\bP_t ).
\end{equation}
This quantity $\widehat \cY$ is a ``candidate'' to be a solution of the 2BSDE \eqref{eq:_general_2BSDE}.

\subsubsection{Mesurability property of $\widehat \cY$}

Our aim is to prove that the conclusion of \cite[Theorem 2.1]{poss:tan:zhou:15} holds in our setting. To avoid to write again the complete machinery developed in \cite{poss:tan:zhou:15}, Section 2.4, we will use their Proposition 2.1. We already know that the solution $(y^\bP,z^\bP,m^\bP)$ exists since {\bf (H)} holds. Moreover from Lemma \ref{lem:L2_approx} and the condition on $\bsp$ and $\varrho$, $(y^\bP,z^\bP,m^\bP)$ can be approximated by solutions of Lipschitz BSDE in the space $\bD^p\times \bH^p \times \bM^p$ for any $1< p \leq \frac{\varrho \bsp}{\varrho + \bsp} < \bsp$. Moreover Lemmas \ref{lem:comp_sol_gene_BSDE} and \ref{lem:L2_stability} (for comparison and stability) hold. Thus the conclusion of \cite[Proposition 2.1]{poss:tan:zhou:15} is satisfied. Hence as in \cite[Theorem 2.1]{poss:tan:zhou:15}, the map 
$$(s,\omega,\bP) \mapsto \widehat \cY_s(\omega) = \sup_{\bP \in \cP(s,\omega) }\bE^\bP( y^\bP_s)$$
is measurable and 
$$(t,\omega) \mapsto \widehat \cY_t(\omega) $$
is $\mathcal{B}([0,T]) \otimes \cF_T$-universally measurable. Finally for all $(t,\omega) \in [0,T] \times \Omega$ and $\bF$-stopping times $\tau$ taking value in $[t,T]$ 
$$\widehat \cY_t(\omega) = \sup_{\bP \in \cP(t,\omega) } \bE^\bP( y^\bP_t(\tau, \widehat \cY_\tau)),$$
where $y^\bP(\tau, \widehat \cY_\tau)$ denotes the first component of the solution $(y,z,m)$ of the BSDE \eqref{eq:general_BSDE_P} with terminal time $\tau$ and terminal condition $\widehat \cY_\tau$ under the probability measure $\bP$. Notice that for any $\bP \in \cP(t,\omega)$ and any stopping time $\tau$ with values in $[t,T]$
\begin{equation} \label{eq:integrability_widehat_Y}
\bE^{\bP} (|\widehat \cY_\tau|^{\bsp \sq}) < +\infty.
\end{equation}
The proof is contained in \cite[Theorem 2.1]{poss:tan:zhou:15}. Let us recall the main ideas. First, for every $\P \in \Pc(t,\om)$ and $\eps > 0$, using the measurable selection theorem (see e.g. Proposition 7.50 of \cite{bert:shre:78} or Theorem III.82 in \cite{dell:meye:78}), one can choose a family of probability measures $(\Q^{\eps}_{\w})_{\w \in \Om}$ such that $\w \longmapsto \Q^{\eps}_{\w}$ is $\Fc_{\tau}-$measurable, and for $\P-a.e.$ $\w \in \Om$,
\begin{equation} \label{eq:DPP_interm_ineq}
\Q^{\eps}_{\w} \in \Pc(\tau(\w), \w)~~\mbox{and}~~\Ych_{\tau(\w)}(\w) - \eps\le \E^{\Q^{\eps}_{\w}} \big[ y_{\tau(\w)}^{\Q^{\eps}_{\w}}(T, \xi) \big] \le \Ych_{\tau(\w)}(\w).
\end{equation}
The integrability of $\Ych_{\tau}$ is a direct consequence of {\it a priori} estimates on the solution of BSDE \eqref{eq:general_BSDE_P} (see Lemma \ref{lem:a_priori_estim_BSDE} and the estimates below). We can then define the concatenated probability $\P^{\eps} := \P \otimes_{\tau} \Q^{\eps}_{\cdot}$ so that, by Assumption \ref{assum:main} (iii), $\P^{\eps} \in \Pc(t, \om)$. Notice that $\P$ and $\P^{\eps}$ coincide on $\Fc_{\tau}$ and hence $\E^{\P^{\eps}}\big[y^{\P^{\eps}}_{\tau} \big| \Fc_{\tau} \big] \in \L^{\bsp \sq}_{t,\om}(\Fc_\tau, \P)$. It follows then from the inequality in \eqref{eq:DPP_interm_ineq} that $\E^{\P}\big[ \big| \Ych_{\tau} \big|^{\bsp \sq} \big] < \infty$ and the upper bound depends on $\| \xi \|^{\bsp \sq}_{\bL_0^{\bsp \sq,\kappa}}$ and $\phi^{\bsp \sq,\kappa}_f $, but not on the choice of $\tau$.

\subsubsection{Path regularization}  \label{sssect:path_regul}

As in \cite[Section 3]{poss:tan:zhou:15}, to obtain a solution of the 2BSDE \eqref{eq:_general_2BSDE}, we shall characterize a c\`adl\`ag modification of $\Ych$ defined by \eqref{eq:def_sup}. Again we don't want to write all the details of the proof. Let us only explain the main difficulties due to the monotonicity condition {\bf H2}. The proof of \cite[Lemma 3.1]{poss:tan:zhou:15} does not use the Lipschitz property of $f$. 

The next step is to prove existence of right- and left- limits for $\Ych$ outside a $\cP_0$-polar sets (\cite[Lemma 3.2]{poss:tan:zhou:15}). The proof is based on a downcrossing estimate and the Lipschitz constant of $f$ w.r.t. $y$ explicitely appears. Since $f$ is no more Lipschitz continuous w.r.t. $y$, we show a downcrossing inequality for $\cY$, but under stronger condtions on $\xi$ and 
$\widehat{f}^{\P,0}$. Let us assume that there exists a constant $\mathfrak{C}$ such that for any $t$ and $\omega$
\begin{equation} \label{eq:ess_bounded_cond}
\underset{\P \in \Pc(t,\omega)}{\esssup}  \E^{\P} \left[  |\xi| + \sup_{s\in [t,T]} \big| \widehat{f}^{\P,0}_s \big|  \right]  \leq \mathfrak{C}.
\end{equation}
Under this condition and Estimate \eqref{eq:boundedness_y}, $y^\bP$ and $\Ych$ are also essentially bounded and we still denote by $\mathfrak{C}$ the upper bound.

\vspace{0.5cm}
\noindent {\bf An estimate on the downcrossings of $\widehat \cY$ under condition \eqref{eq:ess_bounded_cond}}. For simplicity we assume that $L_1= 0$ in the monotonicity condition {\bf H2} (see Remark \ref{rem:H2_L1_equal_0}) and we keep the same notations and the same scheme as in \cite{poss:tan:zhou:15}. For any $a < b$ and
for $J_N=\{\tau_0,\ldots \tau_N\}$ with $0=\tau_0< \tau_1 < \ldots < \tau_N = T$, a finite family of $\bF$-stopping times, we denote by $D_a^b(\widehat \cY, J_N)$ the number of downcrossings of the process $(\widehat \cY_{\tau_k}, \ 0\leq k\leq N)$ from $b$ to $a$. 

Let us fix $\bP \in \cP_0$ and consider the solution $(y^i,z^i,m^i)=(y^{i,\overline{\bP}_\omega^{\tau_{i-1}(\omega)}},z^{i,\overline{\bP}_\omega^{\tau_{i-1}(\omega)}},m^{i,\overline{\bP}_\omega^{\tau_{i-1}(\omega)}})$ of the BSDE with terminal condition $\widehat \cY_{\tau_i}$ and driver $\widehat f$ on the enlarged space under the probability measure $\overline{\bP}_\omega^{\tau_{i-1}(\omega)} = \bP_\omega^{\tau_{i-1}(\omega)} \otimes \bP_0$ and on the interval $[\tau_{i-1},\tau_i]$:
$$y^i_t = \widehat \cY_{\tau_i} + \int_t^{\tau_i}  \widehat f^{ \bP_\omega^{\tau_{i-1}(\omega)}}_u(y^i_u,\widehat a_u^{\frac{1}{2}} z^i_u) du - \int_t^{\tau_i}  z^i_u \widehat a^{\frac{1}{2}}_u dW^{ \bP_\omega^{\tau_{i-1}(\omega)}}_u - \int_t^{\tau_i} d m^i_u, \quad \overline{\bP}_\omega^{\tau_{i-1}(\omega)} -a.s.$$
We can linearize the previous BSDE (see the arguments before Equations \eqref{eq:linear_procedure} and \eqref{eq:density_chgt_measure}) to obtain
\begin{eqnarray*}
y^i_t &=& \widehat \cY_{\tau_i} + \int_t^{\tau_i} \left[ \widehat f^{ \bP_\omega^{\tau_{i-1}(\omega)}}_u(y^i_u,0) - \widehat f^{\bP_\omega^{\tau_{i-1}(\omega)},0}_u \right] du \\
& + & \int_t^{\tau_i}  \widehat f^{\bP_\omega^{\tau_{i-1}(\omega)},0}_u  du + \int_t^{\tau_i} \eta^i_u \widehat a_u^{\frac{1}{2}}  z^i_u du - \int_t^{\tau_i}  z^i_u \widehat a^{\frac{1}{2}}_u dW^{ \bP_\omega^{\tau_{i-1}(\omega)}}_u - \int_t^{\tau_i} d m^i_u.
\end{eqnarray*}
Note that we do not use the complete linearization of the BSDE. By the very definition of $\widehat \cY$, we get 
$$\bE^{\overline{\bP}_\omega^{\tau_{i-1}(\omega)} }(y^i_{\tau_{i-1}})  \leq \widehat \cY_{\tau_{i-1}}(\omega).$$ 

Now we consider again on $[\tau_{i-1},\tau_i]$ and the probability space, the solution of the following BSDE:
\begin{eqnarray*}
\widetilde y^i_t & = & \widehat \cY_{\tau_i} + \int_t^{\tau_i} \left[  \widehat f^{ \bP_\omega^{\tau_{i-1}(\omega)}}_u(\widetilde y^i_u,0)  - \widehat f^{\bP_\omega^{\tau_{i-1}(\omega)},0}_u  \right] \1_{\widetilde y^i_u \geq 0}du - \int_t^{\tau_i} \left| \widehat f^{\bP_\omega^{\tau_{i-1}(\omega)},0}_u \right| du \\
&& + \int_t^{\tau_i} \eta^i_u \widehat a_u^{\frac{1}{2}} \widetilde z^i_u du - \int_t^{\tau_i}  \widetilde z^i_u \widehat a^{\frac{1}{2}}_u dW^{ \bP_\omega^{\tau_{i-1}(\omega)}}_u - \int_t^{\tau_i} d \widetilde m^i_u, \quad \overline{\bP}_\omega^{\tau_{i-1}(\omega)} -a.s.
\end{eqnarray*}
Here the generator is 
$$(t,y,z) \mapsto \left[  \widehat f^{ \bP_\omega^{\tau_{i-1}(\omega)}}_t(y,0)  - \widehat f^{\bP_\omega^{\tau_{i-1}(\omega)},0}_t  \right] \1_{y \geq 0} - \left| \widehat f^{\bP_\omega^{\tau_{i-1}(\omega)},0}_t \right| + \eta^i_t \widehat a_t^{\frac{1}{2}} z$$
and satisfies Condition {\bf (H)}. By the monotonicity condition {\bf H2} with $L_1=0$, for $s\in[\tau_{i-1}, \tau_i]$
$$ \left[ \widehat f^{ \bP_\omega^{\tau_{i-1}(\omega)}}_s(y,0) - \widehat f^{\bP_\omega^{\tau_{i-1}(\omega)},0}_s \right] \mbox{sign}(y) \leq 0.$$
In particular when $y \geq 0$, the increment is  non-positive. Hence from the comparison principle for BSDEs, we have: $\widetilde y^i_{\tau_{i-1}}  \leq y^i_{\tau_{i-1}}$. 

Let for $t \in [\tau_{i-1},\tau_i]$
$$L_t := \cE \left( \int_{\tau_{i-1}}^t \eta^i_u dW^{ \bP_\omega^{\tau_{i-1}(\omega)}}_u \right),$$
be the stochastic exponential and 
$$\Xi^i_t =- \left[ \widehat f^{ \bP_\omega^{\tau_{i-1}(\omega)}}_t(\widetilde y^i_t,0) - \widehat f^{\bP_\omega^{\tau_{i-1}(\omega)},0}_t \right] \1_{\widetilde y^i_t \geq 0} .$$
From our previous arguments, the key point is that $\Xi^i_t \geq 0$. Then
\begin{eqnarray*}
\bE^{\overline{\bP}_\omega^{\tau_{i-1}(\omega)} } \left[ L_{\tau_i} \left( \widehat \cY_{\tau_i} - \int_{\tau_{i-1}}^{\tau_i} \Xi^i_u du -  \int_{\tau_{i-1}}^{\tau_i} \left| \widehat f^{\bP_\omega^{\tau_{i-1}(\omega)},0}_u \right| du 
\right) \right] \leq  \widehat \cY_{\tau_{i-1}}(\omega).
\end{eqnarray*}
And by definition of the r.c.p.d.
\begin{eqnarray*}
\bE^{\overline{\bQ} } \left[  \widehat \cY_{\tau_i} - \int_{\tau_{i-1}}^{\tau_i} \Xi^i_u du -  \int_{\tau_{i-1}}^{\tau_i} \left| \widehat f^{\bP_\cdot^{\tau_{i-1}(\cdot)},0}_u \right| du
 \bigg| \overline{\cF}^+_{\tau_{i-1}}\right] \leq  \widehat \cY_{\tau_{i-1}}, \ \bP \otimes \bP_0-a.s.
\end{eqnarray*}
where $\overline{\bQ}$ is equivalent to $\bP \otimes \bP_0$ with density
$$\frac{d\overline{\bQ}}{d(\bP \otimes \bP_0)}=\cE \left( \int_{\tau_{i-1}}^t \eta^i_u dW^{ \bP}_u \right), \quad t \in [\tau_{i-1},\tau_i].$$
We define $\Xi_t = \sum_{i=1}^n  \Xi^i_t \1_{[\tau_{i-1},\tau_i)}(t)$ and the discrete process
$$V_i := V_{\tau_i} = \widehat \cY_{\tau_i}  - \int_{0}^{\tau_i} \left( \Xi_s  + \left| \widehat f^{\bP,0}_s \right| \right)ds$$
For $b>0$ let
$$\overline{V}_{i} := V_i \wedge \left( b -  \int_0^{\tau_i} \left( \Xi_s + |\widehat f^{\bP,0}_s|\right) ds \right).$$
These two processes $V$ and $\overline{V}$ are $\overline{\bQ}$-supermartingales relative to $\overline{\bF}$ (see also the proof of Lemma A.1 in \cite{bouc:poss:tan:15} for more details). Up to this point we do not change the proof of \cite{poss:tan:zhou:15}, since we do not use Lipschitz continuity argument.

We also introduce 
$$u_t = b - \int_0^t   \left( \Xi_s + |\widehat f^{\bP,0}_s|\right) ds ,\quad \ell_t = - \int_0^t   \left( \Xi_s + |\widehat f^{\bP,0}_s|\right) ds,$$
together with: $u_i = u_{\tau_i}$ and $\ell_i = \ell_{\tau_i}$. Remark that $D_0^b (\widehat \cY,J_N) \leq D_\ell^u (V,J_N) =  D_\ell^u (\overline{V},J_N)$. Let us now explain how to derive a control on the downcrossings under the monotonicity condition, using the proof of Inequality (12.5), page 446 in \cite{doob:01} (see pages 448--449). We define $\theta_0 = 0$, 
$$S_1 = \min \{ j \geq 0, \overline{V}_j \geq u_j \}, \quad \theta_1 = S_1 \wedge N$$
and
$$S_k = \left\{ \begin{array}{ll}
\min \{ j > \theta_{k-1}, \overline{V}_j \geq u_j \}, & k \mbox{ odd}, \ k\geq 3,\\
\min \{ j > \theta_{k-1}, \overline{V}_j \leq \ell_j \}, & k \mbox{ even},\ k\geq 2,
\end{array} \right.$$
$\theta_k = S_k \wedge N$. We have 
$$\overline{V}_0 - \overline{V}_N = \sum_{j=0}^{N-1} \left[\overline{V}_{\theta_j} - \overline{V}_{\theta_{j+1}} \right].$$
Each bracket has a non-negative expectation (supermartingale inequality). We shall give a lower bound for each bracket with odd $j$.
On the set where the number of downcrossings is $k$, the first $k$ brackets in the above sum with odd $j$ are larger than:
$u_{T_j} - \ell_{T_{j+1}} \geq u_{T_j} - \ell_{T_{j}} = b$ 
 since $\ell$ is decreasing. For the other terms (again with odd $j$), only $\overline{V}_{T_{2k+1}} - \overline{V}_{T_{2k+2}}$ (i.e. $j=2k+1$) may be non zero and is bounded from below by $u_T - \overline{V}_T$. Hence we obtain the next estimate:
$$\bE^{\overline{\bQ}} \left[ \overline{V}_0 - \overline{V}_N \right] \geq  \bE^{\overline{\bQ}} \left[ b D_\ell^u (\overline{V},J_N) \right] +  \bE^{\overline{\bQ}} \left[(u_T - \overline{V}_T )\wedge 0 \right].$$
Thereby
\begin{eqnarray*}
b \bE^{\overline{\bQ}} \left[ D_0^b (\widehat \cY,J_N) \right] & \leq & b \bE^{\overline{\bQ}} \left[ D_\ell^u (\overline{V},J_N) \right]\\
& \leq &   \bE^{\overline{\bQ}} \left[\overline{V}_0 - \overline{V}_T - (u_T - \overline{V}_T )\wedge 0 \right] \\
 & \leq & \bE^{\overline{\bQ}} \left[(\widehat \cY_{0}\wedge b)- (\widehat \cY_T\wedge b)   + \int_0^{T}   \left( \Xi_s + |\widehat f^{\bP,0}_s|\right) ds  \right] \\
 & \leq & \bE^{\overline{\bQ}} \left[(\widehat \cY_{0}\wedge b) + (\widehat \cY_T\wedge b)^- + \int_0^{T}   \left( \Xi_s + |\widehat f^{\bP,0}_s|\right) ds  \right] . 
\end{eqnarray*}
Since $\eta$ is a bounded process, using H\"older's inequality, we get that for some $1 < p \leq \varrho$, there exists a constant $C$ depending on $p$ and $L_2$ such that 
$$b \bE^{\overline{\bQ}} \left[ D_0^b (\widehat \cY,J_N) \right]  \leq C \left( \bE^{\overline{\bP}} \left[(\widehat \cY_{0}\wedge b)^p + ((\widehat \cY_T\wedge b)^-)^p + \int_0^{T}   \left( (\Xi_s)^p + |\widehat f^{\bP,0}_s|^p\right) ds  \right] \right)^{1/p}.$$
To finish the proof, from Condition {\bf C2} and Estimate \eqref{eq:integrability_widehat_Y}, we only need to control the term $\Xi$ of the right-hand side. Recall that on $[\tau_{i-1},\tau_i)$,
$$ \Xi_t = \Xi^i_t = - \left[ \widehat f^{ \bP_\omega^{\tau_{i-1}(\omega)}}_t(\widetilde y^i_t,0) - \widehat f^{\bP_\omega^{\tau_{i-1}(\omega)},0}_t \right] \1_{\widetilde y^i_t \geq 0}$$
is non-negative. Using condition \eqref{eq:ess_bounded_cond} and Estimate \eqref{eq:boundedness_y}, from Hypothesis {\bf H4}, we deduce that for $t\in[\tau_{i-1},\tau_i)$ 
\begin{eqnarray*}
| \Xi_t  |^p & \leq & (1+\mathfrak{C}^\sq)^{p} \left( \widehat \Psi_t \right)^p .
\end{eqnarray*}
Since $p \leq \varrho$, from condition {\bf C1}, there exists a constant $C$ independent of the choice of $\tau_i$ such that
$$\bE^{\overline{\bP}} \left[ \int_0^{T}   (\Xi_s)^p  ds \right] \leq C.$$
Therefore $D_0^b (\widehat \cY,J_N)$ is $\overline{\bQ}-a.s.$ finite. Then for $a < b$ we have the same inequality:
\begin{equation*} 
\bE^{\overline{\bQ}} \left[ D_a^b (\widehat \cY,J_N) \right]  \leq \frac{1}{b-a}  \bE^{\overline{\bQ}} \left[(\widehat \cY_{0}\wedge (b-a)) + (\widehat \cY_T\wedge (b-a))^-  + \int_0^{T}\left( |\widehat f^{\bP,0}_s| - \Xi_s \right) ds  \right].
\end{equation*}
This estimate implies that $D_a^b (\widehat \cY,J_N)$ is $\overline{\bQ}-a.s.$ finite. Since the right-hand side does not depend on $N$, we can extend this estimate to any countable family of $\bF$-stopping times. And the conclusion of \cite[Lemma 3.2]{poss:tan:zhou:15} still holds under \eqref{eq:ess_bounded_cond}.

\vspace{0.5cm}
Let us define $\widehat \cY^+$ by
$$\widehat \cY^+_t := \limsup_{r\in \bQ \cap(t,T], r \downarrow t} \widehat \cY_r. $$
Let us stress that \cite[Lemmata 3.3 and 3.5]{poss:tan:zhou:15} does not use Lipschitz continuity of the generator w.r.t. $y$. As we did for the downcrossing estimate, we adapt the proof of \cite[Lemma 3.4]{poss:tan:zhou:15} to obtain that $\widehat \cY^+$ is c\`adl\`ag, $\cP_0-q.s.$. Moreover since the Lipschitz continuity w.r.t. $y$ is not involved, the representation formula of \cite[Lemma 3.5]{poss:tan:zhou:15} holds, that is for any $0\leq t \leq T$, for any $\bP\in \cP_0$, we have $\bP-a.s.$
$$\widehat \cY^+_t = \underset{\bP' \in \cP_0(t,\bP,\bF_+)}{\mbox{esssup}^\bP} y^{\bP'}_t(T,\xi).$$
From condition \eqref{eq:ess_bounded_cond}, we deduce that $\widehat \cY^+$ is essentially bounded (again by $\mathfrak{C}$) and thus belongs to $\bD^{\bsp \sq}_0 (\bF^{\cP_0+})$. Finally from our Section \ref{sect:mono_RBSDE} on reflected BSDE, we can argue as in \cite[Lemma 3.6]{poss:tan:zhou:15} and we obtain the next result.

\begin{Lemma} \label{lem:regularity_estim}
Under Conditions {\bf (H)-C1-C2} and assumption \eqref{eq:ess_bounded_cond}, this process $\widehat \cY^+$ is c\`adl\`ag, $\mathcal{P}_0$-q.s. and belongs to $\bD^{\bsp \sq}_0(\bF^{\mathcal{P}_0 +})$. Moreover it is a semi-martingale under any $\bP \in \cP_0$ with an explicit decomposition: there exists $(Z^\bP,M^\bP,K^\bP) \in \bH^p_0(\bF^{\bP+},\bP) \times \bM^p_0(\bF^{\bP+},\bP) \times \bI^p_0(\bF^{\bP+},\bP)$ with $\displaystyle 1 < p \leq \frac{\varrho \bsp}{\varrho + \bsp}$ and for any $t\in [0,T]$, $\bP$-a.s. 
$$\widehat \cY^+_t = \xi + \int_t^T \widehat f^\bP_s(\widehat \cY^+_s,\widehat a^{\frac{1}{2}}_s Z^\bP_s) ds - \int_t^T Z^\bP_s dX^{c,\bP}_s - \int_t^T dM^\bP_s + \int_t^T dK^\bP_s.$$
Moreover there is some $\bF^{\cP_0}$-predictable process $Z$ which aggregates the family $(Z^{\bP})_{\bP \in \cP_0}$.
\end{Lemma}

\subsubsection{Conclusion}

Now we come to the existence result (equivalent to \cite[Theorems 4.1 and 4.4]{poss:tan:zhou:15}).
\begin{Prop} \label{prop:2BSDE_existence}
Under Conditions {\bf (H)-C1-C2}, there exists a solution $(Y,Z,M^\bP,K^\bP)$ to the 2BSDE \eqref{eq:_general_2BSDE} in the space $ \bD^{\bsp \sq}_0(\bF^{\cP_0}_+) \times \bH^p_0(\bF^{\cP_0}_+) \times \bM^p_0((\bF^{\bP}_+)_{\bP\in \cP_0}) \times \bI^p_0((\bF^{\bP}_+)_{\bP\in \cP_0})$ for any $p> 1$ satisfying \eqref{eq:cond_int}. More precisely there exists a constant $C$ depending on $\bsp$, $\sq$ $T$, $L_1$, $L_2$  such that 
\begin{equation}\label{eq:estimtes_2BSDE_sol}
\|Y\|^{\bsp\sq}_{\bD_0^{\bsp \sq}} + \|Z\|^p_{\bH_0^p} + \sup_{\bP\in \cP_0} \bE^\bP \left( K^\bP_T \right)^p+ \sup_{\bP\in \cP_0} \bE^\bP \left(\left[ M^\bP \right]_T \right)^{p/2} \leq C \left(\|\xi \|^{\bsp \sq}_{\bL_0^{\bsp \sq}} + \phi_f^{\bsp\sq,\kappa} \right).
\end{equation}
\end{Prop}
\begin{proof}
For the existence we argue as in \cite{poss:tan:zhou:15}, except for the minimality condition on $K^\bP$, together with a truncation procedure. Let us define for any $n \in \N$
$$\xi^n = -n\vee \xi \wedge n, \quad \widehat{f}^{\P,0,n}_s=-n\vee  f(s, X_{\cdot\wedge s},0,0,\widehat{a}_s, b^\P_s) \wedge n.$$
$\xi^n$ and $\widehat{f}^{\P,0,n}$ obviously verify condition \eqref{eq:ess_bounded_cond} with $\mathfrak{C}=n$. From Lemma \ref{lem:regularity_estim}, we obtain the existence of a solution $(Y^n,Z^n,M^{n,\bP},K^{n,\bP})$ to the 2BSDE \eqref{eq:_general_2BSDE} with terminal condition $\xi^n$ and generator $f^n$ defined by
$$f^n(t, \om, y, z, a, b) =\left( f (t, \om, y, z, a, b) - \widehat{f}^{\P,0}_t \right)+ \widehat{f}^{\P,0,n}_t.$$
Note that $f^n$ satisfies Conditions {\bf (H)-C1-C2}. The minimality criterion on $K^{n,\bP}$ is proved arguing as in the proof of minimality for $K^\bP$ (see just below).

Now the stability result shows that the sequence $(Y^n,Z^n,M^{n,\bP}-K^{n,\bP})$ converges in $\bD_0^{\bsp \sq} \times \bH_0^p \times  \bM^p_0$ to some process $(Y,Z,N^\bP)$. The supermartingale $N^\bP$ can be decomposed: $N^\bP = M^\bP - K^\bP$, where $M^\bP$ is a martingale under $\bP$, orthogonal to the canonical process and $K^\bP$ is a  non-decreasing process. The limit $(Y,Z,M^\bP,K^\bP)$ is a solution of the 2BSDE \eqref{eq:_general_2BSDE} if $K^\bP$ satisfies the required minimality condition. 

Let us prove the minimality criterion for $K^\bP$ (again the proof is the same for $K^{n,\bP}$). For $\bP' \in \cP(t,\bP,\bF_+)$, let us denote $\delta \widehat \cY^+ = \widehat \cY^+ - y^{\bP'}(T,\xi)$ and use again a linearization argument:
$$\delta \widehat \cY^+_t = \bE^{\bP'\otimes \bP_0} \left[ \int_t^T \Lambda_s^{\bP'}\Delta_s^{\bP'} dK^{\bP'}_s \bigg| \overline{\cF}_t^+ \right]$$
with
$$\Lambda_s^{\bP'} =  \exp \left( \int_{t}^s \lambda_u^{\bP'} du \right) , \quad \Delta_s^{\bP'} = \exp\left[ -\int_t^s \eta_s^{\bP'} dW_s^{\bP'} - \frac{1}{2} \int_t^s \|\eta_s^{\bP'} \|^2 ds \right].$$
Thus $\bP$-a.s.
$$\delta \widehat \cY^+_t \geq \bE^{\bP'\otimes \bP_0} \left[\inf_{t \leq s \leq T} \Delta_s^{\bP'}  \int_t^T\Lambda_s^{\bP'} dK^{\bP'}_s \bigg| \overline{\cF}_t^+ \right]$$
and for $p$ satisfying \eqref{eq:cond_int}, let $p'$ be the H\"older conjugate of $p$:
\begin{eqnarray*}
&& \bE^{\bP'\otimes \bP_0} \left[ \int_t^T \Lambda_s^{\bP'} dK^{\bP'}_s \bigg| \overline{\cF}_t^+ \right] \leq  \left\{\bE^{\bP'\otimes \bP_0} \left[\inf_{t \leq s \leq T} \Delta_s^{\bP'}  \int_t^T\Lambda_s^{\bP'} dK^{\bP'}_s \bigg| \overline{\cF}_t^+ \right] \right\}^{1/2} \\
&& \qquad \times \left\{\bE^{\bP'\otimes \bP_0} \left[e^{pL_1T}(K^{\bP'}_T-K^{\bP'}_t)^p \bigg| \overline{\cF}_t^+ \right] \right\}^{1/(2p)} \left\{\bE^{\bP'\otimes \bP_0} \left[ \left( \inf_{t \leq s \leq T} \Delta_s^{\bP'} \right)^{-p'} \bigg|\overline{\cF}_t^+ \right] \right\}^{1/{2p'}} \\
&& \quad \leq C_T (C_{t}^{\bP} )^{1/(2p)} (\delta \widehat \cY^+_t)^{1/2}.
\end{eqnarray*}
Hence the condition \eqref{eq:minim_cond_K_2BSDE} follows now immediately. 

To obtain the {\it a priori} estimate \eqref{eq:estimtes_2BSDE_sol} for the solution of the 2BSDE, we use the {\it a priori} estimate given in Lemma \ref{lem:a_priori_estim_BSDE}, the representation formula \eqref{eq:representation_formula_2BSDE} and we argue as in the proof of \cite[Theorem 4.4]{poss:tan:zhou:15}. 
\end{proof}

\subsection{Discussion and comparison with \cite{poss:13}} \label{ssect:discussion}

When $f$ is Lipschitz continuous w.r.t. $y$, the process $\lambda$ is bounded also from below. Thus our minimality condition is equivalent to the classical one:
\begin{equation} \label{eq:class_minim_cond_K}
\underset{\bP' \in \cP(t,\bP,\bF_+)}{\essinf^\bP} \bE^{\bP'} \left[K^{\bP'}_T - K^{\bP'}_t\bigg|  \cF_t^+ \right] = 0, \quad 0 \leq t \leq T, \ \bP-a.s., \ \forall \bP \in \cP_0.
\end{equation}
In general we only have that the classical condition \eqref{eq:class_minim_cond_K} implies \eqref{eq:minim_cond_K_2BSDE}.

If there is only one probablity measure $\bP$ in $\cP_0$, the minimality condition \eqref{eq:minim_cond_K_2BSDE} imposed on $K^\bP$ should imply that $K^\bP$ is equivalent to zero. In the Lipschitz setting this is a direct consequence of \eqref{eq:class_minim_cond_K}. In our setting it is still true but the arguments are not direct. From the proof of Proposition \ref{prop:uniq_2BSDE}, \eqref{eq:minim_cond_K_2BSDE} implies uniqueness of the solution. But if $\cP_0$ is the singleton, the solution $(y^\bP,z^\bP,0)$ of the classical BSDE \eqref{eq:general_BSDE_P} becomes a solution of the 2BSDE \eqref{eq:_general_2BSDE}. By uniqueness, $K^\bP\equiv 0$.

The monotone case was already studied in \cite{poss:13}. The generator $f$ satisfies Condition {\bf (H)} and is uniformly continuous in $y$, uniformly in $(t,\omega,z,a)$ and has the linear growth property:
$$|f(t,\omega,y,0,a)|\leq |f(t,\omega,0,0,a)|+C(1+|y|).$$
Then under some integrability condition on $\xi$ and $\widehat f_s^{\bP,0}$, from \cite[Theorem 2.2]{poss:13}, there exists a unique solution of the 2BSDE \eqref{eq:_general_2BSDE} such that $K^\bP$ satisfies the minimality condition \eqref{eq:class_minim_cond_K}. 

Therefore if the generator $f$ satisfies the assumptions of \cite{poss:13} and Condition {\bf (H)}, then the solution obtained by \cite{poss:13} with minimality condition \eqref{eq:class_minim_cond_K} is also the solution given by Propositions \ref{prop:uniq_2BSDE} and \ref{prop:2BSDE_existence} with minimality criterion \eqref{eq:minim_cond_K_2BSDE}. Let us emphasize that the ways to obtain the solution are completely different. Indeed in \cite{poss:13} the generator is approximated by a sequence of Lipschitz generators $f_n$. For any fixed $n$ using \cite{sone:touz:zhan:12}, there exists a unique solution $(Y^n,Z^n,M^{n,\bP}, K^{n,\bP})$ to the 2BSDE \eqref{eq:_general_2BSDE} with generator $f_n$ and $K^{n,\bP}$ verifies \eqref{eq:class_minim_cond_K}. Then the core of the paper \cite{poss:13} consists to show that the sequence $(Y^n,Z^n,M^{n,\bP}, K^{n,\bP})$ converges to $(Y,Z,M^\bP,K^\bP)$ and that \eqref{eq:class_minim_cond_K} is preserved through the limit. The uniform continuity and the linear growth conditions of $f$ w.r.t. $y$ are crucial there.

\section{Liquidation problem} \label{sect:liquidation_pb}

\subsection{The standard formulation of \cite{anki:jean:krus:13,krus:popi:15}} \label{sect:standard_form_liquid_pb}

In \cite{krus:popi:15} the authors consider a probability space $(\Omega,\cF,\bP)$. The filtration $\bF$ is assumed to be complete, right continuous on $[0,T]$ and left-continuous at time $T$ (see \cite{popi:16} for details on this assumption). In \cite{anki:jean:krus:13}, $\bF$ is generated by a $d$-dimensional Brownian motion and thus is quasi-left continuous. 

Given $\xi$ a $\cF_T$-measurable non-negative random variable such that $\cS = \{\xi=+\infty\}$ has a positive probability, $\eta$ (resp. $\gamma$) a positive (resp. non-negative) process, the studied optimal stochastic control problem is defined as follows. For some $\kq > 1$, consider the functional 
\begin{equation}\label{eq:control_pb}
J(t,\cX) = \bE \left[  \int_t^T \left( \eta_s |\alpha_s|^{\kq} + \gamma_s |\cX_s|^{\kq} \right) ds + \xi |\cX_T|^{\kq} \bigg| \cF_t \right]
\end{equation}
over all progressively measurable processes $\cX$ that satisfy the dynamics
\begin{equation}\label{eq:state_dyn}
\cX_s =x +\int_t^s \alpha_u du , \quad s \geq t
\end{equation}
for some $\alpha$ with $\int_t^T |\alpha_s | ds < +\infty$ $\bP$-a.s., and $x \geq 0$. To have a finite value $J(t,\cX)$, the terminal state constraint is
\begin{equation*}
\cX_T \1_{\cS}= 0
\end{equation*}
together with the convention $0\times \infty = 0$. 
The set of such processes $\cX$ is denoted by $\cA^0_{\cS}(t,x)$. 
We introduce the random field $v$ that represents for each initial condition $(t,x)$ the minimal value of $J(t,\cX)$
\begin{equation}\label{eq:value_fct}
v(t,x) = \underset{\cX\in \cA^0_{\cS}(t,x)}{\essinf} J(t,\cX).
\end{equation}
We follow the convention that the infimum over the empty set is equal to $\infty$. For some $L>0$ we also consider the unconstrained minimization problem:
\begin{eqnarray} \nonumber
v^L(t,x) &=& \underset{\cX\in \cA(t,x)}{\essinf} J^L(t,\cX) \\ \label{eq:unconstraint_value_fct}
&  = & \underset{\cX\in \cA(t,x)}{\essinf} \bE \left[  \int_t^T \left( \eta_s |\alpha_s|^{\kq} + (\gamma_s \wedge L) |\cX_s|^{\kq} \right) ds  +(L\wedge \xi) |\cX_T|^{\kq}  \bigg| \cF_t \right]
\end{eqnarray}
where $\cA(t,x)$ is the set of all progressively measurable processes $\cX$ of the form \eqref{eq:state_dyn}. Here no terminal constraint is imposed on $\cX$. 

In \cite{anki:jean:krus:13,krus:popi:15} the authors show that the related singular BSDE is of the following form:
\begin{equation} \label{eq:bsde}
dy_t  =   \frac{y_t^{\sq}}{(\sq-1)\eta_t^{\sq-1}} dt  - \gamma_t dt +  z_t dW_t + dm_t
\end{equation}
with terminal condition equal to $\xi$. Here $\sq>1$ is the H\"older conjugate of $\kq$: $(\kq-1)(\sq-1)=1$. The processes $\eta$ and $\gamma$ satisfy for some $\ell > 1$
$$\bE \int_0^T \left[ (\eta_t + (T-t)^{\kq} \gamma_t)^\ell + \frac 1{\eta_t^{\sq-1}}\right] dt <\infty.$$
It is proved that the singular BSDE \eqref{eq:bsde} has a minimal super-solution $(y,z,m)$ satisfying:
\begin{enumerate}
\item for any $t < T$
$$\bE \left[ \sup_{ 0 \leq s \leq t} |y_s|^\ell + \left( \int_0^t |z_s|^2 ds \right)^{\ell/2} + [m]_t^{\ell/2}\right] < +\infty;$$
\item $Y_t \geq 0$ for any $t$, {\it a.s.}
\item for all $0\leq s \leq t < T$ 
\begin{eqnarray*}
y_{s}  & = & y_{t} + \int_{s}^{t} \left[ - \frac{y_u^{\sq}}{(\sq-1)\eta_u^{\sq-1}} + \gamma_u \right] du - \int_{s}^{t} z_u dW_u - \int_s^t dm_u.
\end{eqnarray*}
\item and the singular terminal condition: $\bP-a.s.$
\begin{equation} \label{eq:terminal_cond}
\liminf_{t \to T} y_{t } \geq \xi.
\end{equation}
\end{enumerate}
To prove the existence of a minimal solution, a truncation procedure is used. For any $L \geq 0$ we consider the BSDE
\begin{equation} \label{eq:truncated_bsde}
dy^L_t = \frac{(y^L_t)^{\sq}}{(\sq-1)\eta_t^{\sq-1}} dt -( \gamma_t \wedge L) dt +z^L_t dW_t + dm^L_t
\end{equation}
with the bounded terminal condition $y^L_T = \xi \wedge L$. This BSDE has a unique solution $(y^L,z^L,m^L)$ (see \cite{bria:dely:hu:03}). Moreover the solution satisfies the {\it a priori} estimate
\begin{equation}\label{eq:a_priori_estimate_Y_L}
0\leq y^{L}_t \leq \frac{1}{(T-t)^{\kq}} \bE \left[ \ \int_t^{T} \left( \eta_s + (T-s)^{\kq} \gamma_s \right) ds \bigg| \cF_t\right]
\end{equation}
Next, by passing to the limit $L\to \infty$, the minimal super-solution $(y,z,m)$ of \eqref{eq:bsde} with terminal condition \eqref{eq:terminal_cond} is obtained. Let us emphasize here that the left-continuity condition on the filtration is used only to obtain the weak terminal condition \eqref{eq:terminal_cond}.

\begin{Lemma} \label{prop:optimal_control_uncons}
Let $(y^L, z^L,m^L)$ be the solution to \eqref{eq:truncated_bsde} with terminal condition $y^L_T= \xi \wedge L$. Then the process $\cX^L$ satisfying the linear dynamics
$$\cX^L_s=x-\int_t^s\left(\frac{y^L_r}{\eta_r}\right)^{\sq-1}\cX^L_rdr,$$
is optimal in \eqref{eq:unconstraint_value_fct}. Moreover, we have $v^L(t,x) = y^L_t x^{\kq}$.

Let $(y,z,m)$ denote the minimal solution to \eqref{eq:bsde} with singular terminal condition \eqref{eq:terminal_cond}. Then we have $v(t,x)=y_t x^{\kq}$. Moreover, the process $\cX$ satisfying the linear dynamics
$$\cX_s=x-\int_t^s\left(\frac{y_u}{\eta_u}\right)^{\sq-1}\cX_udu,$$
belongs to $\cA^0_{\cS}(t,x)$ and is optimal in \eqref{eq:value_fct}.
\end{Lemma}

\subsection{The formulation under uncertainty without terminal constraint}

We work under the setting described in Section \ref{sect:setting}. We consider a $\cF_T$-Borel measurable random variable $\xi$ such that for any $\bP \in \cP_0$, $\xi$ is a.s. non-negative. We denote by $\cS$ the singular set $\{ \xi = +\infty\}$. We define the two Borel measurable functions
\begin{eqnarray*}
\eta: && (t, \om, a)  \in [0,T] \x \Om \x \Sp  \longrightarrow  \R_+^*, \\
\gamma: && (t, \om, a)  \in [0,T] \x \Om \x \Sp  \longrightarrow  \R_+.
\end{eqnarray*}
Here $\eta$ and $\gamma$ (and thus the generator of our BSDE) do not depend on the drift of $X$. This condition is sufficient to obtain an optimal control independent of the probability measure $\P$ (see Propositions \ref{prop:optim_unconst_pb} and \ref{prop:constr_pb} below). This hypothesis is similar to the setting in \cite{mato:poss:zhou:15}.

We define for simplicity
$$\widehat{\eta}_s:= \eta(s, X_{\cdot\wedge s},\widehat{a}_s) ~~\mbox{and}~~ \widehat{\gamma}_s:= \gamma(s, X_{\cdot\wedge s},\widehat{a}_s).$$
Finally we assume that there exists $\varrho > 1$ such that for any $(t,\omega) \in [0,T]\times \Omega$
\begin{equation}\label{eq:int_cond_eta}
\sup_{\bP \in \mathcal{P}(t,\omega)} \bE^\bP \int_t^T \left(  \frac 1{\widehat{\eta}_s} \right)^{\varrho(\sq-1)} ds <\infty.
\end{equation}
Our generator is 
$$f(t,\omega,y,a) = -  \frac{y |y|^{\sq-1}}{(\sq-1)(\eta(t,\omega,a))^{\sq-1}} + (\gamma(t,\omega,a) \wedge L)$$
and satisfies Condition {\bf (H)}: for any $(t,\omega,a,y,y')$
\begin{itemize}
\item {\bf H1.} $y\mapsto f(t,\omega,y,a)$ is continuous.
\item {\bf H2.} Monotonicity assumption: $f$ is non-increasing w.r.t. $y$
$$(f(t,\omega,y,a)-f(t,\omega,y',a))(y-y') \leq 0.$$
\item {\bf H4.} Growth assumption:
$$ |f(t,\omega,y,a) - f(t,\omega,0,a)| =  \frac{1}{(\sq-1)\eta(t,\omega,a)^{\sq-1}} |y|^\sq$$
together with \eqref{eq:int_cond_eta}.
\end{itemize}
Compared to {\bf H4} in our previous section, here 
$$\Psi(t,\omega,a)=\frac{1}{(\sq-1)\eta(t,\omega,a)^{\sq-1}}$$ 
and Assumption \eqref{eq:int_cond_eta} corresponds to Condition {\bf C1} on $\widehat \Psi$.
The terminal condition $\xi \wedge L$ and the process $\widehat{f}^{0}=(\widehat{f}^{0}_t= \widehat{\gamma}_t \wedge L, \ t\geq 0)$ are bounded. Hence {\bf C1} and {\bf C2} hold for any $\bsp > 1$. Hence Condition \eqref{eq:cond_int} becomes here $1 < p < \varrho$. From Propositions \ref{prop:uniq_2BSDE} and \ref{prop:2BSDE_existence}, we deduce that there exists a unique solution $(Y^L,Z^L,M^{L,\bP},K^{L,\bP})$ to the second order BSDE: for any $0 \leq t \leq T$ and any $\bP$
\begin{eqnarray}\nonumber
Y^L_t & = & (\xi \wedge L) - \int_t^T \frac{|Y^L_u|^{\sq-1}Y^L_u}{(\sq-1)(\widehat{\eta}_u)^{\sq-1}} du + \int_t^T (\widehat{\gamma}_u \wedge L) du \\  \label{eq:L_2BSDE}
&& - \left( \int_t^T Z^L_s dX^{c,\bP}_s\right)^{\bP} -\int_t^T dM^{L,\bP}_s + (K^{L,\bP}_T- K^{L,\bP}_t),\ \ \bP-a.s.,
\end{eqnarray}
such that:
\begin{itemize}
\item For any $p>1$, $Y^L$ belongs to $\bD^{p}_0(\bF^{\cP_0}_+)$.
\item For any $1<p<\varrho$, $(Z^L,M^{L,\bP},K^{L,\bP})$ is in $\bH^{p}_0(\bF^{\cP_0}_+) \times \bM^{p}_0((\bF^{\bP}_+)_{\bP\in \cP_0}) \times \bI^{p}_0((\bF^{\bP}_+)_{\bP\in \cP_0})$.
\item $K^{L,\bP}$ is a $\bP-a.s.$ non-decreasing process satisfying the minimality condition \eqref{eq:minim_cond_K_2BSDE}. 
\end{itemize}
Moreover we have the representation formula
\begin{equation}\label{eq:truncated_repres_formula}
Y^L_t = \underset{\bP' \in \cP(t,\bP,\bF_+)}{\mbox{esssup}^\bP} y^{L,\bP'}_t
\end{equation}
where $(y^{L,\bP},z^{L,\bP},m^{L,\bP})$ is the solution under $\bP$ of the BSDE
\begin{equation*}
dy^{L,\bP}_t = \frac{|y^{L,\bP}_t|^{\sq-1}y^{L,\bP}_t}{(\sq-1)(\widehat{\eta}_t)^{\sq-1}} dt -( \widehat{\gamma}_t \wedge L) dt +z^{L,\bP}_t dX^{c,\bP}_t + dm^{L,\bP}_t.
\end{equation*}
Note that by comparison principle for standard BSDEs (Lemma \ref{lem:comp_sol_gene_BSDE}), these solutions $y^{L,\bP}$ satisfy the inequality: $\bP$-a.s.
$$0 \leq y^{L,\bP}_t  \leq L (T+1), \quad \forall t\in [0,T].$$
Thus $\cP_0$-q.s.
$$0 \leq Y^L_t  \leq L(T+1)\quad \forall t\in [0,T].$$

First we define the following control sets:
\begin{itemize}
\item $\cA(t,x)$ is the set of processes $\cX=(\cX_s,\ 0\leq s \leq T)$ such that $\cX_s = x$ if $s \leq t$ and for any $\bP \in \cP_t$, $\bP-a.s.$, $\cX$ is absolutely continuous, that is: $\cX_s(\omega) = x + \int_t^s \alpha_u(\omega) du$ with $\int_t^T |\alpha_u(\omega)| du < +\infty$.
\item For a fixed $\bP \in \cP_t$, $\cA^\bP(t,x)$ is the set of processes $\cX=(\cX_s,\ 0\leq s \leq T)$ such that $\cX_s = x$ if $s \leq t$ and $\bP-a.s.$, $\cX$ is absolutely continuous, that is: $\cX_s(\omega) = x + \int_t^s \alpha_u(\omega) du$ with $\int_t^T |\alpha_u(\omega)| du < +\infty$.
\end{itemize}
The set $\cA^\bP(t,x)$ depends of $\bP$, whereas $\cA(t,x)$ depends only on the probability set $\cP_t$. Of course $\cA(t,x)$ is included in $\cA^\bP(t,x)$. Next for any $L \geq 0$ we define the following unconstrainted control problems
\begin{equation} \label{eq:UP}
J^L(t,x)= \underset{\cX\in \cA(t,x)}{\essinf}  \underset{\bP \in \mathcal{P}_t}{\esssup} \ \bE^\bP \left[  \int_t^T \left( \widehat{\eta}_s |\alpha_s|^{\kq} + (\widehat{\gamma}_s \wedge L) |\cX_s|^{\kq} \right) ds + (L\wedge \xi) |\cX_T|^{\kq} \bigg| \cF^+_t \right],
\end{equation}
together with 
$$I^L(t,x)= \underset{\bP \in \mathcal{P}_t}{\esssup}  \underset{\cX\in \cA(t,x)}{\essinf} \ \bE^\bP \left[  \int_t^T \left( \widehat{\eta}_s |\alpha_s|^{\kq} + (\widehat{\gamma}_s \wedge L) |\cX_s|^{\kq} \right) ds + (L\wedge \xi) |\cX_T|^{\kq} \bigg| \cF^+_t \right],$$
and
$$H^L(t,x)= \underset{\bP \in \mathcal{P}_t}{\esssup}  \underset{\cX\in \cA^\bP(t,x)}{\essinf} \ \bE^\bP \left[  \int_t^T \left(\widehat{\eta}_s |\alpha_s|^{\kq} + (\widehat{\gamma}_s \wedge L) |\cX_s|^{\kq} \right) ds + (L\wedge \xi) |\cX_T|^{\kq} \bigg| \cF^+_t \right].$$
Immediately $H^L(t,x) \leq I^L(t,x) \leq J^L(t,x)$. From the standard formulation (see Section \ref{sect:standard_form_liquid_pb}) we have
$$H^L(t,x) = x^{\kq} \ \underset{\bP \in \mathcal{P}_t}{\esssup} \ y^{L,\bP}_t= x^{\kq}  \ Y^L_t.$$

%

\begin{Lemma} \label{lem:optim_control}
For any $(t,x)$, $J^L(t,x) \leq H^L(t,x)$.
\end{Lemma}
\begin{proof}
For simplicity, we do not write the constant $L$ in this proof. Let us define 
$$\beta_s = - \left(Y_s/\widehat{\eta}_s\right)^{\sq-1},\qquad d\cX^*_s = \beta_s \cX^*_s ds=\alpha_s ds.$$ 
Let us apply the It\^o formula under the probability $\bP$:
\begin{eqnarray*}
d(Y_s (\cX^*_s)^{\kq}) & = & (\cX^*_{s})^{\kq} dY_s + Y_{s} d( (\cX^*_{s})^{\kq}) \\
& =& (\cX^*_{s})^{\kq}  \frac{Y^\sq_s}{(\sq-1)(\widehat{\eta}_s)^{\sq-1}} ds - (\widehat{\gamma}_s\wedge L) (\cX^*_s)^{\kq} ds + pY_s\beta_s (\cX^*_s)^{p} ds \\
& + &   (\cX^*_{s})^{\kq} Z_s dX^{c,\bP}_s +(\cX^*_{s})^{\kq} dM^\bP_s  -  (\cX^*_{s})^{\kq} dK^\bP_s \\
& = & - (\cX^*_{s})^{\kq}  \left( \frac{Y_s}{\widehat{\eta}_s} \right)^{\sq} \widehat{\eta}_s  ds - (\widehat{\gamma}_s \wedge L)(\cX^*_s)^{\kq} ds+   (\cX^*_{s})^{\kq} Z_s dX^{c,\bP}_s \\
& +& (\cX^*_{s})^{\kq} dM^\bP_s -  (\cX^*_{s})^{\kq} dK^\bP_s\\
& =& - \left(\cX^*_{s}  \frac{Y^{\sq-1}_s}{(\widehat{\eta}_s)^{\sq-1}} \right)^{\kq} \widehat{\eta}_s  ds - (\widehat{\gamma}_s \wedge L)(\cX^*_s)^{\kq} ds+   (\cX^*_{s})^{\kq} Z_s dX^{c,\bP}_s \\
& +& (\cX^*_{s})^{\kq} dM^\bP_s -  (\cX^*_{s})^{\kq} dK^\bP_s\\
& =& - \left[\widehat{\eta}_s \left( |\alpha_s| \right)^{\kq} + (\widehat{\gamma}_s\wedge L) (\cX^*_s)^{\kq} \right] ds+   (\cX^*_{s})^{\kq} Z_s dX^{c,\bP}_s \\
& +& (\cX^*_{s})^{\kq} dM^\bP_s -  (\cX^*_{s})^{\kq} dK^\bP_s. 
\end{eqnarray*} 
since $(\sq-1){\kq}=\sq$. 
Now integrate this from $t$ to $T$:
\begin{eqnarray*}
&& Y_T (\cX^*_T)^{\kq} - Y_t (\cX^*_t)^{\kq}  = (\xi \wedge L) (\cX^*_T)^{\kq} - Y_t x^{\kq}  \\
&&\qquad =  -\int_t^T \left[ \widehat{\eta}_s \left( \alpha_s \right)^{\kq} + (\widehat{\gamma}_s \wedge L) (\cX^*_s)^{\kq} \right] ds + \left( \int_t^T (\cX^*_{s})^{\kq} Z_s dX^{c,\bP}_s \right)^\bP\\
&& \qquad + \left( \int_t^T (\cX^*_{s})^{\kq} dM^\bP_s \right)^\bP- \int_t^T (\cX^*_{s})^{\kq} dK^\bP_s. 
\end{eqnarray*} 
And taking the conditional expectation w.r.t. $\bP$ 
\begin{eqnarray*}
\bE^\bP \left[ (\xi \wedge L) (\cX^*_T)^{\kq}  + \int_t^T \left[\widehat{\eta}_s \left( \alpha_s \right)^{\kq} +( \widehat{\gamma}_s \wedge L)(\cX^*_s)^{\kq} \right] ds  \bigg| \cF^+_t \right]  & =& Y_t x^{\kq} - \bE^\bP \left[ \int_t^T (\cX^*_{s})^{\kq} dK^\bP_s \bigg| \cF^+_t \right] \\
& \leq & Y_t x^{\kq}
\end{eqnarray*} 
since $K^\bP$ is non-decreasing. Therefore 
$$\underset{\bP \in \mathcal{P}_t}{\esssup} \ \bE^\bP \left[ (\xi \wedge L) (\cX^*_T)^{\kq}  + \int_t^T \left[\widehat{\eta}_s \left( \alpha_s \right)^{\kq} + (\widehat{\gamma}_s\wedge L) (\cX^*_s)^{\kq} \right] ds  \bigg| \cF^+_t \right]  \leq Y_t x^{\kq}.$$
Moreover the process $\cX^*$ is in $\cA(t,x)$: 
$$\cX^*_s = x - \int_t^s \left( \frac{Y_u}{\widehat{\eta}_u}\right)^{\sq-1} \cX^*_u du.$$
This implies that 
$$J(t,x) \leq  Y_t x^{\kq} = H(t,x).$$
\end{proof}

Therefore we deduce that $H^L(t,x) \leq I^L(t,x) \leq J^L(t,x)\leq H^L(t,x)$ and our first result:
\begin{Prop} \label{prop:optim_unconst_pb}
The unconstrainted problem \eqref{eq:UP} satisfies 
\begin{eqnarray*}
&& \underset{\cX\in \cA(t,x)}{\essinf}  \underset{\bP \in \mathcal{P}_t}{\esssup} \ \bE^\bP \left[  \int_t^T \left(\widehat{\eta}_s |\alpha_s|^{\kq} + \widehat{\gamma}_s |\cX_s|^{\kq} \right) ds + (\xi \wedge L) |\cX_T|^{\kq} \bigg| \cF^+_t \right] \\
&&  \quad =  \underset{\bP \in \mathcal{P}_t}{\esssup}  \underset{\cX\in \cA^{\bP}(t,x)}{\essinf} \ \bE^\bP \left[  \int_t^T \left( \widehat{\eta}_s |\alpha_s|^{\kq} + \widehat{\gamma}_s |\cX_s|^{\kq} \right) ds + (\xi \wedge L) |\cX_T|^{\kq} \bigg| \cF^+_t \right]
\end{eqnarray*}
and the solution of the 2BSDE \eqref{eq:L_2BSDE}, denoted by $Y^L$, gives the optimal process $\cX^{*,L}$:
$$d\cX^{*,L}_s = \left[ - \left(Y^L_s/\widehat{\eta}_s\right)^{\sq-1}\cX^{*,L}_s \right] ds, \quad (\sq-1)({\kq}-1)=1.$$
\end{Prop}

\subsection{The constrained problem under uncertainty.}  

We denote by $\cA_0(t,x)$ the set of admissible controls $\cX \in \cA(t,x)$ such that $\cX_T\1_\cS=0, \ \mathcal{P}_t$-q.s. ($\mathcal{P}_t$-q.s means  $\bP-a.s.\ \forall \bP\in\mathcal{P}_t$) and  $\cA^\bP_0(t,x)$ the set of admissible controls $\cX \in \cA^\bP(t,x)$ such that $\cX_T\1_\cS=0 \ \bP-a.s.$ Now consider 
\begin{equation}\label{eq:control_constraint}
J(t,x)= \underset{\cX\in \cA_0(t,x)}{\essinf}  \underset{\bP \in \mathcal{P}_t}{\esssup} \ \bE^\bP \left[  \int_t^T \left( \widehat{\eta}_s |\alpha_s|^{\kq} + \widehat{\gamma}_s  |\cX_s|^{\kq} \right) ds + \xi |\cX_T|^{\kq} \bigg| \cF^+_t \right].
\end{equation}
Again we use the convention that $0 \times \infty = 0$. As mentioned for the standard formulation, a left-continuity condition is imposed on the underlying filtration to have the terminal condition \eqref{eq:terminal_cond}\footnote{This technical condition can be avoided if $\xi$ is $\cF_{T-}$-measurable (see for example \cite{bank:voss:16}).}. In our present setting we add the next assumption on our set of probability measures $\cP^W_t$:
\begin{itemize}
\item \textbf{Left-continuity condition:} for any probability measure $\bP \in \cP^W_t$, the filtration $\bF^\bP_+$ is left continuous at time $T$.
\end{itemize}
This hypothesis implies that a martingale cannot have a jump at time $T$. For example this assumption holds if $\cP = \{\bP^{a}, \ a \in \mathbb{A}\}$ is the set of all probability measures $\bP^a$ given by:
$$\bP^a = \bP_0\circ(X^a)^{-1},\quad X^a_t = \int_0^t a^{\frac{1}{2}}_s dX_s $$
for all processes $a \in \mathbb{A}$ of the form
$$a = \sum_{n=0}^{\infty}\sum_{i=1}^\infty a^n_i \1_{E^n_i} \1_{[\tau_n,\tau_{n+1})},$$
where $(a^n_i)_{i,n} \in \mathbb{A}_0$, $(\tau_n)_n$ is a nondecreasing sequence of stopping times with $\tau_0=0$ and
\begin{itemize}
\item $\inf\{n, \ \tau_n=+\infty\} < +\infty$, $\tau_n < \tau_{n+1}$ whenever $\tau_n < +\infty$ and each $\tau_n$ takes at most countably many values,
\item for each $n$, $\{E^n_i, \ i\geq 1\} \subset \cF_{\tau_n}$ forms a partition of $\Omega$.
\end{itemize}
$\mathbb{A}_0$ is the class of all deterministic mappings such that $0 < \underline{a} \leq a_t$ for any $t \geq 0$ (see \cite{sone:touz:zhan:11b}, Section 4.4). Every $\bP^a$, $a \in \mathbb{A}_0$, satisfies the martingale representation property. Then from \cite[Proposition A.1]{kazi:poss:zhou:15}, every $\bP \in \cP$ verifies this property too. Thereby any martingale is continuous, which implies the required argument for the filtration (see \cite[Proposition 25.19]{kall:02} for example).

\vspace{0.5cm}
For $L\leq L'$ and any $\bP \in \cP_0$, we have $\bP-a.s.$ for any $t\in [0,T]$
$$y^{L,\bP}_t \leq y^{L',\bP}_t \leq Y^{L'}_t.$$
Hence $\cP_0$-q.s., $Y^{L}_t \leq Y^{L'}_t$ for $t\in [0,T]$ (see also the comparison result \cite[Theorem 4.3]{poss:tan:zhou:15}). 

Let us now assume that there exists $\ell > 1$ and $\kappa \in (1,\ell)$ such that for any $(t,\omega)$
\begin{equation}\label{eq:int_cond_eta_gamma_sing_1}
\sup_{\bP \in \mathcal{P}(t,\omega)} \bE^\bP \left[ \int_t^T \left[ \widehat{\eta}_s + (T-s)^{\kq}\ \widehat{\gamma}_s \right]^\ell ds\right]  <\infty,
\end{equation}
and
\begin{equation}\label{eq:int_cond_eta_gamma_sing_2}
\sup_{\bP \in \mathcal{P}_0} \bE^\bP \left[ \underset{0\leq t \leq T}{\mbox{ess sup}}^\bP \left( \bE^\bP \left[ \int_0^T \left[ \widehat{\eta}_s + (T-s)^{\kq}\ \widehat{\gamma}_s \right]^\kappa ds \bigg| \cF_t^+ \right]\right)^{\frac{\ell}{\kappa}} \right]  <\infty.
\end{equation}

\begin{Lemma}[{\it A priori} estimate]
There exists $U\in \bD^\ell_0(\bF^{\cP_0}_+)$ such that for any $0\leq t \leq T$, $\cP_0$-q.s.
\begin{equation}  \label{eq:sing_apriori_estim}
0\leq Y^{L}_t \leq \frac{1}{(T-t)^{\kq}} U_t.
\end{equation}
Let us emphasize that the right-hand side does not depend on $L$ and is finite on $[0,T)$. 
\end{Lemma}

\begin{proof}
The estimate \eqref{eq:a_priori_estimate_Y_L} gives for any $\bP \in \cP_0$ 
$$0\leq y^{L,\bP}_t \leq  \frac{1}{(T-t)^{\kq}} \bE^\bP \left[ \ \int_t^{T} \left( \widehat{\eta}_s + (T-s)^{\kq} \widehat{\gamma}_s \right) ds \bigg| \cF^+_t\right] =   \frac{1}{(T-t)^{\kq}} \ u^{\bP}_t .$$
The process $(u^\bP,v^\bP,n^\bP)$ is the solution of the BSDE
$$u^\bP_t = \int_t^{T} \left( \widehat{\eta}_s + (T-s)^{\kq} \widehat{\gamma}_s \right) ds -\left(  \int_t^T v^\bP_s dX^{c,\bP}_s\right)^\bP - \int_t^T dn^\bP_s.$$
Then using \eqref{eq:int_cond_eta_gamma_sing_1}, \eqref{eq:int_cond_eta_gamma_sing_2} and \cite[Theorem 4.1]{poss:tan:zhou:15}, there exists a unique solution $(U,V,\mathcal N^\bP, \mathcal{K}^\bP)$ to the 2BSDE:
$$U_t = \int_t^{T} \left( \widehat{\eta}_s + (T-s)^{\kq} \widehat{\gamma}_s \right) ds  - \left( \int_t^T V_s dX^{c,\bP}_s\right)^\bP - \int_t^T d\mathcal N^\bP_s+ (\mathcal{K}^\bP_T - \mathcal{K}^\bP_t),$$
such that 
$U \in \bD^\ell_0(\bF^{\cP_0}_+)$ and $(V,\mathcal N^\bP, \mathcal{K}^\bP)$ is in $\bH^\ell_0(\bF^{\cP_0}_+) \times \bM^\ell_0((\bF^{\bP}_+)_{\bP\in \cP_0}) \times \bI^\ell_0((\bF^{\bP}_+)_{\bP\in \cP_0})$. Moreover for any $\bP \in \cP_0$ and any $t \in [0,T]$, we have the representation formula:
$$\underset{\bP' \in \cP(t,\bP,\bF_+)}{\mbox{esssup}^\bP} u^{\bP'}_t = U_t,\quad \bP-a.s.$$ 
Thus we obtain the desired {\it a priori} estimate since
$$Y^L_t =\underset{\bP' \in \cP(t,\bP,\bF_+)}{\mbox{esssup}^\bP} y^{L,\bP'}_t \leq \underset{\bP' \in \cP(t,\bP,\bF_+)}{\mbox{esssup}^\bP} u^{\bP'}_t = U_t.$$
This achieves the proof of the Lemma.
\end{proof}

\begin{Lemma}
For any $\eps > 0$, the sequence $(Y^L,Z^L,M^{L,\bP},K^{L,\bP})$ converges, when $L$ goes to $+\infty$, to $(Y,Z,M^\bP,K^\bP)$ in the space $\bD_0^\ell(\bF^{\cP_0}_+)\times \bH^\ell_0(\bF^{\cP_0}_+) \times \bM^\ell_0((\bF^{\bP}_+)_{\bP\in \cP_0}) \times \bI^\ell_0((\bF^{\bP}_+)_{\bP\in \cP_0})$ on $[0,T-\eps]$, which means that all processes are restricted on this time interval. Moreover $(Y,Z,M^\bP,K^\bP)$ satisfies the dynamics: for any $\bP \in \cP_0$, and any $0\leq s \leq t < T$:
\begin{equation} \label{eq:2BSDE}
Y_s = Y_t - \int_s^t  \frac{Y^\sq_u}{(\sq-1)(\widehat{\eta}_u)^{\sq-1}} du + \int_s^t \widehat{\gamma}_u du - \left( \int_s^t Z_u dX^{c,\bP}_u\right)^\bP -\int_s^t dM^{\bP}_u + K^\bP_t - K^\bP_s.
\end{equation}
Finally $Y$ satisfies the representation property: for any $t < T$ and any $\bP \in \cP_0$, 
$$Y_t = \underset{\bP' \in \cP(t,\bP,\bF_+)}{\mbox{{\rm esssup}}^\bP} y^{\bP'}_t, \quad \bP-a.s.$$
\end{Lemma}
\begin{proof}
Fix $\eps > 0$ and define
\begin{eqnarray*}
\psi^{\ell,\eps}_{L,L'} & := & \sup_{\bP\in \cP_0} \bE^{\bP} \left[  \int_0^{T-\eps} |(\widehat{\gamma}_s  \wedge L) - (\widehat{\gamma}_s  \wedge L') |^\ell ds \right]\\
\phi^{\ell,\kappa,\eps}_{L,L'} & := & \sup_{\bP\in \cP_0} \bE^{\bP} \left[ \underset{0\leq t \leq T-\eps}{\esssup^\bP} \ \bE^\bP \left[ \left( \int_0^{T-\eps} |(\widehat{\gamma}^\P_s  \wedge L) - (\widehat{\gamma}^\P_s  \wedge L') |^\kappa ds \right)^{\frac{\ell}{\kappa}} \bigg| \cF^+_t \right]\right].
\end{eqnarray*}
From our conditions \eqref{eq:int_cond_eta_gamma_sing_1} and \eqref{eq:int_cond_eta_gamma_sing_2} on $\widehat \gamma$, $ \psi^{\ell,\eps}_{L,L'} $ and $\phi^{\ell,\kappa,\eps}_{L,L'}$ tend to zero when $L$ and $L'$ go to $+\infty$. From \cite[Theorem 4.5]{poss:tan:zhou:15} (stability result for 2BSDE), for any $L$ and $L'$ we have on $[0,T-\eps]$
$$\|Y^L - Y^{L'}\|^\ell_{\bD_0^\ell} \leq C\left[ \|Y^L_{T-\eps} - Y^{L'}_{T-\eps} \|^\ell_{\bL_0^\ell} + \psi^{\ell,\eps}_{L,L'} \right].$$
From the uniform w.r.t. $L$ and $\cP_0$-q.s. bound \eqref{eq:sing_apriori_estim}, and from the monotonicity of the sequence $Y^L$, we deduce that 
$$\|Y^L_{T-\eps} - Y^{L'}_{T-\eps} \|^\ell_{\bL_0^\ell}$$
tends to zero when $L$ and $L'$ go to $+\infty$. Hence there exists $Y \in \bD_0^\ell(\bF^{\cP_0}_+)$ defined on $[0,T-\eps]$ such that on $[0,T-\eps]$, $Y^L$ converges strongly to  $Y$. From \eqref{eq:sing_apriori_estim}, we have: $\cP_0$-q.s. for $t \in [0,T)$
$$0\leq Y_t \leq \frac{1}{(T-t)^{\kq}} U_t.$$

By the representation of $Y^L$, for any $t \in [0,T)$, any $\bP\in \cP_0$ and any $\bP' \in \cP(t,\bP,\bF_+)$
$$ y^{L,\bP'}_t \leq \underset{\bP' \in \cP(t,\bP,\bF_+)}{\mbox{esssup}^\bP} y^{L,\bP'}_t  = Y^{L}_t \leq Y_t, \ \bP-a.s.$$
The (minimal) supersolution $y^{\bP'}$ of the singular BSDE \eqref{eq:bsde} is obtained as the increasing limit of $y^{L,\bP'}$. Thus for any $\bP'$
$$ y^{L,\bP'}_t \leq y^{\bP'}_t \leq Y_t \Rightarrow \underset{\bP' \in \cP(t,\bP,\bF_+)}{\mbox{esssup}^\bP} y^{\bP'}_t \leq Y_t , \ \bP-a.s.$$
Moreover for any $L$
$$Y^{L}_t  = \underset{\bP' \in \cP(t,\bP,\bF_+)}{\mbox{esssup}^\bP}  y^{L,\bP'}_t \leq \underset{\bP' \in \cP(t,\bP,\bF_+)}{\mbox{esssup}^\bP}  y^{\bP'}_t \Rightarrow  Y_t \leq \underset{\bP' \in \cP(t,\bP,\bF_+)}{\mbox{esssup}^\bP}  y^{\bP'}_t.$$
We deduce the representation formula for $Y$. 

Now from the stability property for 2BSDE (\cite[Theorem 4.5]{poss:tan:zhou:15}), if $N^{L,\bP} = M^{L,\bP} -K^{L,\bP}$, then on $[0,T-\eps]$
\begin{eqnarray*}
&& \|Z^L - Z^{L'}\|^\ell_{\bH_0^\ell} + \sup_{\bP \in \cP_0} \bE^\bP \left[ [N^{L,\bP}-N^{L',\bP}]_{T-\eps}^{\frac{\ell}{2}} \right] \\
&& \quad \leq C\left[ \|Y^L_{T-\eps} - Y^{L'}_{T-\eps} \|^\ell_{\bL_0^\ell} + \|Y^L_{T-\eps} - Y^{L'}_{T-\eps} \|^{\frac{\ell}{2}\wedge (\ell-1)}_{\bL_0^\ell} + \phi^{\ell,\kappa,\eps}_{L,L'} + (\phi^{\ell,\kappa,\eps}_{L,L'})^{\frac{\ell}{2}\wedge (\ell-1)} \right].
\end{eqnarray*}
Thereby the sequences $Z^L$ and $N^{L,\bP}$ have a limit $Z$ and $N^\bP$. The process $(Y,Z,N^\bP)$ satisfies the dynamics: for any $\bP \in \cP_0$, and any $0\leq s \leq t < T$:
\begin{equation*} 
Y_s = Y_t - \int_s^t  \frac{Y^\sq_u}{(\sq-1)(\widehat{\eta}_u)^{\sq-1}} du + \int_s^t \widehat{\gamma}_u du - \left( \int_s^t Z_u dX^{c,\bP}_u\right)^\bP -\int_s^t dN^{\bP}_u.
\end{equation*}
Then we decompose the process $N^\bP_t = M^\bP_t - K^\bP_t$ and we check that $(Z,M^\bP,K^\bP)$ is the desired space and that \eqref{eq:2BSDE} holds. 
\end{proof}

Now we come to the main result concerning singular 2BSDEs. 

\begin{Prop} \label{prop:sing_2BSDE}
Under Conditions \eqref{eq:int_cond_eta_gamma_sing_1} and \eqref{eq:int_cond_eta_gamma_sing_2}, the 2BSDE 
\begin{eqnarray*}\nonumber
Y_t & = & \xi - \int_t^T \frac{(Y_u)^\sq}{(\sq-1)(\widehat{\eta}_u)^{\sq-1}} du + \int_t^T (\widehat{\gamma}_u ) du \\ 
&& -\left(  \int_t^T Z_s dX^{c,\bP}_s \right)^{\bP}-\int_t^T dM^{\bP}_s + (K^{\bP}_T- K^{\bP}_t),\ \ \bP-a.s.
\end{eqnarray*}
for any $0 \leq t \leq T$ and any $\bP\in \cP_0$, with the singular terminal condition $\xi$, admits a non-negative super-solution $(Y,Z,M^\bP,K^\bP)$ satisfying:
\begin{itemize}
\item the dynamics \eqref{eq:2BSDE} for any $0\leq s \leq t < T$ ;
\item the integrability property for any $\eps > 0$:
$$\|Y\|^\ell_{\bD_0^\ell(0,T-\eps)} + \|Z\|^\ell_{\bH_0^\ell(0,T-\eps)} + \sup_{\bP \in \cP_0} \bE^\bP \left[ [M^{\bP}]_{T-\eps}^{\frac{\ell}{2}} \right] +\sup_{\bP\in \cP_0} \bE^\bP \left( K^\bP_{T-\eps} \right)^\ell< +\infty \ ;$$
\item the minimality condition \eqref{eq:minim_cond_K_2BSDE_eps} ;
\item the weak terminal condition: $\cP_0$-q.s.
\begin{equation}\label{eq:super_terminal_cond}
\liminf_{s\to T} Y_s \geq \xi.
\end{equation}
\end{itemize}
Moreover this solution is the non-negative minimal solution, that is if $(\overline{Y},\overline{Z},\overline{M}^\bP,\overline{K}^\bP)$ satisfies the four previous conditions together with $\overline{Y}_t\geq 0$ for any $t \in [0,T]$ $\cP_0$-q.s., then $\overline{Y}_t \geq Y_t$ for any $t \in [0,T]$ $\cP_0$-q.s.
\end{Prop}
\begin{proof}
The first two points are direct consequences of the previous lemma. 
Since $Y$ is the essential supremum of the super-solutions $y^\bP$, following the same arguments as in the proof of Proposition \ref{prop:2BSDE_existence}, we deduce that $K^\bP$ satisfies the minimality condition: for any $\bP \in \cP_0$
\begin{equation} \label{eq:minim_cond_K_2BSDE_eps}
\underset{\bP' \in \cP(t,\bP,\bF_+)}{\essinf^\bP} \bE^{\bP'} \left[ \int_t^{T-\eps} \exp \left( \int_{t}^s \lambda_u^{\bP'} du \right)dK^{\bP'}_s \bigg|  \cF_t^+ \right] = 0, \quad 0 \leq t \leq T-\eps, \ \bP-a.s.,
\end{equation}
where
$$\lambda_u^{\bP'} = -\frac{1}{(\sq-1)(\widehat{\eta}_u)^{\sq-1}}\frac{Y_u^\sq - (y^{\bP'}_u)^\sq}{Y_u-y^{\bP'}_u} \1_{Y_u\neq y^{\bP'}_u} .$$
Moreover our left-continuity condition on the filtration implies that for any $\bP \in \cP_0$ 
$$\liminf_{s\to T} y^\bP_s \geq \xi,\quad \bP-a.s.$$
Hence from the representation formula, the same inequality holds for $Y$.

Let us prove minimality of this supersolution. Let us consider $(\overline{Y},\overline{Z},\overline{M}^\bP,\overline{K}^\bP)$ satisfying the dynamics \eqref{eq:2BSDE} for any $0\leq s \leq t < T$, the minimality condition \eqref{eq:minim_cond_K_2BSDE_eps} and the weak terminal condition \eqref{eq:super_terminal_cond}. We also assume that $\overline{Y}$ is $\cP_0-q.s.$ non-negative and $(\overline{Y},\overline{Z},\overline{M}^\bP,\overline{K}^\bP)$ verifies some integrability property similar to $(Y,Z,M^\bP,K^\bP)$ (with maybe a different power $\overline{\ell}$). From the proof of Proposition \ref{prop:uniq_2BSDE} (uniqueness for 2BSDE), we deduce that for any $\eps > 0$, the representation property holds on $[0,T-\eps]$, that is for any $t \in [0,T-\eps]$ and $\bP \in \cP_0$
\begin{equation}\label{eq:representation_formula_sing_2BSDE}
\overline{Y}_{t} = \underset{\bP' \in \cP(t,\bP,\bF_+)}{\esssup^\bP} \bar y_{t}^{\bP'}(T-\eps,\overline{Y}_{T-\eps}), \bP-a.s.
\end{equation}
where $\bar y^{\bP'}=\bar y^{\bP'}(T-\eps,\overline{Y}_{T-\eps})$ is the first part of the solution $(\bar y^{\bP'},\bar z^{\bP'},\bar m^{\bP'})$ of the BSDE \eqref{eq:bsde} on $[0,T-\eps]$ with terminal condition $\overline{Y}_{T-\eps}$ at time $T-\eps$ under $\bP'$.

Recall that $Y^L$ satisfies \eqref{eq:L_2BSDE} and \eqref{eq:truncated_repres_formula}. Fix $L$ and any probability $\bP$ and consider $\bar y^{\bP}$ and $y^{L,\bP}$ on the time interval $[0,T-\eps]$. Set
 $$\widetilde y = \bar y^{\bP} - y^{L,\bP}, \quad \widetilde z = \bar z^{\bP} - z^{L,\bP}, \quad \widetilde m = \bar m^{\bP} - m^{L,\bP}.$$
We have
\begin{eqnarray*}
f(t,\bar y^{\bP}_t) - f(t,y^{L,\bP}) & = &- \frac{1}{(\sq-1)(\widehat{\eta}_t)^{\sq-1}} \left(\frac{(\bar y^{\bP}_t)^{\sq} - (y^{L,\bP}_t)^{\sq}}{\bar y^{\bP}_t - y^{L,\bP}_t } \right)  \1_{\bar y^{\bP}_t\neq y^{L,\bP}_t} \left( \bar y^{\bP}_t - y^{L,\bP}_t \right)\\
& + &  \widehat{\gamma}_t -( \widehat{\gamma}_t \wedge L)  \\
& = & \lambda^{\bP}_t \left( \bar y^{\bP}_t - y^{L,\bP}_t \right) + \widehat{\gamma}_t -( \widehat{\gamma}_t \wedge L) 
\end{eqnarray*}
with $ \lambda^{\bP}_t\leq 0$. Thus the process $(\widetilde y, \widetilde z, \widetilde m)$ solves the BSDE
\begin{eqnarray*}
d\widetilde y_t & = & \left[ - \lambda^\bP_t \widetilde y_t - \widehat{\gamma}_t   \1_{\widehat{\gamma}_t \geq L} \right] dt 
+ \widetilde z_t + d\widetilde m_t
\end{eqnarray*}
on $[0, T-\eps]$ with terminal condition $\widetilde y_{T-\eps} = \overline{Y}_{T-\eps} - y^{L,\bP}_{T-\eps}$. Thereby
\begin{eqnarray*}
\widetilde y_s & =  &  \E^{\bP} \left[\widetilde y_{T-\eps} \Gamma_{s,T-\eps}   + \int_s^{T-\eps} \Gamma_{s,u} \widehat{\gamma}_u   \1_{\widehat{\gamma}_u \geq L} du \bigg| \cF^{\bP,+}_s \right]  \geq   \E^{\bP} \left[\widetilde y_{T-\eps} \Gamma_{s,T-\eps}  \bigg| \cF^{\bP,+}_s \right]
\end{eqnarray*}
where $\Gamma_{s,t} = \exp\left( \int_s^t \lambda^{\bP}_u du \right)$. Note that we have $y^{L,\bP}_t \leq (1+T)L$ and hence $\widetilde y_t \geq -(1+T)L$. Thus $\widetilde y  \Gamma_{s,.}$ is bounded from below. We can apply Fatou's lemma to obtain
\begin{eqnarray*}
\widetilde y_s & = & \liminf_{\eps \downarrow 0} \E^{\bP} \left[\widetilde y_{T-\eps}  \Gamma_{s,T-\eps} \bigg| \cF^{\bP,+}_s \right]  \geq  \E^{\bP} \left[  \liminf_{\eps \downarrow 0} (\widetilde y_{T-\eps}  \Gamma_{s,T-\eps})  \bigg|\cF^{\bP,+}_s \right].
\end{eqnarray*}
The process $(\Gamma_{s,t}, \ s\leq t\leq T)$ is c\`adl\`ag, non-negative and bounded by one. Hence a.s.
$$ \liminf_{\eps \downarrow 0} (\widetilde y_{T-\eps} \Gamma_{s,T-\eps}) = (\liminf_{\eps \downarrow 0} \widetilde y_{T-\eps}) \Gamma_{s,T^-} \geq (\xi - \xi\wedge L) \Gamma_{s,T^-}\geq 0.$$
Here again we use left-continuity of the filtration to exclude jumps for the orthogonal martingale $m^{L,\bP}$. Finally, $\bar y^{\bP}_s \geq y^{L,\bP}_s$ for any $s \in [0,T)$. Since it holds for any $\bP\in \cP_0$, we deduce that $\overline{Y}_s \geq Y^{L}_s$ for any $L \geq 0$. Taking the limit as $L$ goes to $\infty$ yields the claim.
\end{proof}

\begin{Remark}\label{rem:sing_gen_2BSDE}
The previous proposition holds true for more general generators satisfying Conditions {\bf (H)} together with the growth condition:
there exists a constant $\sq > 1$ and a positive process $\eta$ such that for any $y \geq 0$
\begin{equation*}
\widehat f^\bP(t,y,z)\leq -\frac{1}{(\sq-1)\widehat \eta^{\sq-1}_t}|y|^{\sq} + \widehat f^\bP(t,0,z)
\end{equation*}
and the conditions \eqref{eq:int_cond_eta_gamma_sing_1} and \eqref{eq:int_cond_eta_gamma_sing_2} hold replacing $\widehat \gamma$ by $\widehat f^{0,\bP}$.
See \cite{krus:popi:15} for details for singular BSDEs. 
\end{Remark}

We can now obtain an optimal solution for the control problem \eqref{eq:control_constraint}.
\begin{Prop} \label{prop:constr_pb}
The constrainted problem \eqref{eq:control_constraint} has an optimal state process $\cX^*$ defined by
$$\cX^*_s = x - \int_t^s \left(\frac{Y_u}{\widehat{\eta}_u}\right)^{\sq-1} \cX^*_u du,$$
Moreover the value function is given by: $J(t,x) = |x|^{\kq} Y_t$.
\end{Prop}
\begin{proof}
If we define for $t \leq s < T$
$$\theta_s := \left[Y_s (\cX^*_s)^{{\kq}-1} - Y_t (\cX^*_t)^{{\kq}-1} + \int_t^s (\cX^*_u)^{{\kq}-1} \widehat{\gamma}_u du + \int_t^s (\cX^*_u)^{{\kq}-1} dK^\bP_u \right],$$
then we can easily show that under each $\bP \in \cP_0$, $\theta$ is a non-negative local martingale, and thus a non-negative supermartingale. Thereby $\theta_s$ has a limit, $\bP-a.s.$ when $s$ goes to $T$. Hence since $\bP-a.s.$
$$\liminf_{s\to T} Y_s \1_\cS= +\infty,$$
we obtain that
$$\cX^*_s = \left( \frac{\theta_s + Y_t x^{{\kq}-1}- \int_t^s (\cX^*_u)^{{\kq}-1}( \widehat{\gamma}_u du +dK^\bP_u) }{Y_s}\right)^{\sq-1} \leq\left( \frac{\theta_s + Y_t x^{{\kq}-1}}{Y_s}\right)^{\sq-1} $$
tends to zero on $\cS$, $\bP-a.s.$ In other words $\cX^*  \in \cA_0(t,x)$. 

As in Lemma \ref{lem:optim_control}, we have
\begin{eqnarray*}
d(Y_s (\cX^*_s)^{\kq}) & =& - \left[\widehat{\eta}_s \left( |\alpha_s| \right)^{\kq} + (\widehat{\gamma}_s) (\cX^*_s)^{\kq} \right] ds+   (\cX^*_{s})^{\kq} Z_s dX^{c,\bP}_s \\
& +& (\cX^*_{s})^{\kq} dM^\bP_s -  (\cX^*_{s})^{\kq} dK^\bP_s. 
\end{eqnarray*} 
Thus for any $\eps > 0$
\begin{eqnarray*}
Y_t x^{\kq}  & = &\bE^\bP \left[  Y_{T-\eps} (\cX^*_{T-\eps})^{\kq} + \int_t^{T-\eps} \left[ \widehat{\eta}_s \left( \alpha_s \right)^{\kq} + (\widehat{\gamma}_s ) (\cX^*_s)^{\kq} \right] ds \bigg| \cF_t^+\right]  \\
&& \qquad +\bE^\bP \left[ \int_t^{T-\eps}  (\cX^*_{s})^{\kq} dK^\bP_s \bigg| \cF_t^+\right] . 
\end{eqnarray*} 
From the definition of $\theta$, it follows that also the limit $\lim_{t \uparrow T}Y_{t} |\cX^*_{t}|^{{\kq}-1}\in \R$ exists and that $\cX^*_T=\lim_{t \uparrow T}|\cX^*_{t}|=0$ if $\liminf_{t \uparrow T}Y_{t}=\infty$. Recall that $\liminf_{t\uparrow  T} Y_s \geq \xi$ and let us distinguish two cases. First assume that $\liminf_{t \uparrow  T}Y_{t}=\infty$. Then $\liminf_{t \uparrow T}Y_{t} |\cX^*_{t}|^{\kq}=(\lim_{t\uparrow T}Y_{t} |\cX^*_{t}|^{{\kq}-1})(\lim_{t\uparrow T}|\cX^*_{t}|)=0=\xi |\cX^*_T|^{\kq}$ (for the last equality we use that $\infty\cdot 0:=0$). Next assume that $\liminf_{t \uparrow  T}Y_{t}<\infty$. Then it follows that $\liminf_{t \uparrow T}Y_{t} |\cX^*_{t}|^{\kq} \ge \xi |\cX^*_T|^{\kq}$. Hence for any $\bP$, $\liminf_{t \uparrow  T}Y_{t} |\cX^*_{t}|^{\kq} \ge \xi |\cX^*_{T}|^{\kq}$, $\bP-a.s.$ By Fatou's lemma and since $K^\bP$ is non-decreasing, we have
\begin{eqnarray*}
Y_t x^{\kq}  & \geq &\bE^\bP \left[  \xi (\cX^*_{T})^{\kq} + \int_t^{T} \left[ \widehat{\eta}_s \left( \alpha_s \right)^{\kq} + (\widehat{\gamma}_s ) (\cX^*_s)^{\kq} \right] ds \bigg| \cF_t^+\right] . 
\end{eqnarray*} 
Thereby
$$Y_t x^{\kq}  \geq  \underset{\bP \in \mathcal{P}_t}{\esssup} \ \bE^\bP \left[  \int_t^T \left( \widehat{\eta}_s |\alpha_s|^{\kq} + \widehat{\gamma}_s  |\cX_s|^{\kq} \right) ds + \xi |\cX_T|^{\kq} \bigg| \cF^+_t \right] \geq J(t,x).$$
Now we have obviously $J(t,x) \geq J^L(t,x)$ and: 
$$|x|^{\kq} Y_t = \lim_{L\nearrow+\infty} |x|^{\kq} Y^L_t =  \lim_{L\nearrow+\infty} J^L(t,x) \leq J(t,x) \leq Y_t x^{\kq}.$$
This gives the optimality for $\cX^*$ and the value function of our problem.
\end{proof}

\subsubsection*{Discussion around the examples of \cite{anki:jean:krus:13}}

Even in the classical case, that is for a fixed probability $\bP$, there is in general no explicit solution of the BSDE \eqref{eq:bsde}. But when $\xi = +\infty$ $\bP-a.s.$ and $\gamma=0$, then in \cite{anki:jean:krus:13}, Section 5, an explicit solution $y$ is given provided that $\eta$ has uncorrelated multiplicative increments. This condition is equivalent to the property that the process $(\eta_t/\bE^\bP(\eta_t), \ t\geq 0)$ is a $\bP$-martingale (see \cite[Lemma 5.1]{anki:jean:krus:13}). Under this condition, the value function and an optimal state process are given in \cite[Proposition 5.3]{anki:jean:krus:13}.

Assume that $\eta$ is given by: $\eta_t = \eta_0 \exp \left( X_t - \dfrac{1}{2} \int_0^t \widehat a_s ds \right)$ and that under $\bP$, the drift $b^\bP$ of $X$ is deterministic. Note that the contatenation property of the family $\cP$ implies that $b=b^\bP$ should not depend on $\bP$. Then under each $\bP$, $\eta$ satisfies $d\eta_t = \eta_t dX_t$ and has uncorrelated multiplicative increments. The solution $y^\bP$ is explicitely given by:
$$ y^\bP_t = \eta_t  \frac{1}{\left( A_t \int_t^T \frac{1}{A_s} ds \right)^{{\kq}-1}}, \quad A_t = \exp \left(\int_0^s (\sq-1)b_r dr \right)  = [\bE^\bP (\eta_s)]^{\sq-1}.$$
Hence we have an explicit formula for the solution of the 2BSDE 
$$Y_t = \eta_t  \frac{1}{\left( A_t \int_t^T \frac{1}{A_s} ds \right)^{{\kq}-1}}.$$ 
An optimal state process is deterministic: 
$$\cX_t =\frac{1}{\int_0^T \frac{1}{A_s} ds} \int_t^T \frac{1}{A_s} ds.$$
In particular if the canonical process is a local martingale under each $\bP$, $Y = \eta$ and $\cX_t = (T-t)/T$. Roughly speaking, since $\eta$ models the cost (price impact) and $\gamma$ the risk, then the drift is important for $\eta$ (average cost) and the volatility is important for $\gamma$. That's why when taking $\gamma$ is equal to $0$, the volatility uncertainty can not be seen in the generator.

\subsection*{Acknowledgements.}
The authors thank Nicole El Karoui and Shiqi Song for the fruitful remarks concerning the left-continuity property of the filtration. 
The authors are grateful to Dylan Possama\"i for the discussion on the minimality condition.

\section{Appendix: (Reflected) BSDE with monotone generator}\label{sect:recall_monotone_BSDE}

\setcounter{Lemma}{0}
\renewcommand{\theLemma}{A.\arabic{Lemma}}

\setcounter{Prop}{0}
\renewcommand{\theProp}{A.\arabic{Prop}}

\subsection{Notations and Conditions \textbf{(H)} and \textbf{(H')}} \label{ssect:notations}

In this section the setting is the same as in \cite{bouc:poss:zhou:15} or \cite{krus:popi:14}. Let $T  > 0 $ be fixed and let $(\Omega,\cF,\bP)$ be a probability space, equipped with a filtration $\bF=\{ \cF_t, \ 0\leq t\leq T\}$ satisfying the usual conditions and carrying a standard $d$-dimensional $\bF$-Brownian motion $W$. For any $p> 1$ and any $\alpha \in \bR$ we introduce the following spaces:
\begin{itemize}
\item $\bL^{p}$ denotes the space of all $\cF_T-$measurable scalar random variable $\xi$ such that
$$\|\xi\|_{\bL^{p}}^p = \bE \left[|\xi|^p\right]<+\infty.$$
\item $\bD^{p}$ is the space of all $\bR$-valued, $\bF$-adapted processes $Y$ with $\P$-a.s. c\`adl\`ag paths on $[0,T]$, such that
$$\No{Y}_{\mathbb D^{p}}^p = \bE \left[\sup_{0\leq s\leq T} |Y_s|^p\right]<+\infty.$$
\item $\mathbb H^{p,\alpha}$ denotes the space of all $\bF$- predictable $\mathbb R^d-$valued processes $Z$ such that
$$\No{Z}_{\mathbb H^{p,\alpha}}^p = \mathbb E\left[\left(\int_0^T e^{\alpha s} \No{Z_s}^2ds\right)^{\frac p2}\right]<+\infty.$$
\item $\mathbb M^{p,\alpha}$ is the space of all $\bF$-optional martingales $M$ with $\P$-a.s. c\`adl\`ag paths on $[0,T]$, with $M_0=0$, $\P-a.s.$, such that $M$ is orthogonal to $W$ and 
$$\No{M}_{\mathbb M^{p,\alpha}}^p=\E\left[ \int_0^T e^{\alpha  s} d\left[M\right]_s\right]^{\frac p2}<+\infty.$$
\item $\mathbb I^{p,\alpha}$ (resp. $\mathbb I^{o,p,\alpha}$) denotes the space of all $\bF$-predictable (resp. $\bF$-optional) processes $K$ with $\bP$-a.s. c\`adl\`ag and non-decreasing paths on $[0,T]$, with $K_0=0$, $\P$-a.s., and
$$\No{K}_{\mathbb I^{p,\alpha}}^p=\bE \left[ \left( \int_0^T e^{\alpha s/2} dK_s\right)^{p}\right]<+\infty.$$
\end{itemize}
The spaces $\mathbb H^{p,0}$, $\mathbb M^{p,0}$ and $\mathbb I^{p,0}$ will be denoted $\mathbb H^{p}$, $\mathbb M^{p}$ and $\mathbb I^{p}$. In the It\^o formula  for $p> 1$ we will use the constant
\begin{equation} \label{def:c(p)}
c(p) =\frac{p}{2}((p-1)\wedge 1)
\end{equation}
and the function $\phi_p(x) = |x|^{p-1}\mbox{sgn}(x)\1_{x\neq 0}$ (see \cite[Corollary 2.3]{bria:dely:hu:03} or \cite[Corollary 1]{krus:popi:14}). 

We consider a generator function 
\begin{equation*}
f: (t, \om, y, z)  \in [0,T] \x \Om \x \R \x \R^d   \longrightarrow  \R.
\end{equation*}
The generator $f$ satisfies Condition {\bf (H)}. In {\bf H4}, the process $\Psi$ depends on $t$ and $\omega$ and is supposed to satisfy:
$$\bE \int_0^T (\Psi_t)^\varrho dt < +\infty,$$
(see Condition {\bf C1}). Sometimes we will use the stronger condition {\bf H4'}: there exists $\sq \geq 1$ and a constant $L_\sq$ such that 
$$|f(t,\omega,y,0) -f^{0}_t| \leq L_\sq(1+ |y|^\sq).$$
If $f$ verifies {\bf H1}, {\bf H2}, {\bf H3} and {\bf H4'}, we will say that {\bf (H')} holds.

\subsection{Results for monotone BSDE}

We consider the BSDE:
\begin{equation}  \label{eq:general_BSDE}
y_t = \xi +  \int_t^T f(u,y_u,z_u) du  -\int_t^T z_u dW_u - \int_t^T dm_u
\end{equation}
Let us now recall several classical results on the BSDE \eqref{eq:general_BSDE}. If {\bf (H)} holds\footnote{In fact in this section, Hypothesis {\bf H4} could be replaced by a more general condition (see \cite{krus:popi:14}, Assumption (H2)).} and if for some $p> 1$
\begin{equation} \label{eq:integrability_cond}
\bE \left[ |\xi|^p + \int_0^T |f^{0}_t|^p dt \right] < +\infty,
\end{equation}
then there is a unique solution $(y,z,m) \in \bD^p \times \bH^p \times \bM^p$ (see \cite{bria:dely:hu:03}, \cite{elka:huan:97} or \cite{krus:popi:14}). Let us emphasize that the quasi-left continuity property of the filtration assumed in \cite{krus:popi:14} is in fact unnecessary (see the introduction of \cite{bouc:poss:zhou:15} or \cite{popi:16}). The next {\it a priori} estimate on $(y,z,m)$ will be crucial.
\begin{Lemma}[{\it A priori} estimate] \label{lem:a_priori_estim_BSDE}
Under Condition {\bf (H)}, for any $\alpha > p L_1 + \frac{L_2^2}{(p-1)\wedge 1}$ there exists a constant $C_{p}$ such that 
$$\|e^{\alpha.}y\|_{ \bD^p } + \|z\|_{ \bH^{p,\alpha} } + \| m\|_{ \bM^{p,\alpha}} \leq C_{p} \bE \left[ e^{\alpha pT}|\xi|^p + \int_0^T |e^{\alpha t}f^{0}_t|^p dt \right].$$
\end{Lemma}
\begin{proof}
See for example Propositions 2 and 3 in \cite{krus:popi:14}. 
\end{proof}

Moreover from the proof of this lemma, we get the next classical estimates: for any stopping time $0\leq\tau\leq T$
\begin{equation} \label{eq:boundedness_y}
|y_\tau| \leq C\left(  \bE \left[ |\xi|^\kappa + \int_\tau^{T}  |f^{0}_s|^\kappa ds \bigg| \cF_{\tau} \right] \right)^{1/\kappa}
\end{equation}
for any $1 < \kappa \leq p$ and 
\begin{equation} \label{eq:integrability_y}
\bE \int_t^T |y_s|^p ds \leq C \bE \left( |\xi|^p + \int_t^{T}  |f^{0}_s|^p ds \right).
\end{equation}
The constant $C$ depends only on $\kappa$ or $p$, $T$, $L_1$ and $L_2$. 

The next trick ({\it linearization procedure}) is used several times in this paper. If $(y,z,m)$ satisfies the BSDE \eqref{eq:general_BSDE}, then
$$y_t =\xi +  \int_t^T f^{0}_s ds + \int_t^T \lambda_s y_s ds + \int_t^T \zeta_s z_s ds - \int_t^T z_s dW_s - \int_t^T dm_s$$
where 
$$\lambda_s  = \frac{f(s,y_s,0)- f^{0}_s}{y_s} \1_{y_s \neq 0},$$
and 
$$\zeta^{i}_{t} = 
(z^{i}_{t})^{-1}(f(t,y_{t},z^{(i)}_{t})-f(t,y_{t},z^{(i-1)}_{t})) \1_{z^{i}_{t} \neq 0}
$$
where for $1 \leq i \leq d$, $z^{(i)}_{t}$ is the $d$-dimensional vector in which the first $i$-components are equal to the ones of $z_t$ and the $d-i$ others are equal to zero. From {\bf H3}, $\zeta$ is a bounded (by $L_2$) vector-valued process. From the monotone condition {\bf H2}, $\lambda_s \leq L_1$ a.s. If $\Lambda_t = \exp \left(\int_0^t \lambda_s ds \right)$, then
\begin{equation} \label{eq:linear_procedure}
\Lambda_t y_t = \Lambda_T \xi +  \int_t^T \Lambda_s f^{0}_s ds + \int_t^T \zeta_s \Lambda_s z_s ds - \int_t^T \Lambda_s z_s dW_s - \int_t^T \Lambda_s dm_s.
\end{equation}
Hence if $\bQ$ is the probability measure equivalent to $\bP$ defined by the density
\begin{equation} \label{eq:density_chgt_measure}
\cE \left( -\int_0^. \zeta_s dW_s \right) = \exp \left( -\int_0^. \zeta_s dW_s - \frac{1}{2} \int_0^. \|\zeta_s\|^2 ds \right),
\end{equation}
we obtain
$$y_t = \bE^\bQ \left[ \frac{\Lambda_T}{\Lambda_t} \xi +  \int_t^T \frac{\Lambda_s}{\Lambda_t} f^{0}_s ds \bigg| \cF_t \right] .$$
And for any $0\leq t \leq s \leq T$, 
$$0 < \frac{\Lambda_s}{\Lambda_t}  = \exp \left(\int_t^s \lambda_u du \right)\leq  \exp \left(L_1 (t-s)\right).$$ 
This implies immediately that if $\xi$ and $f^{0}$ are $\bP-a.s.$ bounded, then $y$ is bounded. 

Moreover comparison principle (\cite[Proposition 4]{krus:popi:14}) and stability property (\cite[Propositions 2 and 3]{krus:popi:14}) hold for monotone BSDE.
\begin{Lemma}[Comparison] \label{lem:comp_sol_gene_BSDE}
We consider two generators $f_1$ and $f_2$ satisfying {\bf (H)}. Let $\xi^1$ and $\xi^2$ be two terminal conditions for BSDEs \eqref{eq:general_BSDE} driven respectively by $f_1$ and $f_2$. Denote by $(y^1,z^1,m^1)$ and $(y^2,z^2,m^2)$ the respective solutions in some $\bD^p \times \bH^p \times \bM^p$ with $p>1$. If $\xi^1 \leq \xi^2$ and $f_1(t,y^1_t,z^1_t) \leq f_2(t,y^1_t,z^1_t)$, then a.s. for any $t \in [0,T]$, $y^1_t \leq y^2_t$.
\end{Lemma}
Note that a strict comparison principle does not hold in general (see \cite[Proposition 5.34]{pard:rasc:14} and the comments just after).
\begin{Lemma}[Stability] \label{lem:L2_stability}
Let now $(\xi,f)$ and $(\xi',f')$ be two sets of data each satisfying the assumptions {\bf (H)}. Let $(y,z,m)$ (resp. $(y',z',m')$) denote the solution of the BSDE \eqref{eq:general_BSDE} with data $(\xi,f)$ (resp. $(\xi',f')$). Define
$$(\Delta y,\Delta z,  \Delta m, \Delta \xi, \Delta f) = (y-y',z-z',m-m', \xi - \xi', f-f').$$
Then there exists a constant $C$ depending on $L_1$, $L_2$, $p$ and $T$, such that
\begin{eqnarray*}
&& \E\left[  \sup_{t\in [0,T]} |\Delta y_t|^p + \left( \int_0^T |\Delta z_s|^2 ds \right)^{p/2} + [ \Delta m ]^{p/2}_T \right] \\
&& \qquad \leq  C  \E \left[ |\Delta \xi|^p + \int_0^T |\Delta f(t,y'_t,z'_t) |^p dt  \right].
\end{eqnarray*}
\end{Lemma}

Let us describe why we assume Condition {\bf H4}, and not some weaker growth condition. Indeed for second order BSDE (Section \ref{sect:2BSDE}) we used that the solution $(y,z,m)$ is obtained by approximation with a sequence of solutions $(y^{n},z^{n},m^{n})$ of Lipschitz BSDE\footnote{Note that this setting is sufficient to solve our control problem. Weaker conditions could be introduced using Mazur's Lemma and this technical point is left for further research. }. This is the reason why polynomial growth of $f$ w.r.t. $y$ is assumed in {\bf H4}, as in the paper \cite{bria:carm:00}. In their work the filtration is generated by the Brownian motion and the generator $f$ satisfies Condition {\bf (H')}. Hence we extend it to our setting.
\begin{Lemma}[Lipschitz approximation]\label{lem:L2_approx}
Assume that {\rm \textbf{(H)}} holds and that $\xi \in \bL^{\bsp \sq}$ and $f^0 \in \bH^{\bsp \sq}$ for some $\bsp  > \varrho/(\varrho-1)$. The solution $(y,z,m)$ of the BSDE \eqref{eq:general_BSDE} belongs to $\bD^{\bsp \sq} \times \bH^{\bsp \sq} \times \bM^{\bsp \sq}$ and is obtained as the limit in $\bD^{p} \times \bH^{p} \times \bM^p$ of a sequence $(y^n,z^n,m^n)$ solution of Lipschitz BSDEs for $p$ satisfying \eqref{eq:cond_int}, {\it i.e.}
$$1 < p \leq \frac{\varrho \bsp}{\varrho + \bsp} < \bsp.$$
\end{Lemma}
\begin{proof}
The first part of the result is a direct consequence of Lemma \ref{lem:a_priori_estim_BSDE}. Let us now only explain the second assertion. We will adapt the result of \cite{bria:carm:00} and we refer to this paper for the details. We only give the main arguments. W.l.o.g. we can assume in this proof that $L_1=0$ (just consider $\bar y_t= e^{-L_1t} y_t$, $\bar z_t= e^{-L_1t} z_t$, $\bar m_t= e^{-L_1t} m_t$, instead of $(y,z,m)$). 

\vspace{0.5cm}
\noindent {\bf Step 1.} First we consider the following BSDE 
\begin{equation}  
y_t = \xi +  \int_t^T  f(s,y_s,v_s) ds  -\int_t^Tz_s dW_s - \int_t^T dm_s
\end{equation}
where $v$ belongs to $\bH^{\bsp \sq}$ and $f$ satisfies {\bf (H')}. We denote by $h$ the function $h(t,y) = f(t,y,v_t)$. This function $h$ satisfies Conditions {\bf H1}-{\bf H2} and {\bf H4'}. And $h^0_t = h(t,0) \in \bH^{\bsp \sq}$. We construct a sequence of Lipschitz functions $h_n$ which approximate $h$. Let $\varsigma : \bR \mapsto \bR_+$ be a non-negative function with the unit ball for support and such that $\int \varsigma(u)du = 1$ and define for each integer $n>1$, $\varsigma_n(u) =n\varsigma(nu)$. We denote also for each integer $n$, by $\Theta_n$, a $C^\infty$ function from $\bR$ to $\bR_+$ such that $0\leq \Theta_n\leq 1$, $\Theta_n(u)=1$ for $|u | \leq n$ and $\Theta_n(u)=0$ as soon as $|u|>n+1$. We set 
$$\xi_n  = \frac{n\xi}{n\vee |\xi|},\quad \widetilde h_n(t,y) = \frac{n h(t,y) }{n\vee |h^0_t|}.$$
Moreover
$$\varpi(n) = \lfloor e^{1/2}(n+2L_\sq)\sqrt{1+T^2} \rfloor +1 $$
where $\lfloor r\rfloor $ is the integer part of $r$, $L_\sq$ coming from {\bf H4'}, and we define as the convolution product 
\begin{equation} \label{eq:lip_approx_gene_2}
h_n (t,.) = \varsigma_n *  (\Theta_{\varpi(n+1)} \widetilde h_n(t,.)) ,\ t \in [0,T].
\end{equation}
This function $h_n$ is globally Lipschitz w.r.t. $y$ uniformly in $t$ and $\omega$ with $|h_n(t,0)|\leq n\wedge |h^0_t| + 2L_\sq$. Moreover for any $(t,\omega,y)$
$$yh_n(t,y) \leq ((n\wedge |h^0_t|) + 2L_\sq)|y|,$$
and for any $y$ and $y'$ in the ball $\overline{B(0,\varpi(n))}$ then 
$$(y-y')(h_n(t,y)-h_n(t,y')) \leq 0.$$
In other words $h_n$ is only locally monotone (only in a given ball with the radius depending on $n$). Let $(y^n,z^n,m^n)$ be the unique solution of the BSDE
\begin{equation}  
y^n_t = \xi_n +  \int_t^T  h_n(u,y^n_u) du  -\int_t^T z^n_u dW_u - \int_t^T dm^n_u
\end{equation}
in $\bD^{\bsp \sq} \times \bH^{\bsp  \sq} \times \bM^{\bsp  \sq}$. This solution verifies (see \cite[Proposition 2.1]{bria:carm:00})
$$\sup_{t \in [0,T]} |y^n_t| \leq (n+2L_\sq)e^{1/2}\sqrt{1+T^2}.$$
Hence $y^n_t$ is in $\overline{B(0,\varpi(n))}$. And from Lemma \ref{lem:a_priori_estim_BSDE} we also have for any $1 < p\leq \bsp$ and for some $\alpha$ large enough
$$\sup_{n\in \bN^*} \left[ \|e^{\alpha.}y^n\|_{ \bD^{p\sq} } + \|z^n \|_{ \bH^{p\sq,\alpha} } + \|m^n\|_{ \bM^{p\sq,\alpha}} \right] \leq C_{p} \bE \left[ e^{p\alpha \sq T}|\xi|^{p\sq} + \int_0^T e^{\alpha t} (|h^{0}_t| +2L_\sq)^{p\sq} dt \right].$$

Now we fix two integers $\ell$ and $n$ such that $\ell\geq n$ and
$$\delta y = y^{\ell} - y^n,\quad \delta z = z^{\ell}-z^n, \quad \delta m = m^{\ell}-m^n,\quad \delta \xi = \xi_{\ell} - \xi_n.$$
For $1 < p\leq \bsp$, with $c(p)$ defined by \eqref{def:c(p)}, we use It\^o's formula
\begin{eqnarray*}
&& \left| \delta y_t \right|^p + c(p) \int_t^T \left| \delta y_{u^-} \right|^{p-2}\1_{\delta y_{u-} \neq 0} |\delta z_u|^2 du +c(p) \int_t^T \left| \delta y_{u-} \right|^{p-2}\1_{\delta y_{u-} \neq 0} d[\delta m]^c_u  \\
&&+ \sum_{t< u \leq T}  \left[ |\delta y_{u-}+\Delta (\delta m)_u|^p - |\delta y_{u-}|^p - p\phi_p(\delta y_{u-})  \Delta (\delta m)_u \right] \\
&&\quad = |\delta \xi|^p  + p \int_t^T \left[ h_\ell(u,y^\ell_u) - h_n(u,y^n_u) \right]  \phi_p(\delta y_{u-}) du \\
& & \qquad  - p \int_t^T \phi_p(\delta y_{u-}) \delta z_u dW_u  -p  \int_t^T \phi_p(\delta y_{u-}) d(\delta m)_u  \\
&&\quad = |\delta \xi|^p  + p \int_t^T \left[ h_\ell(u,y^\ell_u) - h_\ell(u,y^n_u) \right]  \phi_p(\delta y_{u-}) du \\
&& \qquad +  p \int_t^T \left[ h_\ell(u,y^n_u) - h_n(u,y^n_u) \right] \phi_p(\delta y_{u-}) du \\
& & \qquad  - p \int_t^T \phi_p(\delta y_{u-}) \delta z_u dW_u  -p  \int_t^T\phi_p(\delta y_{u-}) d(\delta m)_u  
\end{eqnarray*}
Since $|y^n_t|\leq \varpi(n) \leq \varpi(\ell)$, we use the local monotonicity of $h_\ell$ and we obtain:
\begin{eqnarray*}
&& \left| \delta y_t \right|^p + c(p) \int_t^T \left| \delta y_{u-} \right|^{p-2}\1_{\delta y_{u-} \neq 0} |\delta z_u|^2 du +c(p) \int_t^T \left| \delta y_{u-} \right|^{p-2}\1_{\delta y_{u-} \neq 0} d[\delta m]^c_u  \\
&&+ \sum_{t< u \leq T}  \left[ |\delta y_{u-}+\Delta (\delta m)_u|^p - |\delta y_{u-}|^p - p\phi_p(\delta y_{u-})  \Delta (\delta m)_u \right] \\
&&\quad \leq |\delta \xi|^p  + p \int_t^T \left[ h_\ell(u,y^n_u) - h_n(u,y^n_u) \right] \phi_p(\delta y_{u-}) du \\
& & \qquad  - p \int_t^T \phi_p(\delta y_{u-}) \delta z_u dW_u  -p  \int_t^T\phi_p(\delta y_{u-}) d(\delta m)_u  
\end{eqnarray*}
Since the set $\{\delta y_u \neq \delta y_{u-}\}$ is countable, arguing as in the proof of \cite[Proposition 3]{krus:popi:14}, we deduce that there exists a constant $C$ such that 
\begin{eqnarray} \nonumber
&& \bE \left[ \sup_{t\in [0,T]} \left| \delta y_t \right|^p +  \left( \int_0^T |\delta z_u|^2 du \right)^{p/2}+ ([\delta m]_T)^{p/2} \right] \\ \label{eq:strong_conv_L_2}
&& \qquad \leq  C \bE \left[ |\delta \xi |^p + \int_0^T \left| h_\ell(u,y^n_u) - h_n(u,y^n_u) \right| |\delta y_{u}|^{p-1} du\right].
\end{eqnarray}
Since $\xi \in \bL^{\bsp \sq}$, then $\delta \xi$ goes to zero in $\bL^p$ as $n$ and $\ell$ tend to $+\infty$. For any given number $k$, we put 
\begin{eqnarray*}
S^\ell_n & = & \bE \left[  \int_0^T \1_{\{(|y^n_u|+|y^\ell_u|)\leq k\}}\left| h_\ell(u,y^n_u) - h_n(u,y^n_u) \right| |\delta y_u|^{p-1} du\right] \\
R^\ell_n & = &  \bE \left[  \int_0^T \1_{\{(|y^n_u|+|y^\ell_u|)\geq k\} }\left| h_\ell(u,y^n_u) - h_n(u,y^n_u) \right| |\delta y_u|^{p-1} du\right] .
\end{eqnarray*}
With these notations we have 
\begin{eqnarray} \nonumber
&& \bE \left[\int_0^T \left| h_\ell(u,y^n_u) - h_n(u,y^n_u) \right| |\delta y_u|^{p-1} du\right]  = S^\ell_n  + R^\ell_n \\ \label{eq:decomp_S_R}
&& \quad \leq C_p k^{p-1} \bE \left[ \sup_{|y|\leq k} \int_0^T \left| h_\ell(u,y) - h_n(u,y) \right| du\right] + R^\ell_n.
\end{eqnarray}
Since $h(s,.)$ is continuous ($\bP$-a.s., for every $s$), $h_n(s,.)$ converges towards $h(s,.)$ uniformly on compact sets. Taking into account that 
$$\sup_{|y|<k} |h_n(s,y)| \leq |h(s,0)| +2^\sq L_\sq(1+k^\sq),$$
Lebesgue's dominated convergence theorem implies that for any fixed number $k$, the quantity
$$C_pk^{p-1} \bE \left[ \sup_{|y|\leq k} \int_0^T \left| h_\ell(u,y) - h_n(u,y) \right| du\right] $$
goes to 0 as $n$ tends to infinity uniformly with respect to $\ell$. The proof will be finished if we prove the convergence of the rest $R^\ell_n$. Using H\"older's inequality we get the following upper bound:
\begin{eqnarray} \label{eq:bound_on_R_m_n}
R^\ell_n &\leq &  \left[\bE \int_0^T \1_{\{(|y^n_u|+|y^\ell_u|)\geq k\}} du \right]^{\frac{(\sq-1)(p-1)}{p\sq}} \\ \nonumber
&& \quad \times  \left[\bE \int_0^T \left| h_\ell(u,y^n_u) - h_n(u,y^n_u) \right|^{\frac{p\sq}{\sq+p-1}} |\delta y_u|^{\frac{p(p-1)\sq}{\sq+p-1}} du\right]^{\frac{\sq+p-1}{p\sq}}.
\end{eqnarray}
For the first term in the product we use Chebychev's inequality:
 \begin{eqnarray} \nonumber
&& \left[\bE \int_0^T \1_{\{(|y^n_u|+|y^\ell_u|)\geq k\}} du\right]^{\frac{(\sq-1)(p-1)}{p\sq}}  \leq  k^{(1-\sq)(p-1)}\left[ \bE \int_0^T  (|y^n_u|+|y^\ell_u|)^{p\sq} du\right]^{\frac{(\sq-1)(p-1)}{p\sq}} \\ \label{eq:first_ineq_R_m_n}
&& \quad \leq  k^{(1-\sq)(p-1)} 2^{p\sq-1} T^{\frac{(\sq-1)(p-1)}{p\sq}} \left[\sup_{n\in \bN^*} \bE \left(\sup_{t\in [0,T]} (|y^n_u|)^{p\sq}\right) \right]^{\frac{(\sq-1)(p-1)}{p\sq}}.  
\end{eqnarray}
Remember that the above expectation is bounded uniformly w.r.t. $n$. Thus the right-hand side of \eqref{eq:first_ineq_R_m_n} is uniformly bounded. We have to control
\begin{equation} \label{eq:def_A_m_n}
A^\ell_n =\bE  \int_0^T \left| h_\ell(u,y^n_u) - h_n(u,y^n_u) \right|^{\frac{p\sq}{\sq+p-1}} |\delta y_u|^{\frac{p(p-1)\sq}{\sq+p-1}} du.
\end{equation}
By Young's inequality
 \begin{eqnarray*}
A^\ell_n & \leq & \frac{p-1}{\sq+p-1} \bE \int_0^T |\delta y_u|^{p\sq} du + \frac{\sq}{\sq+p-1}  \bE \int_0^T\left| h_\ell(u,y^n_u) - h_n(u,y^n_u) \right|^{p} du \\
& \leq & C  \sup_{n\in \bN^*} \bE \left(\sup_{t\in [0,T]} (|y^n_u|)^{p\sq}\right)+  C \bE  \int_0^T \left| h_\ell(u,y^n_u) - h_n(u,y^n_u) \right|^{p} du \\
&\leq & 2 C \sup_{n\in \bN^*} \bE \left(\sup_{t\in [0,T]} (|y^n_u|)^{p\sq}\right) + C \bE  \int_0^T \left( \left| f^0_u \right|^{p} + \|v_u\|^{p\sq} \right) du.
\end{eqnarray*}
Thus $A^\ell_n$ remains bounded w.r.t. $n$ and $\ell$. Collecting \eqref{eq:strong_conv_L_2}, \eqref{eq:decomp_S_R}, \eqref{eq:first_ineq_R_m_n}, \eqref{eq:bound_on_R_m_n} with \eqref{eq:def_A_m_n}, we deduce that there exists a constant $C$ such that for any $k$ and $\eps > 0$, there exists $N$ such for $\ell\geq N$ and $n\geq N$:
\begin{equation*} 
\bE \left[ \sup_{t\in [0,T]} \left| \delta y_t \right|^p + \left(  \int_0^T |\delta z_u|^2 du\right)^{p/2} + [\delta m]^{p/2}_T \right] \leq C \frac{1}{k^{(\sq-1)(p-1)}} + \eps.
\end{equation*}
Since we can fix $k$ large enough to ensure that the right-hand side is smaller than $2\eps$, we deduce the convergence result.

\vspace{0.5cm}
\noindent {\bf Step 2.} We consider now the general BSDE \eqref{eq:general_BSDE}, but with Condition {\bf (H')}. The Lipschitz approximation will be obtained by a fixed point argument in $\bD^p \times \bH^p \times \bM^p$, arguing as in the proof of \cite[Theorem 3.6]{bria:carm:00}, with straightforward modifications.

\vspace{0.5cm}
\noindent {\bf Step 3.} For the more general growth condition {\bf H4}, consider 
$$f_n(t,y,z) = (f(t,y,z)-f(t,0,0)) \frac{n}{\Psi(t) \vee n} + f(t,0,0).$$
Then $f_n$ is still Lipschitz continuous w.r.t. $z$, continuous and monotone w.r.t. $y$ and satisfies:
$$|f_n(t,y,0)- f_n(t,0,0)| \leq   \frac{n\Psi(t) }{\Psi(t) \vee n}(1+|y|^\sq)\leq n(1+|y|^\sq).$$
$f_n$ satisfies {\bf H4'} with $L_\sq=n$. Thus there is a sequence $(y^n,z^n,m^n)$ of solutions for the BSDE with generator $f_n$. As before let us define 
$$\delta y = y^\ell - y^n, \quad \delta z = z^\ell-z^n, \quad \delta m = m^\ell-m^n.$$
We use It\^o's formula for $1< p\leq  \frac{\varrho \bsp}{\varrho + \bsp}$ and $\alpha > \frac{p}{2(p-1)} L_2^2$
\begin{eqnarray*}
&&e^{\alpha t} \left| \delta y_t \right|^p + c(p) \int_t^Te^{\alpha u} \left| \delta y_{u-} \right|^{p-2}\1_{\delta y_{u-} \neq 0} |\delta z_u|^2 du +c(p) \int_t^Te^{\alpha u} \left| \delta y_{u-} \right|^{p-2}\1_{\delta y_{u-} \neq 0} d[\delta m]^c_u  \\
&&+ \sum_{t< u \leq T}  \left[ |\delta y_{u-}+\Delta (\delta m)_u|^p - |\delta y_{u-}|^p - p\phi_p(\delta y_{u-})  \Delta (\delta m)_u \right]\\
&&\quad =  p \int_t^T e^{\alpha u}\left[ f_\ell(u,y^\ell_u,z^\ell_u) - f_n(u,y^n_u,z^n_u) \right] \phi_p(\delta y_{u-}) du \\
& & \qquad -\int_t^T \alpha e^{\alpha u}\left| \delta y_u  \right|^p du - p \int_t^T e^{\alpha u}\phi_p(\delta y_{u-}) \delta z_u dW_u \\
&& \qquad -p  \int_t^T e^{\alpha u}\phi_p(\delta y_{u-}) d(\delta m)_u \\
&&\quad \leq \left( \frac{p L_2^2}{(p-1)} -\alpha \right) \int_t^T e^{\alpha u}|\delta  y_u |^p du + \frac{c(p)}{2}\int_t^T e^{\alpha u}\left| \delta y_{u-} \right|^{p-2} \1_{\delta y_{u-} \neq 0}  |\delta z_u |^2 du \\
&& \qquad + p \int_t^T e^{\alpha u}\left[ f_\ell(u,y^n_u,z^n_u) - f_n(u,y^n_u,z^n_u) \right] \phi_p(\delta y_{u-}) du\\
& & \qquad - p \int_t^T e^{\alpha u}\phi_p(\delta y_{u-}) \delta z_u dW_u -p  \int_t^T e^{\alpha u}\phi_p(\delta y_{u-}) d(\delta m)_u .
\end{eqnarray*}
Moreover
$$\left| f_\ell(u,y^n_u,z^n_u) - f_n(u,y^n_u,z^n_u) \right| \leq (\Psi(u) \1_{\Psi(u) \geq n\wedge \ell}) \left( 1 + |y^n_u|^\sq \right).$$
Young's inequality implies that for any $\mathfrak{k} > 0$
\begin{eqnarray*}
&& \left| e^{\alpha u}\left[ f_\ell(u,y^n_u,z^n_u) - f_n(u,y^n_u,z^n_u) \right] \phi_p(\delta y_u) \right| \\
&& \quad \leq \frac{e^{\alpha u}}{p\mathfrak{k}^{p-1}} (\Psi(u) \1_{\Psi(u) \geq n\wedge \ell})^p \left( 1 + |y^n_u|^\sq \right)^p + \mathfrak{k} \frac{p-1}{p} e^{\alpha u} |\delta y_u|^{p}. 
\end{eqnarray*}
Therefore 
\begin{eqnarray*}
&&e^{\alpha t} \left| \delta y_t \right|^p + \frac{c(p)}{2} \int_t^Te^{\alpha u} \left| \delta y_u \right|^{p-2}\1_{\delta y_u \neq 0} |\delta z_u|^2 du \\
&& +c(p) \int_t^Te^{\alpha u} \left| \delta y_{u-} \right|^{p-2}\1_{\delta y_{u-} \neq 0} d[\delta m]^c_u \\
&& + \sum_{t< u \leq T} \left[ |\delta y_{u-}+\Delta (\delta m)_u|^p - |\delta y_{u-}|^p - p\phi_p(\delta y_{u-})  \Delta (\delta m)_u \right] \\
&&\quad  \leq \left( \frac{p L_2^2}{(p-1)} -\alpha + \mathfrak{k} \frac{p-1}{p} \right) \int_t^T e^{\alpha u}|\delta  y_u |^p du \\
&& \qquad + \frac{1}{p\mathfrak{k}^{p-1}} \int_t^T e^{\alpha u} (\Psi(u) \1_{\Psi(u) \geq n\wedge \ell})^p \left( 1 + |y^n_u|^\sq \right)^p du\\
& & \qquad - p \int_t^T e^{\alpha u}\phi_p(\delta y_{u-}) \delta z_u dW_u -p  \int_t^T e^{\alpha u}\phi_p(\delta y_{u-}) d(\delta m)_u . 
\end{eqnarray*}
Fix $\alpha$ large enough such that 
$$\alpha > \frac{p L_2^2}{(p-1)} +\mathfrak{k}  \frac{p-1}{p},$$
and take the expectation. Martingale terms are true martingales thus with H\"older's inequality
\begin{eqnarray*}
&& \frac{c(p)}{2} \bE \int_0^Te^{\alpha u} \left| \delta y_{u-} \right|^{p-2}\1_{\delta Y_{u-} \neq 0} |\delta z_u|^2 du + c(p) \bE \int_0^Te^{\alpha u} \left| \delta y_{u-} \right|^{p-2}\1_{\delta y_{u-} \neq 0} d[\delta m]^c_u \\
&& + \bE \sum_{0< u \leq T}\left[ |\delta y_{u-}+\Delta (\delta m)_u|^p - |\delta y_{u-}|^p - p\phi_p(\delta y_{u-})  \Delta (\delta m)_u \right]\\
&&\quad  \leq  \frac{1}{p\mathfrak{k}^{p-1}} \bE \int_0^T e^{\alpha u} (\Psi(u) \1_{\Psi(u) \geq n\wedge \ell})^p \left( 1 + |y^n_u|^\sq \right)^p du \\
&& \quad \leq \frac{1}{p\mathfrak{k}^{p-1}} \left(\bE \int_0^T e^{\alpha u} (\Psi(u) \1_{\Psi(u) \geq n\wedge \ell})^{p\bsp/(\bsp-p)}  du \right)^{(\bsp-p)/\bsp}\\
&& \qquad \qquad \times \left( \bE \int_0^T e^{\alpha u}  \left( 1 + |y^n_u|^\sq \right)^{\bsp} du \right)^{p/\bsp}.
\end{eqnarray*}
But for $p\leq\frac{\bsp \varrho}{\bsp + \varrho} < \bsp$, we have $p\bsp/(\bsp-p) \leq \varrho$. Hence the right-hand side of the previous inequality goes to zero as $\ell$ and $n$ tend to $+\infty$. Now we proceed as in the proof of Proposition 3 in \cite[Proposition 3]{krus:popi:14} and we deduce that the sequence $(y^n,z^n,m^m)$ converges in $\bD^p\times \bH^{p} \times \bM^p$ to $(y,z,m)$.

\end{proof}

\begin{Remark} \label{rem:on_lipschitz_approx}
The condition $\bsp > \varrho/(\varrho-1)$ is equivalent to $1< \frac{ \varrho \bsp}{\varrho + \bsp}$. 
\begin{itemize}
\item If {\bf H4'} holds, then $\varrho = +\infty$ and $p=\bsp$. 

\item If everything is bounded ($\xi$ and $f^0_s$), then $y$ is also bounded and we only need $\varrho > 1$.

\end{itemize}
\end{Remark}

Finally let us recall a technical but crucial lemma, called Lemma A.2 in \cite{poss:tan:zhou:15}. The result is the same, but the proof has to be modified since $f$ is no more Lipschitz continuous in $y$.
\begin{Lemma}
For any $\bF$-stopping times $0\leq r \leq \mathfrak{u} \leq \tau \leq T$, any decreasing sequence of $\bF$-stopping times $(\tau_n)_{n\geq 1}$ converging $\bP$-a.s. to $\tau$ and any $\bF$- progressively measurable and right-continuous process $V \in \bD^{\bsp \sq}$, if $y(.,V_.)$ denotes the first component of the solution to the BSDE with terminal condition $V_.$ and some generator $f$ satisfying {\bf (H)}, we have
$$\lim_{n\to +\infty} \bE\left[ | y_\mathfrak{u}(\tau,V_\tau) - y_\mathfrak{u}(\tau_n,V_{\tau_n}) | \right] =0.$$
\end{Lemma}
\begin{proof}
By classical stability result, for any $\kappa < p$, there exists a constant $C$ depending only on $T$, $\kappa$, $L_1$ and $L_2$ such that
\begin{eqnarray*}
\bE \left[ | y_\mathfrak{u}(\tau,V_\tau) - y_\mathfrak{u}(\tau_n,V_{\tau_n}) | \right] & = & \bE \left[ | y_\mathfrak{u}(\tau,V_\tau) - y_\mathfrak{u}(\tau, y_\tau(\tau_n,V_{\tau_n})) | \right] \\
& \leq & C  \bE \left[ | V_\tau  - y_\tau(\tau_n,V_{\tau_n}) |^\kappa \right]
\end{eqnarray*}
Compared to the proof of \cite[Lemma A.2]{poss:tan:zhou:15} we do not use the complete linearization argument. But we strongly use the growth condition {\bf H4} with the {\it a priori} estimate given in Lemma \ref{lem:a_priori_estim_BSDE}. Indeed we only write that 
$$y_\tau (\tau_n,V_{\tau_n}) = \bE \left[ \cE \left( \int_\tau^{\tau_n} \zeta_s dW_s \right) \left( V_{\tau_n} + \int_\tau^{\tau_n} f(s,y_s(\tau_n,V_{\tau_n}),0)ds\right) \bigg| \cF_\tau \right].$$
Then 
\begin{eqnarray*}
&& \bE \left[ | y_\mathfrak{u}(\tau,V_\tau) - y_\mathfrak{u}(\tau_n,V_{\tau_n}) | \right]  \leq  C  \bE \left[ \cE \left( \int_\tau^{\tau_n} \zeta_s dW^\bP_s \right)^\kappa | V_\tau  - V_{\tau_n} |^\kappa \right] \\
& &\qquad +  C  \bE \left[ \cE \left( \int_\tau^{\tau_n} \zeta_s dW_s \right)^\kappa \int_\tau^{\tau_n} | f(s,y_s(\tau_n,V_{\tau_n}),0) |^\kappa ds \right] \\
&& \quad  \leq  C   \bE \left[ \cE \left( \int_\tau^{\tau_n} \zeta_s dW_s \right)^\kappa | V_\tau  - V_{\tau_n} |^\kappa \right] \\
& &\qquad +  C  \bE \left[ \cE \left( \int_\tau^{\tau_n} \zeta_s dW_s \right)^\kappa \int_\tau^{\tau_n} | f(s,0,0) |^\kappa ds \right] \\
& &\qquad +  C   \bE\left[ \cE \left( \int_\tau^{\tau_n} \zeta_s dW_s \right)^\kappa \int_\tau^{\tau_n} \Psi_s^\kappa \left( 1+| y_s(\tau_n,V_{\tau_n})|^\sq\right)^\kappa ds \right] . 
\end{eqnarray*}
Since $\zeta$ is bounded (by $L_2$), the Dol\'eans-Dade exponential appearing above has finite moments of any order. Now since we have an {\it a priori} estimate on $y_.(\tau_n,V_{\tau_n})$ in $\bD^{\bsp \sq}$, uniformly in $n$, we argue as in \cite{poss:tan:zhou:15} to conclude.
\end{proof}

\subsection{Reflected BSDE with monotone driver} \label{sect:mono_RBSDE}

In this section we extend the results contained in \cite{bouc:poss:zhou:15}, where the driver $f$ is supposed to be Lipschitz continuous w.r.t. $y$ and $z$. One of the main contributions of \cite{bouc:poss:zhou:15} is the existence of a solution for reflected BSDE in a general filtration, without quasi-left continuity condition. Here we follow the same scheme but for monotone generators satisfying hypothesis {\bf (H)}. Thus we do not give all details but we point out the differences. 

Let us remark that our condition {\bf H4} on the growth of the driver $f$ or the integrability assumption {\bf C1} on the terminal value $\xi$ and $f^0$ are not optimal, compared to the conditions imposed in \cite{klim:12} for example (see also among others \cite{klim:13b,klim:15,lepe:mato:xu:05} for reflected BSDE with monotone generator). This improvement of our result would be quite long and is left for further research.

\subsubsection*{Estimates on supersolution}

We first consider supersolution of the BSDE 
\begin{equation}  \label{eq:supersolution_general_BSDE}
y_t = \xi +  \int_t^T f(u,y_u,z_u) du - \int_t^T z_u dW_u - \int_t^T dm_u + \int_t^T dk_u. 
\end{equation}
The generator $f$ satisfies Condition {\bf (H)}. 
Let us begin with some {\it a priori} estimates. Here we use the notation $z\star W$ to denote the stochastic integral of $z$ w.r.t. $W$ and 
$$n = z\star W  + m - k, \quad \ell = m-k.$$
\begin{Lemma}[Equivalent to Lemma 2.1 in \cite{bouc:poss:zhou:15}]
Let us fix some $1< p \leq \frac{\varrho \bsp}{\varrho + \bsp} < \bsp$ (condition \eqref{eq:cond_int}). 

For all $\alpha > 0$, there exists a constant $C$ depending only on $L_1$, $L_2$, $p$ and $T$ such that 
\begin{equation} \label{eq:estim_2_8}
\| k \|^p_{\bI^{p,\alpha}} \leq C \left( \| e^{\frac{\alpha}{2} .} y \|^p_{\bD^p} + \| e^{\frac{\alpha }{2} .} (1+|y|^\sq) \|^p_{\bD^{\bsp}} \|\Psi\|^p_{\bL^{\widehat p}}+ \|z\|^p_{\bH^{p,\alpha}} + \|f^0\|^p_{\bH^{p,\alpha}}\right)
\end{equation}
with 
\begin{equation}\label{eq:def_widehat_p}
\widehat p = \frac{p\bsp}{(\bsp-p)} \leq \varrho.
\end{equation} 
Moreover for any $\eps > 0$, there exists $\alpha > 0$ and $C^{\eps,\alpha}$ such that if $p\geq 2$
\begin{eqnarray} \nonumber
&& \| y \|^p_{\bH^{p,\alpha}}  +  \| n \|^p_{\bM^{p,\alpha}}  \leq  \eps \|f^0 \|^p_{\bH^{p,\alpha}} \\  \nonumber
&&\qquad  + C^{\eps,\alpha}  \Bigg[ \| \xi \|^p_{\bL^p} + \| e^{\alpha .} y \|^p_{\bD^{p}} + \|(e^{\alpha. } y_- \star n)_T \|^{\frac{p}{2}}_{\bL^{\frac{p}{2}}}\1_{p> 2}  \\ \label{eq:estim_2_9}
&& \qquad \qquad \qquad \left. +  \bE \left(  \int_0^T e^{\alpha s} \phi_p(y_{s-}) dk_s\right)^+ \1_{p=2}\right].
\end{eqnarray}
and if $p \in (1,2)$
\begin{eqnarray} \nonumber
\| n \|^p_{\bM^{p,\alpha}} & \leq &  \eps \|f^0 \|^p_{\bH^{p,\alpha}} \\ \label{eq:estim_2_10}
 &  + & C^{\eps,\alpha}  \left[ \| \xi \|^p_{\bL^p} + \| e^{\alpha .} y \|^p_{\bD^{p}} +  \bE \left(  \int_0^T e^{\alpha s} \phi_p(y_{s-}) dk_s\right)^+ \right].
\end{eqnarray}
\end{Lemma}
\begin{proof}
For any $\alpha \geq 0$, the process
$$\Upsilon^\alpha_t = e^{\alpha t} y_t - \int_0^t e^{\alpha u} \left( f(u,y_u,z_u) du + \alpha y_u \right) du$$
is a supermartingale. And the non-decreasing process in its Doob-Meyer decomposition is $\int_0^t e^{\alpha u} dk_u$. Indeed the It\^o formula gives:
\begin{eqnarray*}
e^{\alpha t} y_t & = & y_0 + \int_0^t  \alpha e^{\alpha u}y_u du - \int_0^t e^{\alpha u}f(u,y_u,z_u) du \\
& + & \int_0^t e^{\alpha u} z_u dW_u + \int_0^t e^{\alpha u} dm_u - \int_0^t e^{\alpha u} dk_u.
\end{eqnarray*}
Therefore from \cite[Lemma A.1]{bouc:poss:zhou:15}, there exists a constant $C_p$ depending only on $p$ (and changing from line to line) such that 
\begin{eqnarray*}
&& \bE \left[\left( \int_0^T e^{\alpha u/2} dk_u \right)^p\right]  \leq C_p \|\Upsilon^{\alpha/2}\|^p_{\bD^p} \\
&& \leq  C_p \left( \|e^{\frac{\alpha}{2} .} y \|^p_{\bD^p} +\bE \left[\left( \int_0^T e^{\alpha u/2} \left| f(u,y_u,z_u) du + \frac{\alpha}{2} y_u \right| du \right)^p\right]  \right) \\
&& \leq  C_p \left( T^p  \alpha^p  \|e^{\frac{\alpha}{2} .} y \|^p_{\bD^p} + L_2^p \| z \|_{\bH^{p,\alpha}}^p + \|f^0\|_{\bH^{p,\alpha}}^p \right) \\
&& \qquad + C_p T^{p-1} \bE \left[ \int_0^T e^{\alpha u p/2} \Psi(u)^p (1+|y_u|^\sq)^p du \right]   \\
&&  \leq  C_p \left( T^p  \alpha^p  \|e^{\frac{\alpha}{2} .} y \|^p_{\bD^p} + L_2^p \| z \|_{\bH^{p,\alpha}}^p + \|f^0\|_{\bH^{p,\alpha}}^p \right) \\
&& \qquad + C_p T^{p-1}\left(  \bE \left[ \int_0^T e^{\frac{\alpha u \bsp}{2}}  (1+|y_u|^\sq)^{\bsp} du \right] \right)^{\frac{p}{\bsp}}  \left(  \bE \left[ \int_0^T   \Psi(u)^{\frac{p\bsp}{(\bsp-p)}} du \right] \right)^{\frac{\bsp-p}{\bsp}}.
\end{eqnarray*}
The proof of Estimate \eqref{eq:estim_2_9} is exactly the same as for Estimate (2.9) in \cite{bouc:poss:zhou:15}. Indeed only the monotonicity assumption {\bf H2} and not the growth condition {\bf H4} is used. Let us only detail the proof of Estimate \eqref{eq:estim_2_10}, which uses only the monotonicity condition {\bf H2} w.r.t. $y$ (not the Lipschitz condition). Indeed using \cite[Lemmas 7 and 8]{krus:popi:14} (see also \cite[Lemma A.2]{bouc:poss:zhou:15}), for $0\leq t\leq T$:
\begin{eqnarray*}
&& e^{\alpha t} |y_t|^p +c(p) \int_t^T  |y_{s}|^{p-2}  \1_{y_s\neq 0} d[ n ]^c_s \\
&& +  c(p) \sum_{t< s \leq T} e^{\alpha s} |\Delta n_s|^2 \left[  |y_{s-}|^2 \vee |y_{s-}+ \Delta n_s|^{2} \right]^{\frac{p}{2}-1} \1_{|y_{s-}| \vee |y_{s-}+ \Delta n_s| \neq  0} \\
&& \quad  \leq e^{\alpha T}|\xi|^p + p \int_t^T e^{\alpha s} \phi_p(y_s) f(s,y_s,z_s) ds  - \alpha \int_t^T e^{\alpha s} |y_s|^p ds \\
&& \qquad -  p \int_t^T e^{\alpha s} \phi_p(y_{s-}) dn_s . 
\end{eqnarray*}
Then the assumptions on $f$ and Young's inequality imply that 
\begin{eqnarray*}
&& e^{\alpha t} |y_t|^p +c(p)(1 - \mathfrak{k}) \int_t^T  |y_{s}|^{p-2}  \1_{y_s\neq 0} d[ n ]^c_s \\
&& +  c(p) \sum_{t< s \leq T} e^{\alpha s} |\Delta n_s|^2 \left[  |y_{s-}|^2 \vee |y_{s-}+ \Delta n_s|^{2} \right]^{\frac{p}{2}-1} \1_{|y_{s-}| \vee |y_{s-}+ \Delta n_s| \neq  0} \\
&& \quad  \leq e^{\alpha T}|\xi|^p + p \int_t^T e^{\alpha s} |y_s|^{p-1} | f^0_s| ds + \left(pL_1+ \frac{p L_2^2}{2\mathfrak{k} (p-1)}  - \alpha\right) \int_t^T e^{\alpha s} |y_s|^p ds \\
&& \qquad -  p \int_t^T e^{\alpha s} \phi_p(y_{s-}) dn_s  \\
&& \quad  \leq e^{\alpha T}|\xi|^p + \int_t^T e^{\alpha s} | f^0_s|^p ds-  p \int_t^T e^{\alpha s} \phi_p(y_{s-}) dn_s  \\
&& \qquad + (p-1) \sup_{s\in [0,T]}  e^{\alpha s} |y_s|^p  + \left( pL_1 + \frac{p L_2^2}{2\mathfrak{k}(p-1)}  - \alpha\right) \int_t^T e^{\alpha s} |y_s|^p ds
\end{eqnarray*}
for any $\mathfrak{k} > 0$. We choose $\mathfrak{k}$ such that $\mathfrak{k} < 1$, and $\alpha \geq pL_1+ \frac{p L_2^2}{2\mathfrak{k}(p-1)}$. Then taking the expectation we obtain an explicit constant $C$ depending only on $L_1$, $L_2$, $p$, $T$ and $\alpha$ such that 
\begin{eqnarray*}
A & = & \bE \int_0^T  |y_{s}|^{p-2}  \1_{y_s\neq 0} d[ n ]^c_s \\
 &+ &  \bE \sum_{0< t \leq T} e^{\alpha s} |\Delta n_s|^2 \left[  |y_{s-}|^2 \vee |y_{s-}+ \Delta n_s|^{2} \right]^{\frac{p}{2}-1} \1_{|y_{s-}| \vee |y_{s-}+ \Delta n_s| \neq  0} \\
& \leq & C \left( \bE |\xi|^p + \bE \int_0^T e^{\alpha s} | f^0_s|^p ds +  \bE \left[\left(  \int_0^T e^{\alpha s} \phi_p(y_{s-}) dk_s\right)^+ \right] \right).
\end{eqnarray*}
The arguments of \cite{krus:popi:14}, proof of Proposition 3, Step 2, give:
$$\| n\|^p_{\bM^{p,\alpha}} \leq (2-p) \left( \frac{2C}{\eps p}\right)^{\frac{1}{p-1}}\| e^{\alpha .} y\|^p_{\bD^{p}}  + \frac{\eps}{C} A.$$
This achieves the proof of the Lemma. 
\end{proof}

From this lemma, we can copy the arguments in the proof of \cite[Theorem 2.1]{bouc:poss:zhou:15} and we deduce that 
\begin{Prop}[\cite{bouc:poss:zhou:15}, Theorem 2.1] \label{prop:a_priori_estim_N}
If $(y,z,m,k)$ is a supersolution of \eqref{eq:supersolution_general_BSDE} in the space $\bD^{\bsp \sq}\times \bH^{p }\times \bM^{p } \times \bI^{p }_+$ with $\bsp > \varrho/(\varrho-1)$ then for any $1 < p \leq \frac{\varrho \bsp}{\varrho + \bsp}$ and for $\alpha$ large enough, there exists a constant $C$ such that 
\begin{equation*} 
 \|z\|^p_{\bH^{p,\alpha}}+ \|m\|^p_{\bM^{p,\alpha}} + \| k \|^p_{\bI^{p,\alpha}} \leq C \left( \|\xi\|^p_{\bL^p} + \| y \|^p_{\bD^p} + \| (1+|y|^\sq) \|^p_{\bD^{\bsp}}  \|\Psi\|^p_{\bL^{\varrho}}+ \|f^0\|^p_{\bH^{p,\alpha}}\right). 
\end{equation*}
\end{Prop}
\begin{Remark}
As in Lemma \ref{lem:L2_approx} and Remark \ref{rem:on_lipschitz_approx}, if $\Psi$ is bounded (Condition {\bf H4'}), the result holds for $\bsp = p > 1$. Note that $C$ may depend on $\alpha$.

\end{Remark}
The results of \cite[Theorem 2.2]{bouc:poss:zhou:15} hold. More precisely if we have two solutions $(y^i,z^i,m^i,k^i) \in \bD^{\bsp \sq}\times \bH^{\bsp}\times \bM^{\bsp} \times \bI^{\bsp}_+$ of \eqref{eq:supersolution_general_BSDE} with terminal condition $\xi^i$ and generator $f^i$, we define
$$\delta y = y^1 - y^2, \quad \delta z = z^1 - z^2, \quad \delta m = m^1 - m^2, \quad \delta k = k^1 - k^2,$$ 
$$ \delta f(t,\omega,y,z) = f^1(t,\omega,y,z) - f^2(t,\omega,y,z).$$
$f^2$ satisfies Conditions {\bf (H)}.  
Then for any $\alpha \geq 0$ and $p$ satisfying \eqref{eq:cond_int}, namely $1 < p \leq \frac{\varrho \bsp}{\varrho + \bsp}$, there exists a constant $C$ such that 
\begin{equation*} 
 \|\delta z\|^p_{\bH^{p,\alpha}}+ \|\delta (m-k)\|^p_{\bM^{p,\alpha}} \leq C \left( \|\delta \xi\|^p_{\bL^p} + \| \delta y \|^p_{\bD^p} + \| \delta y \|^{\frac{p}{2}\wedge(p-1)}_{\bD^{p}} + \|\delta f(y^1,z^1)\|^p_{\bH^{p,\alpha}}\right).
\end{equation*}
Here the constant $C$ depends on $L_1$, $L_2$, $p$, $\alpha$ and also on $\|\Psi\|_{\bL^{\varrho}}$, $\| y^i \|_{\bD^{p}}$, $\| y^i \|_{\bD^{\bsp \sq}}$, $\| \xi^i \|_{\bL^{p}}$, $\| f^i(0,0) \|_{\bH^{p,\alpha}}$ for $i=1,2$. To prove this inequality we argue as in the proof of \cite[Theorem 2.2]{bouc:poss:zhou:15}. 

\subsubsection*{Application to reflected monotone BSDE}

Now we study the reflected BSDE:
\begin{equation}  \label{eq:reflected_general_BSDE}
\widetilde y_t = \xi +  \int_t^T f(u,\widetilde y_u,\widetilde z_u) du - \int_t^T \widetilde z_u dW_u - \int_t^T d\widetilde m_u + \int_t^T d\widetilde k_u
\end{equation}
with $\widetilde y_t \geq S_t$ and $\int_0^T (\widetilde y_{t-} - S_{t-}) d\widetilde k_t = 0$, $\bP$-a.s. ({\it Skorokhod condition}).  $S$ is a c\`adl\`ag process such that $S^+ = S \vee 0$ belongs to $\bD^{p}$. 

Using again the linearization procedure \eqref{eq:linear_procedure} and the new probability measure $\bQ$ defined by \eqref{eq:density_chgt_measure}, if $(\widetilde y, \widetilde z,\widetilde m,\widetilde k)$ is a solution, then
$$\Lambda_t \widetilde y_t = \Lambda_T \xi +  \int_t^T \Lambda_s f^0_s ds + \int_t^T \zeta_s \Lambda_s \widetilde z_s ds - \int_t^T \Lambda_s \widetilde z_s dW_s - \int_t^T \Lambda_s d\widetilde m_s + \int_t^T \Lambda_s d\widetilde k_s$$
with 
$\Lambda_t \widetilde y_t \geq \Lambda_t S_t$ and $\int_0^T (\Lambda_{u-} \widetilde y_{u-} - \Lambda_{u-} S_{u-}) d\widetilde k_u = 0$ a.s. Again the key point is that for $0\leq s \leq t \leq T$, $0<\Lambda_s/\Lambda_t \leq \exp(L_1(t-s))$. As in \cite[Proposition 3.1]{bouc:poss:zhou:15} the following representation holds: 
\begin{equation} \label{eq:linear_snell_envelop}
\Lambda_t \widetilde y_t = \underset{\tau \in \cT_{t,T}}{\esssup} \ \bE^{\bQ} \left[\int_t^\tau \Lambda_s f^0_s ds +\Lambda_T \xi \1_{\tau=T} + \Lambda_\tau S_\tau \1_{\tau < T}\bigg| \cF_t \right].
\end{equation}
Now we denote by $(y,z,m)$ the unique solution of the BSDE \eqref{eq:general_BSDE} (or \eqref{eq:supersolution_general_BSDE} with $k=0$).
\begin{equation*}  
y_t = \xi +  \int_t^T f(u,y_u,z_u) du  -\int_t^T z_u dW_u - \int_t^T dm_u.
\end{equation*}
From the comparison principle, a.s. for any $0\leq t\leq T$, $\widetilde y_t \geq y_t$. Let us begin with two technical lemmas corresponding to \cite[Propositions 3.2 and 3.3]{bouc:poss:zhou:15}.
\begin{Lemma} \label{lem:a_priori_estimate}
For $p> 1$, if $(\widetilde y,\widetilde z,\widetilde m,\widetilde k)$ is a solution of \eqref{eq:reflected_general_BSDE}, then
\begin{equation}  \label{eq:a_priori_estim_Y_RBSDE}
\| e^{\alpha .} \widetilde y \|^p_{\bD^{p}} \leq C_{\alpha,L_1,L_2,T} \left[ \| \xi \|^p_{\bL^p} + \|S^+\|^p_{\bD^{p}} +  \bE \left(  \int_0^T  |f^0_s| ds \right)^p \right] + \widehat C_{\alpha,L_1,T}\| e^{\alpha .} y \|^p_{\bD^{p}}.
\end{equation}
Moreover if we have two solutions $(\widetilde y^i,\widetilde z^i,\widetilde m^i,\widetilde k^i)$ of the reflected BSDE \eqref{eq:reflected_general_BSDE} with terminal condition $\xi^i$, generator $f^i$ and barrier $S^i$, then
\begin{equation} \label{eq:stability_RBSDE_Y}
\| e^{\alpha .} \delta \widetilde y \|^p_{\bD^{p}} \leq \bar C_{\alpha,L_1,L_2,T} \left[ \|\delta  \xi \|^p_{\bL^p} + \|\delta S\|^p_{\bD^{p}} +  \bE \left(  \int_0^T  |\delta f(s,\widetilde y^1_s,\widetilde z^1_s)| ds \right)^p \right].
\end{equation}
\end{Lemma}
\begin{proof}
Fix $p> 1$. Using the representation \eqref{eq:linear_snell_envelop}, for any $\alpha >0$, we obtain
\begin{eqnarray*}
\sup_{t\in [0,T]} \left( e^{\alpha t} |\widetilde y_t| \right) & \leq & e^{L_1 T} \sup_{t\in [0,T]} \left[ e^{\alpha t} \bE^\bQ\left( \int_0^T  |f^0_s| ds + \sup_{u\in [t,T]} S^+_u + |\xi| \bigg| \cF_t \right) \right]  + \sup_{t\in [0,T]} \left( e^{\alpha t} |y_t| \right) \\
& = &e^{L_1 T} \sup_{t\in [0,T]} \left[ e^{\alpha t} \bE \left( \cE \left( \int_t^T \zeta_s dW_s \right) \left(  \int_0^T  |f^0_s| ds + \sup_{u\in [t,T]} S^+_u + |\xi| \right) \bigg| \cF_t \right) \right]  \\
& & + \sup_{t\in [0,T]} \left( e^{\alpha t} |y_t| \right)\\
& \leq & C_{p,L_2,T}  e^{(\alpha +L_1)T }\sup_{t\in [0,T]} \left[  \bE \left( \left(  \int_0^T  |f^0_s| ds \right)^p + \sup_{u\in [t,T]} (S^+_u)^p + |\xi|^p \bigg| \cF_t \right) \right]^{1/p}  \\
& & + \sup_{t\in [0,T]} \left( e^{\alpha t} |y_t| \right)
\end{eqnarray*}
Using Doob's inequality we deduce \eqref{eq:a_priori_estim_Y_RBSDE}. 

The second point in \cite[Proposition 3.2]{bouc:poss:zhou:15} is the stability of solutions for reflected BSDE. We assume that we have two solutions $(\widetilde y^i,\widetilde z^i,\widetilde m^i,\widetilde k^i)$ of the reflected BSDE \eqref{eq:reflected_general_BSDE} with terminal condition $\xi^i$, generator $f^i$ and barrier $S^i$. The functions $f^i$ satisfy Assumptions {\bf (H)} and again we can assume that the monotonicity constant $L_1$ is non-positive. Then \eqref{eq:stability_RBSDE_Y} can be obtained with the same proof. In fact we only need that $f^2$ satisfies {\bf H3} with $L_1 \leq 0$. No particular condition on $L_1$ of the generator $f^1$ is used here. Hence \cite[Proposition 3.2]{bouc:poss:zhou:15} again holds under our setting. 
\end{proof}

\begin{Lemma} \label{lem:a_priori_estimate_2}
If we have two solutions $(\widetilde y^i,\widetilde z^i,\widetilde m^i,\widetilde k^i)$ of the reflected BSDE \eqref{eq:reflected_general_BSDE} with terminal condition $\xi^i$, generator $f^i$ and barrier $S^i$, then
\begin{eqnarray} \label{eq:stability_RBSDE_N}
\|  \delta \widetilde z \|^p_{\bH^{p,\alpha}} + \|  \delta (\widetilde m-\widetilde k) \|^p_{\bM^{p,\alpha}} \leq \eps \|  \delta f(\widetilde y^1,\widetilde z^1) \|^p_{\bH^{p,\alpha}} +C^\alpha \left(\|  \delta \xi \|^p_{\bL^{p}}  + \| \delta S \|^p_{\bD^{p}}\right)
\end{eqnarray}
where the constant $C^\alpha$ depends on $L_2$, $p$, $\alpha$, $\eps$ and $\|\Psi\|^p_{\bL^{\widehat p}}$, $\| \widetilde y^i \|^p_{\bD^{p}}$, $\| \widetilde y^i \|^{p\sq}_{\bD^{p\sq}}$, $\| \xi^i \|^p_{\bL^{p}}$, and $\| f^i(0,0) \|^p_{\bH^{p,\alpha}}$ for $i=1,2$ and $\widehat p$ is defined by \eqref{eq:def_widehat_p}.
\end{Lemma}
\begin{proof}
The arguments are the same as \cite[Proposition 3.3]{bouc:poss:zhou:15}. Indeed we can use Estimates \eqref{eq:estim_2_9} and \eqref{eq:estim_2_10} and the Skorokhod condition:
$$ \bE \left[  \int_0^T e^{\alpha s} \phi_p(\delta \widetilde y_{s-}) d(\delta \widetilde k_s)\right] \leq \bE \left[  \int_0^T e^{\alpha s} \phi_p(\delta S_{s-}) d(\delta \widetilde k_s)\right] ,$$
the function $\phi_p(x) = |x|^{p-1}\mbox{sgn}(x)$ being non-decreasing. We conclude using H\"older's inequality. 
\end{proof}

As before the right-hand side of \eqref{eq:stability_RBSDE_N} will be finite if we have the same condition as in Lemma \ref{lem:L2_approx}.

\begin{Prop}[Theorem 3.1 of \cite{bouc:poss:zhou:15}] \label{prop:ref_BSDE_monotone_driver}
Assume that $\xi \in \bL^{\bsp \sq}$, $S^+ \in \bD^{\bsp \sq}$ and $f^0 \in \bH^{\bsp \sq}$ for some $\bsp > 1$ with $\bsp > \varrho/(\varrho-1)$. There exists a unique solution $(Y,Z,M,K)$ to the reflected BSDE \eqref{eq:reflected_general_BSDE} in $\bD^{\bsp \sq}\times \bH^{p}\times \bM^{p} \times \bI^{p}$ for any $p$ such that
$$1 < p \leq \frac{\varrho \bsp}{\varrho + \bsp} < \bsp.$$
\end{Prop}
\begin{proof}
Uniqueness is a direct consequence of Estimates \eqref{eq:stability_RBSDE_Y} and \eqref{eq:stability_RBSDE_N}. W.l.o.g. we can assume that $L_1=0$. For the existence of a solution, we proceed in several steps. 
\begin{itemize}
\item {\bf Step 1.} Assume that {\bf (H')} holds and that $\xi$, $f^0$ and $S^+$ are bounded: there exists a constant $L_\infty$ such that a.s. 
\end{itemize}
\begin{equation} \label{eq:boundedness_condition}
|\xi| +\sup_{t \in [0,T]} |f^0_t| + \sup_{t \in [0,T]}  S^+_t  \leq L_\infty.
\end{equation}
Then the estimate of Proposition \ref{prop:a_priori_estim_N} holds for any $p> 1$. We denote
$$\bH^\infty = \bigcap_{\mathfrak{e} \geq \bsp \sq} \bH^{\mathfrak{e}}.$$
We proceed in two substeps. 

\begin{itemize}
\item[$\star$] {\bf Substep i.} Let us take $V \in \bH^{\infty}$ and we denote by $g(t,y)$ the function $f(t,y,V_t)$.
\end{itemize}
The generator $g$ satisfies the same condition {\bf (H')} as $f$, with $g^0 \in \bH^{\infty}$.
From \cite[Theorem 3.1]{bouc:poss:zhou:15}, there exists a unique solution $(\widetilde y^n,\widetilde z^n,\widetilde m^n,\widetilde k^n) \in \bS^{p}\times \bH^{p}\times \bM^{p} \times \bI^{p}_+$ to the reflected BSDE \eqref{eq:reflected_general_BSDE} where $g$ is replaced by $g_{n}$:
\begin{equation*}  
\widetilde y^{n}_t = \xi +  \int_t^T g_{n}(u,\widetilde y^n_u) du - \int_t^T \widetilde z^{n}_u dW_u - \int_t^T d\widetilde m^{n}_u + \int_t^T d\widetilde k^{n}_u
\end{equation*}
with $\widetilde y^{n}_t \geq S_t$ and $\int_0^T (\widetilde y^{n}_{t-} - S_{t-}) d\widetilde k^{n}_t = 0$, $\bP$-a.s. Here $g_n$ is defined as in \cite{bria:carm:00} (see also Equation \eqref{eq:lip_approx_gene_2}), that is by the convolution product
\begin{equation} \label{eq:lip_approx_gene}
g_n (t,.) = \varsigma_n *  (\Theta_{n+1} g(t,.)) ,\ t \in [0,T].
\end{equation}
where 
\begin{itemize}
\item $\varsigma : \bR \mapsto \bR_+$ is a non-negative function with the unit ball for support and such that $\int \varsigma(u)du = 1$ and we define for each integer $n>1$, $\varsigma_n(u) =n\varsigma(nu)$. 
\item For each integer $n$, $\Theta_n$ is a $C^\infty$ function from $\bR$ to $\bR_+$ such that $0\leq \Theta_n \leq 1$, $\Theta_n(u)=1$ for $|u| \leq n$ and $\Theta_n(u)=0$ as soon as $|u|>n+1$. 
\end{itemize}
This function $g_n$ is globally Lipschitz w.r.t. $y$ uniformly in $t$ and $\omega$ with $|g_n(t,0)|\leq  |f^0_t| + 2L_\sq$. Moreover for any $(t,\omega,y)$
\begin{equation} \label{eq:mono_prop_g_n}
yg_n(t,y) \leq (|f^0_t| + 2L_\sq)|y|,
\end{equation}
and for any $y$ and $y'$ in the ball $\overline{B(0,n)}$ then 
$$(y-y')(g_n(t,y)-g_n(t,y')) \leq 0.$$
In other words $g_n$ is only locally monotone (only in a given ball with the radius depending on $n$). 

From \eqref{eq:mono_prop_g_n}, \eqref{eq:linear_snell_envelop} and \eqref{eq:a_priori_estim_Y_RBSDE}
\begin{equation*}  
\sup_n \sup_{t\in [0,T]} |\widetilde y^n_t | \leq C_{L_\sq,L_\infty,T} < +\infty.
\end{equation*}
Hence we consider only the case $n \geq C_{L_\sq,L_\infty,T}$. Let us take $n' \geq n\geq C_{L_\sq,L_\infty,T}$ and 
$$\delta \widetilde y = \widetilde y^{n'} - \widetilde y^n, \quad \delta \widetilde z = \widetilde z^{n'}- \widetilde z^n, \quad \delta \widetilde m = \widetilde m^{n'} - \widetilde m^n, \quad \delta \widetilde k = \widetilde k^{n'}- \widetilde k^n.$$
We apply It\^o's formula to $(\delta  \widetilde y)^2$: for $t \in [0,T]$
\begin{eqnarray*}
&& \left| \delta  \widetilde y_t \right|^2 +  \int_t^T |\delta  \widetilde z_u|^2 du + \int_t^T d [\delta ( \widetilde m- \widetilde k)]_u \\
&&\quad =  2 \int_t^T  \left( g_\ell(u, \widetilde y^\ell_u) - g_n(u, \widetilde y^n_u) \right) (\delta  \widetilde y_{u}) du  \\
&& \qquad - 2 \int_t^T (\delta  \widetilde y_{u-}) \delta  \widetilde z_u dW_u  - 2 \int_t^T(\delta  \widetilde y_{u-}) d(\delta ( \widetilde m- \widetilde k))_u. 
\end{eqnarray*}
Since $| \widetilde y^n_t|\leq n \leq n'$, we use the local monotonicity of $g_l$ and we obtain:
\begin{eqnarray*}
&& \left| \delta  \widetilde y_t \right|^2 +  \int_t^T |\delta  \widetilde z_u|^2 du + \int_t^T d [\delta ( \widetilde m- \widetilde k)]_u  \leq 2 \int_t^T \left[ g_\ell(u, \widetilde y^n_u) - g_n(u, \widetilde y^n_u) \right] (\delta  \widetilde y_{u-}) du \\
& & \qquad - 2 \int_t^T (\delta  \widetilde y_{u-}) \delta  \widetilde z_u dW_u - 2 \int_t^T (\delta  \widetilde y_{u-}) d(\delta ( \widetilde m- \widetilde k))_u. 
\end{eqnarray*}
The Skorokhod condition implies that 
$$\bE \left[ \int_t^T e^{\alpha u} (\delta  \widetilde y_{u-}) d(\delta  \widetilde k)_u \right] \leq \bE \left[ \int_t^T e^{\alpha u} (\delta S_{u-}) d(\delta  \widetilde k)_u \right] = 0.$$
Since the set $\{\delta  \widetilde y_u \neq \delta  \widetilde y_{u-}\}$ is countable, classical arguments (using BDG inequality) imply that there exists a constant $C$ such that 
\begin{eqnarray*} 
&& \bE \left[ \sup_{t\in [0,T]} \left| \delta  \widetilde y_t \right|^2 +  \int_0^T |\delta  \widetilde z_u|^2 du + \int_0^T d [\delta ( \widetilde m- \widetilde k)]_u \right] \\
&& \qquad \leq  C \bE \left[  \int_0^T \left| g_{n'}(u, \widetilde y^n_u) - g_n(u, \widetilde y^n_u) \right| |\delta  \widetilde y_{u}| du\right].
\end{eqnarray*}
But the last expectation is equal to 
\begin{eqnarray*}
&& \bE \left[ \int_0^T \1_{| \widetilde y^n_u|+| \widetilde y^{n'}_u|\leq 2C_{L_\sq,L_\infty,T}} \left| g_{n'}(u, \widetilde y^n_u) - g_n(u, \widetilde y^n_u) \right| |\delta \widetilde y_u| du\right] \\
&& \qquad \leq 2C_{L_f,L_\infty,T}\bE \left[ \sup_{|y| \leq C_{L_\sq,L_\infty,T}} \int_0^T \left| g_{n'}(u,y) - g_n(u,y) \right|  du\right]
\end{eqnarray*}
We argue as in Step 1 of the proof of Lemma \ref{lem:L2_approx} and we obtain that $( \widetilde y^n, \widetilde z^n, \widetilde \mu^n=\widetilde m^n- \widetilde k^n) \in \bD^{2} \times \bH^{2}\times \bM^{2}$ is a Cauchy sequence. But we also have:
\begin{equation*}  
\widetilde \mu^{n'}_t-\widetilde \mu^n_t =  \int_0^t (g_{n'}(u, \widetilde y^{n'}_u) -g_{n}(u, \widetilde y^n_u)) du - \int_0^t ( \widetilde z^{n'}_u- \widetilde z^{n}_u) dW_u + (\widetilde y^{n'}_t -  \widetilde y^n_t) -(\widetilde y^{n'}_0 -  \widetilde y^n_0) .
\end{equation*}
Hence $\widetilde \mu^n$ converges also in $\bD^{2}$. Arguing as Step (iii) in the proof of \cite[Theorem 3.1]{bouc:poss:zhou:15}, we deduce that the limit $(\widetilde y,\widetilde z,\widetilde m,\widetilde k)$ satisfies the reflected BSDE 
\begin{equation} \label{eq:RBSDE_fixed_V}  
\widetilde y_t = \xi +  \int_t^T f(u,\widetilde y_u,V_u) du - \int_t^T \widetilde z_u dW_u - \int_t^T d\widetilde m_u + \int_t^T d\widetilde k_u
\end{equation}
with $\widetilde y_t \geq S_t$ and $\int_0^T (\widetilde y_{t-} - S_{t-}) d\widetilde k_t = 0$, $\bP$-a.s. From Proposition \ref{prop:a_priori_estim_N} with $\varrho = +\infty$ and $\bsp =+\infty$, we obtain for any $p \geq 1$
\begin{equation*} 
 \|\widetilde z\|^p_{\bH^{p}}+ \|\widetilde m\|^p_{\bM^{p}} + \| \widetilde k \|^p_{\bI^{p}} < +\infty.
\end{equation*}
Hence for each $V \in \bH^{\infty}$, we have a unique solution  $(\widetilde y,\widetilde z,\widetilde m,\widetilde k) \in \bS^{\infty}\times \bH^{\infty}\times \bM^{\infty} \times \bI^{\infty}$ to \eqref{eq:RBSDE_fixed_V}.

\begin{itemize}
\item[$\star$] {\bf Substep ii.} General case under boundedness conditions.
\end{itemize}
We use a fixed point argument. Let $(\widetilde y^0,\widetilde z^0,\widetilde m^0,\widetilde k^0)=(0,0,0,0)$ and let $(\widetilde y^{n+1},\widetilde z^{n+1},\widetilde m^{n+1},\widetilde k^{n+1}) \in \bD^{\infty}\times \bH^{\infty}\times \bM^{\infty} \times \bI^{\infty}$ be the unique solution of the reflected BSDE:
\begin{equation*}  
\widetilde y^{n+1}_t = \xi +  \int_t^T f(u,\widetilde y^{n+1}_u,\widetilde z^n_u) du - \int_t^T \widetilde z^{n+1}_u dW_u - \int_t^T d\widetilde m^{n+1}_u + \int_t^T d\widetilde k^{n+1}_u
\end{equation*}
with $\widetilde y^{n+1}_t \geq S_t$ and $\int_0^T (\widetilde y^{n+1}_{t-} - S_{t-}) d\widetilde k^{n+1}_t = 0$, $\bP$-a.s. From our first substep, the sequence is well-defined. For any $i$, we denote 
$$\delta \widetilde y^{n,i} = \widetilde y^{n+i} - \widetilde y^n, \delta \widetilde z^{n,i} = \widetilde z^{n+i} - \widetilde z^n, \delta \widetilde m^{n,i} = \widetilde m^{n+i} - \widetilde m^n,\delta \widetilde k^{n,i} = \widetilde k^{n+i} - \widetilde k^n$$
and $\widetilde \ell^n := \widetilde m^n - \widetilde k^n$. 
We apply It\^o's formula to $e^{\alpha t} |\delta \widetilde y_t^{n,i}|^2$:
\begin{eqnarray*}
&&e^{\alpha t} \left| \delta \widetilde y^{n,i}_t \right|^2 +  \int_t^T e^{\alpha u} |\delta \widetilde z^{n,i}_u|^2 du +  \int_t^T e^{\alpha u} d [\delta \widetilde \ell^{n,i}]_u \\
&&\quad \leq  2 \int_t^T e^{\alpha u}  \delta  \widetilde y^{n,i}_u \left[ f(u,\widetilde y^{n+i}_u,\widetilde z^{n-1+i}_u) - f(u,\widetilde y^{n}_u,\widetilde z^{n-1}_u) \right] du \\
&&\qquad - \alpha\int_t^T e^{\alpha u} \left| \delta \widetilde y^{n,i}_u \right|^2 du - 2 \int_t^T e^{\alpha u}\delta \widetilde y^{n,i}_u \delta  \widetilde z^{n,i}_u dW_u  - 2 \int_t^T e^{\alpha u} \delta \widetilde y^{n,i}_{u-} d(\delta \widetilde \ell^{n,i})_u \\
&&\quad \leq (2 L_1 - \alpha) \int_t^T e^{\alpha u}\left| \delta \widetilde y^{n,i}_u \right|^{2} du +2L_2 \int_t^T e^{\alpha u}\left| \delta \widetilde y^{n,i}_u \right|   |\delta \widetilde z^{n-1,i}_u| du \\
&&\qquad 
 - 2 \int_t^T e^{\alpha u}  \delta \widetilde y^{n,i}_u \delta  \widetilde z^{n,i}_u dW_u   
- 2 \int_t^T e^{\alpha u} \delta \widetilde y^{n,i}_{u-} d(\delta \widetilde m^{n,i})_u
\end{eqnarray*}
since 
$$ \int_t^T e^{\alpha u} \delta \widetilde y^{n,i}_{u-} d(\delta \widetilde k^{n,i})_u \leq \int_t^T e^{\alpha u} \delta S_{u-} d(\delta \widetilde k^{n,i})_u =0.$$
Young's inequality gives for any $\eps > 0$
\begin{eqnarray}\nonumber
&&e^{\alpha t} \left| \delta \widetilde y^{n,i}_t \right|^2 + \int_t^T e^{\alpha u} |\delta \widetilde z^{n,i}_u|^2 du +\int_t^T e^{\alpha u}  d [\delta \widetilde \ell^{n,i}]_u \\ \nonumber
&&\quad \leq \left(2 L_1 + \frac{L_2^2}{\eps} - \alpha \right) \int_t^T e^{\alpha u}\left| \delta \widetilde y^{n,i}_u \right|^{2} du +\eps \int_t^T e^{\alpha u} |\delta \widetilde z^{n-1,i}_u|^2 du \\ \label{eq:tech_estim_L_p_RBSDE}
&&\qquad  - 2 \int_t^T e^{\alpha u} \delta \widetilde y^{n,i}_{u-} \left( d(\delta \widetilde m^{n,i})_u+ \delta  \widetilde z^{n,i}_u dW_u  \right) .
\end{eqnarray}
Take $\eps =1/2$, $\alpha = 1 + 2 L_1 + 2L_2^2$ and the expectation and deduce that 
\begin{eqnarray*}
&&\bE  \int_0^T e^{\alpha u}\left| \delta \widetilde y^{n,i}_u \right|^{2} du  + \bE \int_0^T e^{\alpha u} |\delta \widetilde z^{n,i}_u|^2 du +  \bE\int_t^T e^{\alpha u}  d [\delta \widetilde \ell^{n,i}]_u  \\
&&\qquad \leq \frac{1}{2} \bE \int_0^T e^{\alpha u} |\delta \widetilde z^{n-1,i}_u|^2 du.
\end{eqnarray*}
Thus $(\widetilde y^n,\widetilde z^n,\widetilde \ell^n)$ is a Cauchy sequence in $\bH^2 \times \bH^2 \times \bM^2$, and using BDG inequality we will have convergence in $\bD^2 \times \bH^2 \times \bM^2$. Then the conclusion follows by the same arguments as in \cite[Theorem 3.1]{bouc:poss:zhou:15}. Then since \eqref{eq:boundedness_condition} holds, from Lemmas \ref{lem:a_priori_estimate} and \ref{lem:a_priori_estimate_2}, we deduce that the limit is also in $\bD^{\infty}\times \bH^{\infty}\times \bM^{\infty} \times \bI^{\infty}$.

\begin{itemize}
\item {\bf Step 2.} Assume that for some $\bsp > \varrho/(\varrho-1)$, $\xi$, $f^0$ and $S^+$ are in $\bL^{\bsp \sq}\times \bH^{\bsp \sq } \times \bD^{\bsp \sq}$. We fix $1<p \leq \varrho\bsp / (\varrho+\bsp)$.
\end{itemize}
For any $n\in \bN^*$
$$\xi_n  = \frac{n\xi}{n\vee |\xi|},\quad S^n_t = \frac{n S_t}{n\vee S_t}$$
and  
$$f_n(t,y,z) = (f(t,y,z)-f^0_t) + \frac{n f^0_t }{n\vee |f^0_t|}.$$
Then 
$$|f_n(t,0,0)| =  \left| \frac{n f^0_t }{n\vee |f^0_t|} \right| \leq n.$$
We apply the result of Step 1: there exists a unique solution $(\widetilde y^n,\widetilde z^n,\widetilde m^n,\widetilde k^n) \in \bD^{\bsp \sq}\times \bH^{p}\times \bM^{p} \times \bI^{p}$ to the reflected BSDE:
\begin{equation*}  
\widetilde y^{n}_t = \xi_n +  \int_t^T f_n(u,\widetilde y^n_u,\widetilde z^n_u) du - \int_t^T \widetilde z^{n}_u dW_u - \int_t^T d\widetilde m^{n}_u + \int_t^T d\widetilde k^{n}_u
\end{equation*}
with $\widetilde y^{n}_t \geq S^n_t$ and $\int_0^T (\widetilde y^{n}_{t-} - S^n_{t-}) d\widetilde k^{n}_t = 0$, $\bP$-a.s. From Lemma \ref{lem:a_priori_estimate} and Estimate \eqref{eq:a_priori_estim_Y_RBSDE}, 
$$\sup_{n\in\bN} \|e^{\alpha .} \widetilde y^n \|_{\bD^{\bsp \sq}}  < +\infty.$$
Thus from Proposition \ref{prop:a_priori_estim_N}, the $\bH^p\times \bM^p \times \bI^p$-norm of the sequence $(\widetilde z^n,\widetilde m^n,\widetilde k^n)$ is bounded uniformly w.r.t. $n$.
We denote $\widetilde \ell :=\widetilde m-\widetilde k$, we take $n\leq n'$, we define again:
$$\delta \widetilde y = \widetilde y^{n'} - \widetilde y^n, \quad \delta \widetilde z = \widetilde z^{n'}- \widetilde z^n, \quad \delta \widetilde m = \widetilde m^{n'} - \widetilde m^n, \quad \delta \widetilde k = \widetilde k^{n'}- \widetilde k^n, \quad \delta \widetilde \ell = \widetilde \ell^{n'}- \widetilde \ell^n,$$
and we apply It\^o's formula for $t \in [0,T]$ 
\begin{eqnarray*}
&&e^{\alpha t} |\delta \widetilde y_t|^p + c(p) \int_t^T  |\delta \widetilde y_{s-}|^{p-2}  \1_{\delta \widetilde y_{s-}\neq 0} (\delta \widetilde z_s)^2 ds +c(p) \int_t^T  |\delta \widetilde y_{s-}|^{p-2}  \1_{\delta \widetilde y_{s^-}\neq 0} d[ \delta \widetilde \ell ]^c_s \\
&& +  c(p) \sum_{t< s \leq T} e^{\alpha s} |\Delta (\delta\widetilde \ell)_s|^2 \left[  |\delta \widetilde y_{s-}|^2 \vee |\delta \widetilde y_{s-}+ \Delta (\delta \widetilde \ell)_s|^{2} \right]^{\frac{p}{2}-1} \1_{|\delta \widetilde y_{s-}| \vee |\delta \widetilde y_{s-}+ \Delta (\delta\widetilde \ell)_s| \neq  0} \\
&&\quad \leq e^{\alpha T}|\delta \xi|^p +  p \int_t^T e^{\alpha u} \phi_p(\delta \widetilde y_{u-}) \left[ f_{n'}(u,\widetilde y^{n'}_u,\widetilde z^{n'}_u) - f_n(u,\widetilde y^n_u,\widetilde z^n_u) \right]du \\
&&\qquad - \alpha\int_t^T e^{\alpha u} \left| \delta \widetilde y_u \right|^p du - p \int_t^T e^{\alpha u} \phi_p(\delta \widetilde y_{u-}) d(\delta \widetilde \ell)_u . 
\end{eqnarray*}
By the assumptions on $f$, thus on $f_{n'}$ and Young's inequality
\begin{eqnarray*}
&& p \int_t^T e^{\alpha u} \phi_p(\delta \widetilde y_{u-})  \left[ f_{n'}(u,\widetilde y^{n'}_u,\widetilde z^{n'}_u) - f_{n'}(u,\widetilde y^n_u,\widetilde z^n_u) \right]du- \alpha\int_t^T e^{\alpha u} \left| \delta \widetilde y_u \right|^p du \\
&& \quad \leq \left( pL_1 + \frac{p}{(p-1)\wedge 1}L_2^2 -\alpha  \right) \int_t^T e^{\alpha u} \left| \delta \widetilde y_u \right|^p du \\
&& \qquad + \frac{c(p)}{2} \int_t^T e^{\alpha u}|\delta \widetilde z_u|^2 |\delta \widetilde y_{u-}|^{p-2}  \1_{\delta \widetilde y_{u-}\neq 0}du .
\end{eqnarray*}
Moreover from the Skorohod condition:
$$\int_t^T e^{\alpha u}\phi_p(\delta \widetilde y_{u-}) d(\delta \widetilde k)_u \leq \int_t^T e^{\alpha u}\phi_p(\delta S_{u-}) d(\delta \widetilde k)_u.$$
Hence for $\alpha \geq pL_1 + \frac{p}{(p-1)\wedge 1}L_2^2$
\begin{eqnarray*}
&&e^{\alpha t} |\delta \widetilde y_t|^p +\frac{c(p)}{2} \int_t^T  |\delta \widetilde y_{s-}|^{p-2}  \1_{\delta \widetilde y_{s-}\neq 0} (\delta \widetilde z_s)^2 ds  +c(p) \int_t^T  |\delta \widetilde y_{s-}|^{p-2}  \1_{\delta \widetilde y_{s-}\neq 0} d[ \delta \widetilde \ell ]^c_s \\
&& +  c(p) \sum_{t< s \leq T} e^{\alpha s} |\Delta (\delta\widetilde \ell)_s|^2 \left[  |\delta \widetilde y_{s-}|^2 \vee |\delta \widetilde y_{s-}+ \Delta (\delta \widetilde \ell)_s|^{2} \right]^{\frac{p}{2}-1} \1_{|\delta \widetilde y_{s-}| \vee |\delta \widetilde y_{s-}+ \Delta (\delta\widetilde \ell)_s| \neq  0} \\
&&\quad \leq e^{\alpha T}|\delta \xi|^p +  p \int_t^T e^{\alpha u} \phi_p(\delta \widetilde y_{u-}) \left[ f_{n'}(u,\widetilde y^n_u,\widetilde z^n_u) - f_n(u,\widetilde y^n_u,\widetilde z^n_u) \right]du \\
&&\qquad - p \int_t^T e^{\alpha u} \phi_p(\delta \widetilde y_{u-}) (\delta \widetilde z_u dW_u + d\delta \widetilde m_u) + p \int_t^T e^{\alpha u}\phi_p(\delta S_{u-}) d(\delta \widetilde k)_u. 
\end{eqnarray*}
Since the $\bI^{p,\alpha}$-norm of $(\widetilde k^n)$ is bounded uniformly w.r.t. $n$, we deduce that there exists a constant $C$ such that 
\begin{eqnarray*}
&& \|e^{\alpha .} \delta \widetilde y^n \|_{\bD^{p}} + \|e^{\alpha .}\delta \widetilde z^n \|_{\bH^{p}} + \|e^{\alpha .} \delta (\widetilde m-\widetilde k)^n \|_{\bM^{p}} \\
&&\qquad  \leq C \left( \|\delta \xi\|_{\bL^p} + \|e^{\alpha  .} \delta S \|_{\bD^p} +\|e^{\alpha .} \delta f^0 \|_{\bH^{p}}   \right).
\end{eqnarray*}
Note that the constant $C$ depends in particular on $\|\Psi\|_{\bL^{\widehat p}}$ which will be finite since $\widehat p \leq \varrho$. Thus we have a Cauchy sequence in $\bD^{p}\times \bH^{p}\times \bM^{p}$, and in $\bD^{\bsp \sq}\times \bH^{p}\times \bM^{p}$, which converges to $(\widetilde y,\widetilde z,\widetilde \nu)$. We argue as in the Step 1,(iii) of the proof of \cite[Theorem 3.1]{bouc:poss:zhou:15}, to obtain the desired result.

\end{proof}

Remark \ref{rem:on_lipschitz_approx} also holds for this last result. In particular if {\bf (H')} holds then $\bsp = p$.

\bibliography{biblio}

\end{document}